\documentclass[12pt]{amsart}

\usepackage{amsmath}
\usepackage{amsfonts}
\usepackage{amssymb}
\usepackage[all]{xy}           
\usepackage{xcolor}
\usepackage{bbding}
\usepackage{txfonts}
\usepackage{amscd}

\usepackage[shortlabels]{enumitem}
\usepackage{ifpdf}
\ifpdf
  \usepackage[colorlinks,final,
  hyperindex]{hyperref}
\else
  \usepackage[colorlinks,final,
  hyperindex]{hyperref}
\fi
\usepackage{tikz}
\usepackage[active]{srcltx}

\makeatletter

\newtheorem{thm}{Theorem}[section]
\newtheorem{lem}[thm]{Lemma}
\newtheorem{cor}[thm]{Corollary}
\newtheorem{pro}[thm]{Proposition}
\newtheorem{ex}[thm]{Example}
\newtheorem{rmk}[thm]{Remark}
\newtheorem{defi}[thm]{Definition}

\numberwithin{equation}{section}

\newcommand{\g}{\mathfrak {g}}

\newcommand{\kl}{\mathfrak l}
\newcommand{\kr}{\mathfrak r}

\newcommand{\bz}{\mathbb Z}

\newcommand{\fl}{\mathbf{l}}
\newcommand{\fr}{\mathbf{r}}

\newcommand{\id}{\mathrm{id}}

\newcommand{\md}{\mathrm{D}}

\newcommand{\End}{\mathrm{End}}
\newcommand{\ad}{\mathrm{ad}}

\newcommand{\DYBE}{\mathrm{DYBE}}
\newcommand{\CYBE}{\mathrm{CYBE}}
\newcommand{\AYBE}{\mathrm{AYBE}}
\newcommand{\PLYBE}{\mathrm{PLYBE}}

\begin{document}

\title[Affinization of dendriform $\md$-bialgebras, Lie bialgebras and solutions of YBE]
{Affinization of dendriform $\md$-bialgebras, Lie bialgebras and solutions of
classical Yang-Baxter equation}

\author{Bo Hou}
\address{School of Mathematics and Statistics, Henan University, Kaifeng 475004,
China}
\email{houbo@henu.edu.cn, bohou1981@163.com}


\vspace{-5mm}


\begin{abstract}
In this paper, we mainly discuss how to use dendriform $\md$-bialgebras to construct
Lie bialgebras and the relationship between the solutions of their corresponding
Yang-Baxter equations. We provide two methods for obtaining Lie algebras from
dendriform algebras using the tensor product with perm algebras, one by means
of associative algebras and the other by means of pre-Lie algebras. We elevate both
approaches to the level of bialgebras and prove that the Lie bialgebraa obtained using
these two approaches are the same. There is a correspondence between symmetric
solutions of the dendriform Yang-Baxter equation in dendriform algebras and
certain skew-symmetric solutions of the classical Yang-Baxter equation in the
Lie algebras induced from the dendriform algebras. The connections between
triangular bialgebra structures, $\mathcal{O}$-operators related to the solutions
of these Yang-Baxter equations are discussed in detail. During the discussion,
we also present a method for constructing infinite-dimensional antisymmetric
infinitesimal bialgebra by using the affineization of dendriform $\md$-bialgebras.
\end{abstract}

\makeatletter
\@namedef{subjclassname@2020}{\textup{2020} Mathematics Subject Classification}
\makeatother

\maketitle

\vspace{-10mm}
\tableofcontents 



\vspace{-10mm}

\section{Introduction}\label{sec:intr}

A bialgebra structure on a given algebra structure is obtained as a coalgebra
structure together which gives the same algebra structure on the dual space with
a set of compatibility conditions between the products and coproducts.
Two famous examples of bialgebras are Lie bialgebras \cite{Dri} and antisymmetric
infinitesimal (for short ASI) bialgebras \cite{Zhe,Bai}.  Lie bialgebras can be viewed
as a linearization of Poisson-Lie groups. It were first introduced by Drinfeld in
the context of the theory of Yang-Baxter equations and quantum groups. Later it has
found applications in many other areas of mathematics and mathematical physics such as
the theory of Hopf algebra deformations of universal enveloping algebras \cite{ES},
string topology and symplectic field theory \cite{CFL},
Goldman-Turaev theory of free loops in Riemann surfaces with boundaries \cite{Gol},
and so on. Infinitesimal bialgebras first appeared in the work of Joni and Rota to give
an algebraic framework for the calculus of divided differences \cite{JR}.
The antisymmetric version of infinitesimal bialgebras was introduced in \cite{Zhe}
by Zhelyabin by using the name associative $\md$-bialgebra as an associative analog
of Lie bialgebra. Later this structure was studied systematically by Bai under
the name antisymmetric infinitesimal bialgebra in \cite{Bai}.
The antisymmetric infinitesimal bialgebra can be characterized by the well-known
matched pair of associative algebras and a double construction of a Frobenius algebra.
In recent years, the bialgebra theories of various algebra structures have been
extensive developed, such as left-symmetric bialgebras (also called pre-Lie
bialgebras)\cite{Bai1}, Jordan bialgebras \cite{Zhe}, Novikov bialgebras \cite{HBG},
perm bialgebras \cite{Hou,LZB}, Jacobi-Jordan bialgebras \cite{BCHM}, dendriform
$\md$-bialgebras \cite{Bai}, and so on. The relationship between these bialgebra
structures, as well as the relationship between these bialgebras and conformal
bialgebras, have been discussed in \cite{HBG,LZB,LLB,HBG1,HBG2,CH}.
We will discuss in detail the relationship between dendriform $\md$-bialgebras,
pre-Lie bialgebras, antisymmetric infinitesimal bialgebras and Lie bialgebras,
and then provide how to use dendriform $\md$-bialgebras, pre-Lie bialgebras,
antisymmetric infinitesimal bialgebras to construct Lie bialgebras.

\smallskip\noindent
1.1. {\bf Dendriform $\md$-bialgebras, pre-Lie bialgebras, antisymmetric
infinitesimal bialgebras and Lie bialgebras.} For dendriform algebras, pre-Lie algebras,
associative algebras and Lie algebras, there exists a well-known commutative diagram:
$$
\xymatrix@C=3cm@R=0.5cm{
\txt{$(D, \prec, \succ)$ \\ {\tiny a dendriform algebra}}
\ar[d]_-{\mbox{\tiny $d_{1}\diamond d_{2}=d_{1}\succ d_{2}-d_{2}\prec d_{1}$}}
\ar[r]^-{\mbox{\tiny $d_{1}\ast d_{2}=d_{1}\prec d_{2}+d_{1}\succ d_{2}$}}
&\txt{$(D, \ast)$\\ {\tiny an associative algebra}}
\ar[d]^-{\mbox{\tiny $[d_{1}, d_{2}]=d_{1}\ast d_{2}-d_{2}\ast d_{1}$}} \\
\txt{$(D, \diamond)$ \\ {\tiny a pre-Lie algebra}}
\ar[r]^-{\mbox{\tiny $[d_{1}, d_{2}]=d_{1}\diamond d_{2}-d_{2}\diamond d_{1}$}}
& \txt{$(D, [-,-])$ \\
{\tiny a Lie algebra}}}
$$
This diagram is also commutative for the categories of dendriform algebras,
pre-Lie algebras, associative algebras and Lie algebras \cite{LFCG}.
But unfortunately, a pre-Lie bialgebra generally does not have a Lie bialgebra structure,
and a dendriform $\md$-bialgebra also does not have an antisymmetric infinitesimal
bialgebra structure. Please refer to \cite{Bai} for details. So, it is difficult
to elevate this commutative diagram to the level of bialgebras.
Note that there is also a Lie algebra structure on the tensor product of a pre-Lie
algebra and a perm algebra, and there is also an associative algebra structure on
the tensor product of a dendriform algebra and a perm algebra. Thus, we get a new
commutative diagram:
$$
\xymatrix@C=3cm@R=0.5cm{
\txt{$(D, \prec, \succ)$ \\ {\tiny a dendriform algebra}}
\ar[d]_-{\mbox{\tiny $d_{1}\diamond d_{2}=d_{1}\succ d_{2}-d_{2}\prec d_{1}$}}
\ar[r]^-{\mbox{\tiny $-\otimes B$}}
&\txt{$(D\otimes B, \ast)$\\ {\tiny an associative algebra}}
\ar[d]^-{\mbox{\tiny $[x_{1}, x_{2}]=x_{1}\ast x_{2}-x_{2}\ast x_{1}$}} \\
\txt{$(D, \diamond)$ \\ {\tiny a pre-Lie algebra}} \ar[r]^-{\mbox{\tiny $-\otimes B$}}
& \txt{$(D\otimes B, [-,-])$ \\
{\tiny a Lie algebra}}}
$$
where $(B, \cdot)$ is a perm algebra. A quadratic perm algebra is a perm algebra
equipped with an antisymmetric invariant bilinear form. Note that a quadratic perm
algebra $(B, \cdot, \omega)$ naturally induces a perm coalgebra structure on $B$.
We can give an ASI bialgebra structure on the tensor product of a dendriform
$\md$-bialgebra and a quadratic perm algebra (In fact, this is a direct result of
the affinization of dendriform $\md$-bialgebras.)
Thus, we can elevate the commutative diagram above to the level of bialgebras,
leading to the following conclusion:

\smallskip\noindent
{\bf Theorem I } ( Theorem \ref{thm:bialgebras} )
{\it Let $(D, \prec, \succ, \theta_{\prec}, \theta_{\succ})$ be a dendriform $\md$-bialgebra
and $(B, \cdot, \omega)$ be a quadratic perm algebra. Then we have the following
commutative diagram:
$$
\xymatrix@C=3cm@R=0.5cm{
\txt{$(D, \prec, \succ, \theta_{\prec}, \theta_{\succ})$ \\ {\tiny a dendriform bialgebra}}
\ar[d]_-{\mbox{\tiny $d_{1}\diamond d_{2}=d_{1}\succ d_{2}-d_{2}\prec d_{1}$}}
^-{\mbox{\tiny\rm\textcolor{red} {Pro.} ~\ref{pro:ASI-Liebia}}}
\ar[r]^-{\mbox{\tiny $-\otimes(B, \cdot, \omega)$}}_-{\mbox{\tiny\rm\textcolor{red}
{Cor.} ~\ref{cor:indassbia}}}
&\txt{$(D\otimes B, \ast, \Delta)$\\ {\tiny an ASI bialgebra}}
\ar[d]^-{\mbox{\tiny $[x_{1}, x_{2}]=x_{1}\ast x_{2}-x_{2}\ast x_{1}$}}
_-{\mbox{\tiny\rm\textcolor{red} {Pro.} ~\ref{pro:ASI-Liebia}}} \\
\txt{$(D, \diamond, \vartheta)$ \\ {\tiny a pre-Lie bialgebra}}
\ar[r]^-{\mbox{\tiny $-\otimes(B, \cdot, \omega)$}}_-{\mbox{\tiny\rm\textcolor{red}
{Thm.} ~\ref{thm:pre-lie}}}
& \txt{$(D\otimes B, [-,-], \delta)$ \\
{\tiny a Lie bialgebra}}}
$$
}

\noindent
1.2. {\bf Dendriform $\md$-bialgebras affinization construction of
antisymmetric infinitesimal bialgebras.} An important construction of
infinite-dimensional Lie algebras is a process of affinization.
Balinsky and Novikov showed that the affinization of a Novikov algebra naturally
defines a Lie algebra, a property that in fact characterizes the Novikov algebra \cite{BN}.
Recently, Hong, Bai and Guo have introduced the notion of quadratic $\bz$-graded right
Novikov algebra and shown that the tensor product of a finite-dimensional Novikov
bialgebra and a quadratic $\bz$-graded right Novikov algebra can be naturally endowed
with a completed Lie bialgebra. The converse of this result also holds when
the quadratic $\bz$-graded right Novikov algebra is some special right Novikov algebra,
giving the desired characterization of the Novikov bialgebra that its affinization
is a Lie bialgebra \cite{HBG}. The affinization of perm bialgebras by quadratic
$\bz$-graded pre-Lie algebras is another example \cite{LZB}. All these examples fit
into the interpretation of the Ginzburg-Kapranov operadic Koszul duality \cite{GK}.
Here, we give an affinization of dendriform $\md$-bialgebras which
is independent of operadic Koszul duality.

Roughly speaking, the affinization of a given algebra structure is to define
algebra structure on the vector space of Laurent polynomials and obtain another algebra
structure on the tensor product, which could resolve the given algebra structure in turn.
In this paper, we consider the affinization of dendriform $\md$-bialgebras.
We first show that the tensor product of a dendriform algebra and a special
perm algebra on the vector space of Laurent polynomials (which is constructed in
\cite{LZB}) being an associative algebra characterizes the dendriform algebra.
That is, we obtain associative algebras via affinization of dendriform algebras
by perm algebras, which gives an affinization characterization of dendriform algebras.
Second, we consider the dual version of the dendriform algebra affinization, for
dendriform coalgebras. We introduce the notion of completed coassociative coalgebras
and prove that there is a natural completed coassociative coalgebra structure on the
tensor product of a dendriform coalgebra and the completed perm coalgebra of Laurent
polynomials, a property that in fact characterizes the dendriform coalgebra.
Finally, we use the quadratic $\bz$-graded perm algebra lift the affinization
characterizations of dendriform algebras and dendriform coalgebras to the level
of bialgebras.

\smallskip\noindent
{\bf Theorem II } ( Theorem \ref{thm:den-perm-ass} )
{\it Let $(D, \prec, \succ, \theta_{\prec}, \theta_{\succ})$ be a finite-dimensional
dendriform $\md$-bialgebra, $(B=\oplus_{i\in\bz}B_{i}, \cdot, \varpi)$ be a quadratic
$\bz$-graded perm algebra and $(D\otimes B, \ast)$ be the induced $\bz$-graded
associative algebra from $(D, \prec, \succ)$ by $(B, \cdot)$.
Define a linear map $\Delta: D\otimes B\rightarrow(D\otimes B)\otimes(D\otimes B)$ by
\begin{align*}
\Delta(d\otimes b)&=\theta_{\succ}(d)\bullet\nu_{\varpi}(b)
+\theta_{\prec}(d)\bullet\,\hat{\tau}(\nu_{\varpi}(b))\\
&:=\sum_{[d]}\sum_{i,j,\alpha}(d_{[1]}\otimes b_{1,i,\alpha})
\otimes(d_{[2]}\otimes b_{2,j,\alpha})+\sum_{(d)}\sum_{i,j,\alpha}(d_{(1)}\otimes
b_{2,j,\alpha})\otimes(d_{(2)}\otimes b_{1,i,\alpha}),
\end{align*}
for any $d\in D$ and $b\in B$, where $\theta_{\prec}(d)=\sum_{(d)}d_{(1)}
\otimes d_{(2)}$, $\theta_{\succ}(d)=\sum_{[d]}d_{[1]}\otimes d_{[2]}$ in the
Sweedler notation and $\nu_{\varpi}(b)=\sum_{i,j,\alpha}b_{1,i,\alpha}\otimes
b_{2,j,\alpha}$. Then $(D\otimes B, \ast, \Delta)$ is a completed ASI bialgebra.

Moreover, if $(B=\oplus_{i\in\bz}B_{i}, \cdot, \varpi)$ is the quadratic $\bz$-graded
perm algebra given in Example \ref{ex:qu-perm}, then $(D\otimes B, \ast, \Delta)$ is
a completed ASI bialgebra if and only if $(D, \prec, \succ, \theta_{\prec},
\theta_{\succ})$ is a dendriform $\md$-bialgebra.}

\smallskip
In particular, we also consider the finite-dimensional ASI bialgebra structure on
the tensor product of a dendriform $\md$-bialgebra and a quadratic perm algebra.

\smallskip\noindent
1.3. {\bf Solution of Yang-Baxter equation, triangular bialgebra structure and
$\mathcal{O}$-operator.}
The classical Yang-Baxter equation (for short $\CYBE$) arose from the study of
inverse scattering theory in the 1980s and was recognized as the semi-classical limit
of the quantum Yang-Baxter equation \cite{Yang,Bax}. The study of the $\CYBE$
in a Lie algebra has substantial ramifications and applications in the areas
of symplectic geometry, quantum groups, integrable systems, and quantum field theory
\cite{BD,Sto}. Semenov-Tian-Shansky discovered that, on quadratic Lie algebras,
the $\CYBE$ in tensor form can be transformed into a Rota-Baxter operator of weight
$0$ \cite{Sem}. Subsequently, Kupershmidt found that the $\CYBE$ in tensor form on
general Lie algebras can also be converted into an $\mathcal{O}$-operator associated
to the coadjoint representation \cite{Kup}. Nowadays, the $\mathcal{O}$-operator is
regarded as an operator form of a solution of classical Yang-Baxter equation.

In the Lie bialgebra theory, quasi-triangular Lie bialgebras especially triangular Lie
bialgebras play important roles in mathematical physics \cite{Sem,BGN,RS}.
Factorizable Lie bialgebras constitute a special subclass of quasi-triangular
Lie bialgebras that bridge classical $r$-matrices with certain factorization problems,
which is finded diverse applications in integrable systems \cite{Dri1,LS}.
The solutions of the $\CYBE$ in a Lie algebra $(\g, [-,-])$ is closely related to
these Lie bialgebra structures on $(\g, [-,-])$: a skew-symmetric solution of the
$\CYBE$ in $(\g, [-,-])$ gives rise to a triangular Lie bialgebra structure on
$(\g, [-,-])$; a solution of the $\CYBE$ in $(\g, [-,-])$ whose skew-symmetric part
is invariant gives rise to a quasi-triangular Lie bialgebra structure on $(\g, [-,-])$.
Recently, by introducing the classical Yang-Baxter equation in various algebra structures
and considering some special solutions of this equation, various quasi-triangular
bialgebra structures have been studied, such as quasi-triangular ASI bialgebras \cite{SW},
quasi-triangular pre-Lie bialgebras \cite{WBLS}, quasi-triangular Leibniz bialgebras
\cite{BLST}, quasi-triangular Novikov bialgebras \cite{CH}, and so on.

In this paper, we propose two approaches to construct a solution of
the $\CYBE$ in the induced Lie algebra $(D\otimes B, [-,-])$ (by a perm algebra
$(B, \cdot)$) starting from a solution of the dendriform Yang-Baxter equation in a
dendriform algebra $(D, \prec, \succ)$, and the solutions of the $\CYBE$ obtained
by these two approaches are the same. Furthermore, by using the relationship between
solutions of the Yang-Baxter equation and triangular bialgebra structures,
$\mathcal{O}$-operators, we apply this result to the level of triangular bialgebra
structures and $\mathcal{O}$-operators, and obtain the following commutative diagram:
{\small
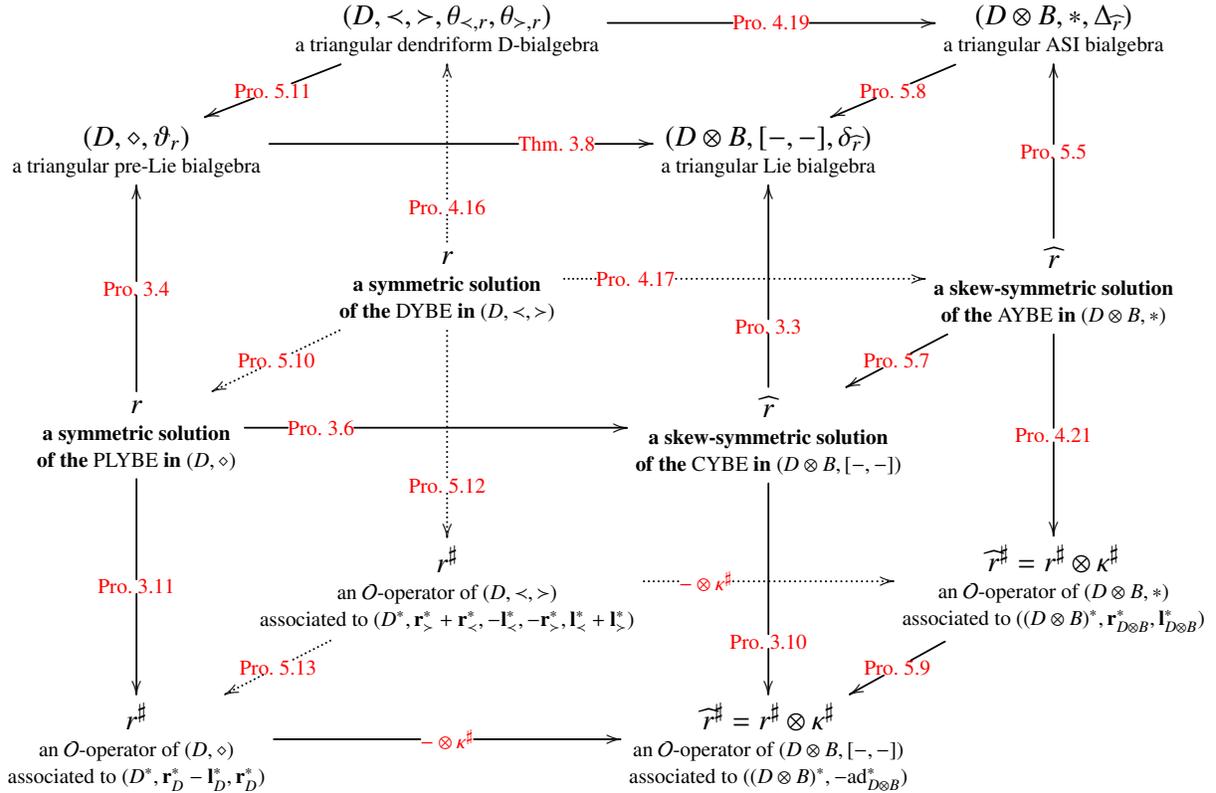
\begin{displaymath}
\xymatrix@R=0.7cm@C=-0.3cm{
&\txt{{\small $(D, \prec, \succ, \theta_{\prec,r}, \theta_{\succ,r})$ }\\
{\tiny a triangular dendriform $\md$-bialgebra}}
\ar@{->}[rr]|-{\txt{\tiny\textcolor{red}{Pro.}~\ref{pro:spre-ASIbia}}}
\ar@{->}[ld]|-{\txt{\tiny\textcolor{red}{Pro.}~\ref{pro:qtD-qtpreLie}}}
\ar@{<.}[dd]|-(0.8){\txt{\tiny\textcolor{red}{Pro.}~\ref{pro:tridend}}}&
&\txt{{\small $(D\otimes B, \ast, \Delta_{\widehat{r}})$} \\
{\tiny a triangular ASI bialgebra}}
\ar@{->}[ld]|-{\txt{\tiny\textcolor{red}{Pro.}~\ref{pro:qtAss-qtLie}}}
\ar@{<-}[dd]|-{\txt{\tiny\textcolor{red}{Pro.}~\ref{pro:qtass-bia}}}&\\
\txt{{\small $(D, \diamond, \vartheta_{r})$} \\ {\tiny a triangular pre-Lie bialgebra}}
\ar@{<-}[dd]|-{\txt{\tiny\textcolor{red}{Pro.}~\ref{pro:spec-plbia}}}
\ar@{->}[rr]|-(0.75){\txt{\tiny\textcolor{red}{Thm.}~\ref{thm:indu-sLiebia}}}&
&\txt{{\small $(D\otimes B, [-,-], \delta_{\widehat{r}})$}\\
{\tiny a triangular Lie bialgebra}}
\ar@{<-}[dd]|-(0.7){\txt{\tiny\textcolor{red}{Pro.}~\ref{pro:splie-bia}}} \\
&\txt{{\small\bf $r$}\\ {\tiny\bf a symmetric solution} \\
{\tiny\bf of the $\DYBE$ in $(D, \prec, \succ)$}}
\ar@{.>}[dd]|-(0.75){\txt{\tiny\textcolor{red}{Pro.}~\ref{pro:o-D}}}
\ar@{.>}[rr]|-(0.2){\rm\textcolor{red}{Pro.}~\ref{pro:DYBE-AYBE}}&
&\txt{{\small\bf $\widehat{r}$} \\  {\tiny\bf a skew-symmetric solution} \\
{\tiny\bf of the $\AYBE$ in $(D\otimes B, \ast)$}}
\ar@{->}[dd]|-{\txt{\tiny\textcolor{red}{Pro.}~\ref{pro:o-ass}}}\\
\txt{{\small\bf $r$}\\ {\tiny\bf a symmetric solution} \\
{\tiny\bf of the $\PLYBE$ in $(D, \diamond)$}}
\ar@{->}[dd]|-{\txt{\tiny\textcolor{red}{Pro.}~\ref{pro:o-PL}}}
\ar@{<.}[ru]|-{\txt{\tiny\textcolor{red}{Pro.}~\ref{pro:D-PL-YBE}}}
\ar@{->}[rr]|-(0.2){\txt{\tiny\textcolor{red}{Pro.}~\ref{pro:PLYBE-CYBE}}}&
&\txt{{\small\bf $\widehat{r}$}\\ {\tiny\bf a skew-symmetric solution} \\
{\tiny\bf of the $\CYBE$ in $(D\otimes B, [-,-])$}}
\ar@{<-}[ru]|-{\txt{\tiny\textcolor{red}{Pro.}~\ref{pro:ass-Lie-YBE}}}
\ar@{->}[dd]|-(0.75){\txt{\tiny\textcolor{red}{Pro.}~\ref{pro:o-lie}}}\\
&\txt{{\small $r^{\sharp}$}\\{\tiny an $\mathcal{O}$-operator of $(D, \prec, \succ)$} \\
{\tiny associated to $(D^{\ast}, \fr_{\succ}^{\ast}+\fr_{\prec}^{\ast}, -\fl_{\prec}^{\ast},
-\fr_{\succ}^{\ast}, \fl_{\prec}^{\ast}+\fl_{\succ}^{\ast})$}}
\ar@{.>}[ld]|-{\txt{\tiny\textcolor{red}{Pro.}~\ref{pro:D-ind-ooper}}} &
&\txt{{\small $\widehat{r}^{\sharp}=r^{\sharp}\otimes\kappa^{\sharp}$}\\
{\tiny an $\mathcal{O}$-operator of $(D\otimes B, \ast)$} \\
{\tiny associated to $((D\otimes B)^{\ast}, \fr_{D\otimes B}^{\ast},
\fl_{D\otimes B}^{\ast})$}}
\ar@{<.}[ll]|-(0.75){\txt{\tiny\textcolor{red}{$-\otimes\kappa^{\sharp}$}}}\\
\txt{{\small $r^{\sharp}$}\\{\tiny an $\mathcal{O}$-operator of $(D, \diamond)$} \\
{\tiny associated to $(D^{\ast}, \fr_{D}^{\ast}-\fl_{D}^{\ast}, \fr_{D}^{\ast})$}}&
&\txt{{\small $\widehat{r}^{\sharp}=r^{\sharp}\otimes\kappa^{\sharp}$}\\
{\tiny an $\mathcal{O}$-operator of $(D\otimes B, [-,-])$} \\
{\tiny associated to $((D\otimes B)^{\ast}, -\ad_{D\otimes B}^{\ast})$}}
\ar@{<-}[ru]|-{\txt{\tiny\textcolor{red}{Pro.}~\ref{pro:Ass-ind-ooper}}}
\ar@{<-}[ll]|-{\txt{\tiny\textcolor{red}{$-\otimes\kappa^{\sharp}$}}}&}
\end{displaymath}
}

This paper is organized as follows. In Section \ref{sec:Prelim} we recall the notions
of dendriform algebras, pre-Lie algebras and perm algebras, together with the bimodules
over dendriform algebras, pre-Lie algebras. We show in Proposition \ref{pro:comm-diag}
that there is a commutative diagram of dendriform algebras, the induced pre-Lie algebras,
the induced associative algebras and Lie algebras (by perm algebras).
In Section \ref{sec:Lie-preli} we analyze the Lie bialgebra structure on the
tensor product of a pre-Lie bialgebra and a quadratic perm algebra, and provide in
Theorem \ref{thm:indu-sLiebia} that the Lie bialgebra structure on this tensor product
is quasi-triangular (resp. triangular, factorizable) if the original pre-Lie bialgebra
is quasi-triangular (resp. triangular, factorizable). In this proof process,
we also provide that a symmetric solution (resp. a solution whose skew-symmetric part
is invariant) of the pre-Lie Yang-Baxter equation in a pre-Lie algebra
$(A, \diamond)$ can induce a skew-symmetric solution (resp. a solution whose symmetric
part is invariant) of the $\CYBE$ in the tensor product Lie algebra $(A\otimes B, [-,-])$.
In Section \ref{sec:dend-ASI} we introduce the notion of completed ASI bialgebras
and show that there is a natural completed ASI bialgebra structure on the tensor product
of a finite-dimensional dendriform $\md$-bialgebra and a quadratic $\bz$-graded
perm algebra. In the special case when the quadratic $\bz$-graded perm algebra is
on the space of Laurent polynomials, we obtain a characterization of the dendriform
$\md$-bialgebra by the affinization. In the special case when the quadratic
$\bz$-graded perm algebra is a finite-dimensional quadratic perm algebra,
we provide that a symmetric solution of the dendriform Yang-Baxter equation in
a dendriform algebra $(D, \prec, \succ)$ can induce a skew-symmetric solution of
the associative Yang-Baxter equation in the tensor product associative
algebra $(D\otimes B, \ast)$.
In Section \ref{sec:dend-lie} we give the proof of Theorem I (Theorem
\ref{thm:bialgebras}), and show that a skew-symmetric solution (resp. a solution whose
symmetric part is invariant) of the associative Yang-Baxter equation in an associative
algebra is also a skew-symmetric solution (resp. a solution whose
symmetric part is invariant) of the induced Lie algerba (by the commutator)
and a symmetric solution of the dendriform Yang-Baxter equation in
a dendriform algebra is also a symmetric solution of the dendriform
Yang-Baxter equation in the induced pre-Lie algebra. In addition to the
conclusions of the previous sections, we have obtained the three-dimensional
commutative diagram above.

Throughout this paper, we fix $\Bbbk$ as a field of characteristic zero.
All the vector spaces, algebras are over $\Bbbk$ and are finite-dimensional
unless otherwise specified, and all tensor products are also over $\Bbbk$.
For any vector space $V$, we denote $V^{\ast}$ the dual space of $V$. We denote the
identity map by $\id$. For vector spaces $A$ and $V$, we fix the following notations:
\begin{enumerate}
\item[$(i)$] Let $f: V\rightarrow V$ be a linear map. Define a linear map $f^{\ast}:
      V^{\ast}\rightarrow V^{\ast}$ by
      $\langle f^{\ast}(\xi),\; v\rangle=\langle\xi,\; f(v)\rangle$,
      for any $v\in V$ and $\xi\in V^{\ast}$.
\item[$(ii)$] Let $\mu: A\rightarrow\End_{\Bbbk}(V)$ be a linear map. Define a linear
      map $\mu^{\ast}: A\rightarrow\End_{\Bbbk}(V^{\ast})$ by
      $\langle\mu^{\ast}(a)(\xi),\; v\rangle=\langle\xi,\; \mu(a)(v)\rangle$,
      for any $a\in A$, $v\in V$ and $\xi\in V^{\ast}$.
\item[$(iii)$] For any $r\in A\otimes A$, we define a linear map
      $r^{\sharp}: A^{\ast}\rightarrow A$ by
      $\langle\xi_{2},\; r^{\sharp}(\xi_{1})\rangle=\langle\xi_{1}\otimes\xi_{2},\;
      r\rangle$, for any $\xi_{1}, \xi_{2}\in A^{\ast}$.
\end{enumerate}

\smallskip
\section{Preliminaries on pre-Lie algebras and dendriform algebras}\label{sec:Prelim}
In this section, we recall the notions of pre-Lie algebras, dendriform algebras,
bimodules over pre-Lie algebras and dendriform algebras, the $\mathcal{O}$-operator of
pre-Lie algebras and dendriform algebras. Starting from a dendriform
algebra, using the tensor product with a perm algebra, we present two methods for
obtaining a Lie algebra, and the Lie algebras obtained by these two methods are
consistent. Dendriform algebras were introduced by Loday.
First, we recall the definition of dendriform algebras.

\begin{defi}\label{def:Dalg}
A {\bf dendriform algebra} $(D, \prec, \succ)$ is a vector space $D$ with two
bilinear operators $\prec, \succ: D\otimes D\rightarrow D$, such that
\begin{align*}
&(d_{1}\prec d_{2})\prec d_{3}=d_{1}\prec(d_{2}\prec d_{3})+d_{1}\prec(d_{2}\succ d_{3}),\\
&\qquad\quad (d_{1}\succ d_{2})\prec d_{3}=d_{1}\succ(d_{2}\prec d_{3}),\\
&d_{1}\succ(d_{2}\succ d_{3})=(d_{1}\prec d_{2})\succ d_{3}+(d_{1}\succ d_{2})\succ d_{3},
\end{align*}
for any $d_{1}, d_{2}, d_{3}\in D$.
\end{defi}

Recall that an {\bf associative Rota-Baxter algebra} is an associative algebra $(A, \ast)$
endowed with a linear operator $R: A\rightarrow A$ subject to the following relation:
$$
R(a_{1})\ast R(a_{2})=R\big(R(a_{1})\ast a_{2}+a_{1}\ast R(a_{2})\big),
$$
for any $a_{1}, a_{2}\in A$.

\begin{ex}\label{ex:D-alg}
$(i)$ Let $(A, \ast, R)$ be an associative Rota-Baxter algebra and define operators
$\prec, \succ: A\otimes A\rightarrow A$ by $a_{1}\prec a_{2}:=a_{1}\ast R(a_{2})$ and
$a_{1}\succ a_{2}:=R(a_{1})\ast a_{2}$ for any $a_{1}, a_{2}\in A$.
Then $(A, \prec, \succ)$ is a dendriform algebra.

$(ii)$ Let $(A, \ast)$ be an associative algebra and $x\in A$ be an element such that
$x^{2}=0$. Define $a_{1}\prec a_{2}:=a_{1}\ast(x\ast a_{2})$ and
$a_{1}\succ a_{2}:=a_{1}\ast(a_{2}\ast x)$ for any $a_{1}, a_{2}\in A$.
Then $(A, \prec, \succ)$ is a dendriform algebra.

$(iii)$ Recently, Rakhimov and Bekbaev provided a classification of 2-dimensional
dendriform algebras in \cite{RB}. Let vector space $D=\Bbbk\{e_{1}, e_{2}\}$.
Define the products by $e_{1}\succ e_{1}=e_{1}$, $e_{2}\prec e_{1}=e_{2}$ and
others are all zero. Then $(D, \prec, \succ)$ is a 2-dimensional dendriform algebra.
\end{ex}

\begin{defi}\label{def:mod-Dalg}
A {\bf bimodule} over a dendriform algebra $(D, \prec, \succ)$ is a quintuple
$(V, \kl_{\prec}, \kr_{\prec}, \kl_{\succ}, \kr_{\succ})$, where $V$ is a vector
space and $\kl_{\prec}, \kr_{\prec}, \kl_{\succ}, \kr_{\succ}: A\rightarrow\End_{\Bbbk}(V)$
are linear maps satisfying the following equations:
\begin{align*}
&\kl_{\prec}(d_{1}\prec d_{2})=\kl_{\prec}(d_{1})\kl_{\prec}(d_{2})
+\kl_{\prec}(d_{1})\kl_{\succ}(d_{2}),&
&\kl_{\prec}(d_{1}\prec d_{2})=\kl_{\succ}(d_{1})\kl_{\prec}(d_{2}),\\
&\kr_{\prec}(d_{1})\kl_{\prec}(d_{2})=\kl_{\prec}(d_{2})\kr_{\prec}(d_{1})
+\kl_{\prec}(d_{2})\kr_{\succ}(d_{1}),&
&\kr_{\prec}(d_{1})\kl_{\succ}(d_{2})=\kl_{\succ}(d_{2})\kr_{\prec}(d_{1}),\\
&\kr_{\prec}(d_{1})\kr_{\prec}(d_{2})=\kr_{\prec}(d_{2}\prec d_{1}+d_{2}\succ d_{1}),&
&\kr_{\prec}(d_{1})\kr_{\succ}(d_{2})=\kr_{\succ}(d_{2}\prec d_{1}),\\
&\kr_{\succ}(d_{1})\kl_{\prec}(d_{2})+\kr_{\succ}(d_{1})\kl_{\succ}(d_{2})
=\kl_{\succ}(d_{2})\kr_{\succ}(d_{1}),&
&\kl_{\succ}(d_{1}\prec d_{2}+d_{1}\succ d_{2})=\kl_{\succ}(d_{1})\kl_{\succ}(d_{2}),\\
&\kr_{\succ}(d_{1})\kr_{\prec}(d_{2})+\kr_{\succ}(d_{1})\kr_{\succ}(d_{2})
=\kr_{\succ}(d_{2}\succ d_{1}),
\end{align*}
for any $d_{1}, d_{2}\in D$
\end{defi}

Let $(D, \prec, \succ)$ be a dendriform algebra. Define linear maps
$\fl_{\prec}, \fr_{\prec}, \fl_{\succ}, \fr_{\succ}: A\rightarrow\End_{\Bbbk}(A)$ by
$\fl_{\prec}(d_{1})(d_{2})=d_{1}\prec d_{2}$, $\fr_{\prec}(d_{1})(d_{2})=d_{2}\prec d_{1}$,
$\fl_{\succ}(d_{1})(d_{2})=d_{1}\succ d_{2}$ and $\fl_{\succ}(d_{1})(d_{2})=
d_{2}\succ d_{1}$ respectively, for any $d_{1}, d_{2}\in D$. Then $(D, \fl_{\prec},
\fr_{\prec}, \fl_{\succ}, \fr_{\succ})$ is a bimodule over $(D, \prec, \succ)$,
which is called the {\bf regular bimodule} over $(D, \prec, \succ)$. Let $(A, \ast)$
be an associative algebra. Recall that the regular bimodule $(A, \fl_{A},
\fr_{A})$ of $(A, \ast)$ is given by $\fl_{A}(a_{1})(a_{2})=a_{1}\ast a_{2}$ and
$\fr_{A}(a_{1})(a_{2})=a_{2}\ast a_{1}$ for any $a_{1}, a_{2}\in D$.
It is easy to see that the regular bimodule $(D, \fl_{D}, \fr_{D})$ of the
sub-adjacent associative algebra $(D, \ast)$ of $(D, \prec, \succ)$ is also
induced by the regular bimodule $(D, \fl_{\prec}, \fr_{\prec}, \fl_{\succ},
\fr_{\succ})$ of $(D, \prec, \succ)$, i.e., $\fl_{D}=\fl_{\prec}+\fl_{\succ}$
and $\fr_{D}=\fr_{\prec}+\fr_{\succ}$.

Dendriform algebras are not only closely related to the associative algebras,
but also closely related to the pre-Lie algebras.

\begin{defi}\label{def:pre-Liealg}
A {\bf pre-Lie algebra} is a pair $(A, \diamond)$, where $A$ is a vector space and
$\diamond: A\otimes A\rightarrow A$ is a bilinear operator such that,
$$
a_{1}\diamond(a_{2}\diamond a_{3})-(a_{1}\diamond a_{2})\diamond a_{3}
=a_{2}\diamond(a_{1}\diamond a_{3})-(a_{2}\diamond a_{1})\diamond a_{3},
$$
for any $a_{1}, a_{2}, a_{3}\in A$.
\end{defi}

In \cite{Agu}, Aguiar pointed out that every dendriform algebra has a pre-Lie
algebra structure.

\begin{pro}[\cite{Agu,Bai}]\label{pro:dend-pre}
Let $(D, \prec, \succ)$ be a dendriform algebra.
Define a bilinear operator $\diamond: D\otimes D\rightarrow D$ by
$$
d_{1}\diamond d_{2}=d_{1}\succ d_{2}-d_{2}\prec d_{1},
$$
for any $d_{1}, d_{2}\in D$. Then $(D, \diamond)$ is a pre-Lie algebra,
called the {\bf induced pre-Lie algebra} from $(D, \prec, \succ)$.
\end{pro}


Recall the notion of bimodule over a pre-Lie algebra.

\begin{defi}\label{def:pre-Liemod}
Let $(A, \diamond)$ be a pre-Lie algebra. A {\bf bimodule} over $(A, \diamond)$ is a
triple $(V, \kl, \kr)$, such that $V$ is a vector space, and $\kl, \kr: A\rightarrow
\End_{\Bbbk}(V)$ are linear maps satisfying
\begin{align*}
\kl(a_{1})\kl(a_{2})-\kl(a_{1}\diamond a_{2})
&=\kl(a_{2})\kl(a_{1})-\kl(a_{2}\diamond a_{1}),\\
\kl(a_{1})\kr(a_{2})-\kr(a_{2})\kl(a_{1})
&=\kr(a_{1}\diamond a_{2})-\kr(a_{2})\kr(a_{1}),
\end{align*}
for any $a_{1}, a_{2}\in A$.
\end{defi}

Let $(A, \diamond)$ be a pre-Lie algebra. Define linear maps $\fl_{A}, \fr_{A}:
A\rightarrow\End_{\Bbbk}(A)$ by $\fl_{A}(a_{1})(a_{2})=a_{1}\diamond a_{2}$ and
$\fr_{A}(a_{1})(a_{2})=a_{2}\diamond a_{1}$ respectively, for any $a_{1}, a_{2}\in A$.
Then $(A, \fl_{A}, \fr_{A})$ is a bimodule over $(A, \diamond)$,
which is called the {\bf regular bimodule} over $(A, \diamond)$.

Given a pre-Lie algebra $(A, \diamond)$ one may define a bracket on $A$ by
$[a_{1}, a_{2}]=a_{1}\diamond a_{2}-a_{2}\diamond a_{1}$ for any $a_{1}, a_{2}\in A$.
Then $(A, [-,-])$ is a Lie algebra, which is called the sub-adjacent
Lie algebra of $(A, \diamond)$. Moreover, since the properad of pre-Lie algebras
and the the properad of perm algebras are Koszul dual, we can also get a Lie algebra
by the tensor product of a pre-Lie algebra and a perm algebra.

\begin{defi}\label{def:permalg}
A {\bf perm algebra} is a pair $(B, \cdot)$, where $B$ is a vector space and
$\cdot: B\otimes B\rightarrow B$ is a bilinear operator such that,
$$
b_{1}(b_{2}b_{3})=(b_{1}b_{2})b_{3}=(b_{2}b_{1})b_{3},
$$
for any $b_{1}, b_{2}, b_{3}\in B$, where $b_{1}b_{2}:=b_{1}\cdot b_{2}$.
\end{defi}

\begin{pro}[\cite{Val,LZB}]\label{pro:perm-per}
Let $(A, \diamond)$ be a pre-Lie algebra and $(B, \cdot)$ be a perm algebra.
Define a bracket $[-,-]:(A\otimes B)\otimes(A\otimes B)\rightarrow A\otimes B$ by
$$
[a_{1}\otimes b_{1},\; a_{2}\otimes b_{2}]=(a_{1}\diamond a_{2})\otimes(b_{1}b_{2})
-(a_{2}\diamond a_{1})\otimes(b_{2}b_{1}),
$$
for any $a_{1}, a_{2}\in A$ and $b_{1}, b_{2}\in B$. Then $(A\otimes B, [-,-])$
is a Lie algebra, called {\bf the induced Lie algebra} from $(A, \diamond)$ by $(B, \cdot)$.
\end{pro}

Similarly,  we can get an associative algebra
by the tensor product of a dendriform algebra and a perm algebra.

\begin{pro}\label{pro:perm-dend}
Let $(D, \prec, \succ)$ be a dendriform algebra and $(B, \cdot)$ be a perm algebra.
Define a bilinear operator $\ast:(D\otimes B)\otimes(D\otimes B)\rightarrow D\otimes B$ by
$$
(d_{1}\otimes b_{1})\ast(d_{2}\otimes b_{2})=(d_{1}\succ d_{2})\otimes(b_{1}b_{2})
+(d_{1}\prec d_{2})\otimes(b_{2}b_{1}),
$$
for any $d_{1}, d_{2}\in D$ and $b_{1}, b_{2}\in B$. Then $(D\otimes B, \ast)$
is an associative algebra, called {\bf the induced associative algebra} from
$(D, \prec, \succ)$ by $(B, \cdot)$.
\end{pro}

\begin{proof}
For any $d_{1}, d_{2}, d_{3}\in D$ and $b_{1}, b_{2}, b_{3}\in B$,
\begin{align*}
&\;((d_{1}\otimes b_{1})\ast(d_{2}\otimes b_{2}))\ast(d_{3}\otimes b_{3})\\
=&\;\big((d_{1}\succ d_{2})\otimes(b_{1}b_{2})
+(d_{1}\prec d_{2})\otimes(b_{2}b_{1})\big)\ast(d_{3}\otimes b_{3})\\
=&\;((d_{1}\succ d_{2})\succ d_{3})\otimes((b_{1}b_{2})b_{3})
+((d_{1}\succ d_{2})\prec d_{3})\otimes(b_{3}(b_{1}b_{2}))\\[-1mm]
&\qquad+((d_{1}\prec d_{2})\succ d_{3})\otimes((b_{2}b_{1})b_{3})
+((d_{1}\prec d_{2})\prec d_{3})\otimes(b_{3}(b_{2}b_{1}))\\
=&\;(d_{1}\otimes b_{1})\ast\big((d_{2}\succ d_{3})\otimes(b_{2}b_{3})
+(d_{2}\prec d_{3})\otimes(b_{3}b_{2})\big)\\
=&\;(d_{1}\otimes b_{1})\ast((d_{2}\otimes b_{2})\ast(d_{3}\otimes b_{3})).
\end{align*}
Thus, $(D\otimes B, \ast)$ is an associative algebra.
\end{proof}

Moreover, note that each associative algebra $(A, \ast)$ has a Lie algebra structure
by the commutator, i.e., the bracket $[-,-]: A\otimes A\rightarrow A$ is given by
$[a_{1}, a_{2}]=a_{1}\diamond a_{2}-a_{2}\diamond a_{1}$ for any $a_{1}, a_{2}\in A$.
This Lie algebra $(A, [-,-])$ is called the sub-adjacent Lie algebra of $(A, \ast)$.
We have multiple methods to construct a Lie algebra from a dendriform algebra.
One approach is to use the tensor product with a perm algebra to obtain an associative
algebra, and then use commutator to obtain a Lie algebra; Another approach is to
first obtain a pre-Lie algebra, and then use the tensor product of a pre-Lie algebra
and a perm algebra to obtain the Lie algebra. In fact, the Lie algebras obtained
by these two methods are the same.

\begin{pro}\label{pro:comm-diag}
Let $(D, \prec, \succ)$ be a dendriform algebra and $(B, \cdot)$ be a perm algebra.
Then we have a commutative diagram:
$$
\xymatrix@C=3cm@R=0.5cm{
\txt{$(D, \prec, \succ)$ \\ {\tiny a dendriform algebra}}
\ar[d]_-{\mbox{\tiny $d_{1}\diamond d_{2}=d_{1}\succ d_{2}-d_{2}\prec d_{1}$}}
\ar[r]^-{\mbox{\tiny $-\otimes B$}}
&\txt{$(D\otimes B, \ast)$\\ {\tiny an associative algebra}}
\ar[d]^-{\mbox{\tiny $[x_{1}, x_{2}]=x_{1}\ast x_{2}-x_{2}\ast x_{1}$}} \\
\txt{$(D, \diamond)$ \\ {\tiny a pre-Lie algebra}} \ar[r]^-{\mbox{\tiny $-\otimes B$}}
& \txt{$(D\otimes B, [-,-])$ \\
{\tiny a Lie algebra}}}
$$
\end{pro}

\begin{proof}
On the one hand, by Proposition \ref{pro:perm-dend}, we get a Lie algebra structure
on $D\otimes B$, where the bracket is given by
\begin{align*}
&\;[d_{1}\otimes b_{1},\; d_{2}\otimes b_{2}]\\
=&\;(d_{1}\otimes b_{1})\ast(d_{2}\otimes b_{2})-(d_{2}\otimes b_{2})\ast(d_{1}\otimes b_{1})\\
=&\;(d_{1}\succ d_{2})\otimes(b_{1}b_{2})+(d_{1}\prec d_{2})\otimes(b_{2}b_{1})
-(d_{2}\succ d_{1})\otimes(b_{2}b_{1})-(d_{2}\prec d_{1})\otimes(b_{1}b_{2}),
\end{align*}
for any $b_{1}, b_{2}\in B$ and $d_{1}, d_{2}\in D$.
On the other hand, by Propositions \ref{pro:perm-per} and \ref{pro:dend-pre},
we get a Lie algebra $(D\otimes B, [-,-])$, where
\begin{align*}
&\;[d_{1}\otimes b_{1},\; d_{2}\otimes b_{2}]\\
=&\;(d_{1}\diamond d_{2})\otimes(b_{1}b_{2})-(d_{2}\diamond d_{1})\otimes(b_{2}b_{1})\\
=&\;(d_{1}\succ d_{2})\otimes(b_{1}b_{2})-(d_{2}\prec d_{1})\otimes(b_{1}b_{2})
-(d_{2}\succ d_{1})\otimes(b_{2}b_{1})+(d_{1}\prec d_{2})\otimes(b_{2}b_{1}),
\end{align*}
for any $b_{1}, b_{2}\in B$ and $d_{1}, d_{2}\in D$. Thus, the diagram is commutative.
\end{proof}

\begin{ex}\label{ex:comm-algs}
Consider the 2-dimensional dendriform algebra $(D=\Bbbk\{e_{1}, e_{2}\}, \prec, \succ)$
given in Example \ref{ex:D-alg}, i.e., $e_{1}\succ e_{1}=e_{1}$ and $e_{2}\prec e_{1}
=e_{2}$. Then $(D, \prec, \succ)$ induces a pre-Lie algebra $(D, \diamond)$, where the
nonzero products are given by $e_{1}\diamond e_{1}=e_{1}$ and $e_{1}\diamond e_{2}=-e_{2}$.
Let $(B=\Bbbk\{x_{1}, x_{2}\}, \cdot)$ be a perm algebra given by $x_{2}x_{1}=x_{1}$
and $x_{2}x_{2}=x_{2}$. Then by Proposition \ref{pro:perm-per}, we get a 4-dimensional
Lie algebra $(D\otimes B, [-,-])$, where
$$
[e_{1}\otimes x_{1},\; e_{1}\otimes x_{2}]=-e_{1}\otimes x_{1},\quad
[e_{1}\otimes x_{2},\; e_{2}\otimes x_{1}]=-e_{2}\otimes x_{1},\quad
[e_{1}\otimes x_{2},\; e_{2}\otimes x_{2}]=-e_{2}\otimes x_{2},
$$
and others are all zero. On the other hand, by the tensor product of $(D, \prec, \succ)$
and $(B, \cdot)$, we obtain a 4-dimensional associative algebra $(D\otimes B, \ast)$,
where
\begin{align*}
& (e_{1}\otimes x_{2})\ast(e_{1}\otimes x_{1})=e_{1}\otimes x_{1},\qquad
(e_{1}\otimes x_{2})\ast(e_{1}\otimes x_{2})=e_{1}\otimes x_{2},\\
& (e_{2}\otimes x_{1})\ast(e_{1}\otimes x_{2})=e_{2}\otimes x_{1},\qquad
(e_{2}\otimes x_{2})\ast(e_{1}\otimes x_{2})=e_{2}\otimes x_{2},
\end{align*}
and others are all zero. It is easy to see that the Lie algebra obtained by taking the
commutator of $(D\otimes B, \ast)$ is exactly $(D\otimes B, [-,-])$.
\end{ex}

In this paper, we will prove that this commutative diagram is also correct at
the level of bialgebras. 

\bigskip
\section{Quasi-triangular Lie bialgebras induced by pre-Lie bialgebras}\label{sec:Lie-preli}
In \cite{LZB}, a Lie bialgebra structure on the tensor product of a pre-Lie bialgebra
and a quadratic perm algebra is given. In this section, we consider a special class
of Lie bialgebras, the quasi-triangular Lie bialgebras. We show that a quasi-triangular
pre-Lie bialgebra induces a quasi-triangular Lie bialgebra structure on the tensor product.

Recall that a {\bf pre-Lie coalgebra} is a pair $(A, \vartheta)$, where $A$ is a vector
space and $\vartheta: A\rightarrow A\otimes A$ is a linear map satisfying
$$
(\id\otimes\vartheta)\vartheta-(\tau\otimes\id)(\id\otimes\vartheta)\vartheta
=(\vartheta\otimes\id)\vartheta-(\tau\otimes\id)(\vartheta\otimes\id)\vartheta,
$$
where $\tau: A\otimes A\rightarrow A\otimes A$ is the flip operator defined by
$\tau(a_{1}\otimes a_{2}):=a_{2}\otimes a_{1}$ for all $a_{1}, a_{2}\in A$.

\begin{defi}[\cite{Bai1}]\label{def:pre-li-bia}
Let $(A, \diamond)$ be a pre-Lie algebra and $(A, \vartheta)$ be a pre-Lie coalgebra.
If for any $a_{1}, a_{2}\in A$, the following compatibility conditions are satisfied:
\begin{align*}
&(\vartheta-\tau\vartheta)(a_{1}\diamond a_{2})
-(\fl_{A}(a_{1})\otimes\id)((\vartheta-\tau\vartheta)(a_{2}))\\[-1mm]
&\quad-(\id\otimes\fl_{A}(a_{1}))((\vartheta-\tau\vartheta)(a_{2}))
-(\id\otimes\fr_{A}(a_{2}))(\vartheta(a_{1}))
+(\fr_{A}(a_{2})\otimes\id)(\tau(\vartheta(a_{1})))=0, \\
&\vartheta(a_{1}\diamond a_{2}-a_{2}\diamond a_{1})
-(\id\otimes(\fr_{A}(a_{2})-\fl_{A}(a_{2})))(\vartheta(a_{1}))\\[-1mm]
&\quad-(\id\otimes(\fl_{A}(a_{1})-\fr_{A}(a_{1})))(\vartheta(a_{2}))
-(\fl_{A}(a_{1})\otimes\id)(\vartheta(a_{2}))
+(\fl_{A}(a_{2})\otimes\id)(\vartheta(a_{1}))=0,
\end{align*}
then we call $(A, \diamond, \vartheta)$ a {\bf pre-Lie bialgebra}.
\end{defi}

Recall that  {\bf Lie coalgebra} $(\g, \delta)$
is a vector space $\g$ with a linear map $\delta: \g\rightarrow\g\otimes\g$,
such that $\tau\delta=-\delta$ and $(\id\otimes\delta)\delta
-(\tau\otimes\id)(\id\otimes\delta)\delta=(\delta\otimes\id)\delta$.
A {\bf Lie bialgebra} is a triple $(\g, [-,-], \delta)$ such that
$(\g, [-,-])$ is a Lie algebra, $(\g, \delta)$ is a Lie coalgebra, and the following
compatibility condition holds:
$$
\delta([g_{1}, g_{2}])=(\ad_{\g}(g_{1})\otimes\id+\id\otimes\ad_{\g}(g_{1}))
(\delta(g_{2}))-(\ad_{\g}(g_{2})\otimes\id+\id\otimes\ad_{\g}(g_{2}))(\delta(g_{1})),
$$
where $\ad_{\g}(g_{1})(g_{2})=[g_{1}, g_{2}]$, for all $g_{1}, g_{2}\in\g$.

\begin{thm}[\cite{LZB}]\label{thm:pre-lie}
Let $(A, \diamond, \vartheta)$ be a pre-Lie bialgebra and $(B, \circ, \omega)$
be a quadratic perm algebra, and $(A\otimes B, [-,-])$ be the induced Lie
algebra from $(A, \diamond)$ by $(B, \circ)$. Define a linear map $\delta:
A\otimes B\rightarrow(A\otimes B)\otimes(A\otimes B)$ by
\begin{align}
\delta(a\otimes b)=(\id\otimes\id-\tau)(\vartheta(a)\bullet\nu_{\omega}(b))
:=(\id\otimes\id-\tau)\Big(\sum_{(a)}\sum_{(b)}
(a_{(1)}\otimes b_{(1)})\otimes(a_{(2)}\otimes b_{(2)})\Big), \label{pl-pm}
\end{align}
for any $a\in A$ and $b\in B$, where $\vartheta(a)=\sum_{(a)}a_{(1)}\otimes a_{(2)}$
and $\nu_{\omega}(b)=\sum_{(b)}b_{(1)}\otimes b_{(2)}$ in the Sweedler notation.
Then $(A\otimes B, [-,-], \delta)$ is a Lie bialgebra, which is called the
{\bf Lie bialgebra induced from $(A, \diamond, \vartheta)$ by $(B, \circ, \omega)$}.
\end{thm}

Recall that a Lie bialgebra $(\g, [-,-], \delta)$ is called {\bf coboundary}
if there exists an element $r\in\g\otimes\g$ such that $\delta=\delta_{r}$,
where
\begin{align}
\delta_{r}(g)=(\id\otimes\ad_{\g}(g)+\ad_{\g}(g)\otimes\id)(r), \label{lie-cobo}
\end{align}
for any $g\in\g$. Let $(\g, [-,-])$ be a Lie algebra. An element
$r=\sum_{i}x_{i}\otimes y_{i}\in\g\otimes\g$ is said to be {\bf $\ad_{\g}$-invariant}
if $(\id\otimes\ad_{\g}(g)+\ad_{\g}(g)\otimes\id)(r)=0$. The equation
$$
\mathbf{C}_{r}:=[r_{12}, r_{13}]+[r_{13}, r_{23}]+[r_{12}, r_{23}]=0
$$
is called the {\bf classical Yang-Baxter equation} (or, $\CYBE$) in $(\g, [-,-])$,
where $[r_{12}, r_{13}]=\sum_{i,j}[x_{i}, x_{j}]\otimes y_{i}\otimes y_{j}$,
$[r_{13}, r_{23}]=\sum_{i,j}x_{i}\otimes x_{j}\otimes[y_{i}, y_{j}]$ and
$[r_{12}, r_{23}]=\sum_{i,j}x_{i}\otimes[y_{i}, x_{j}]\otimes y_{j}$.

\begin{pro}[\cite{RS,LS}]\label{pro:splie-bia}
Let $(\g, [-,-])$ be a Lie algebra, $r\in\g\otimes\g$ and $\delta_{r}:
\g\rightarrow\g\otimes\g$ be the linear map defined by Eq. \eqref{lie-cobo}.
\begin{enumerate}
\item[$(i)$] If $r$ is a solution of the $\CYBE$ in $(\g, [-,-])$ and $r+\tau(r)$ is
     $\ad_{\g}$-invariant, then $(\g, [-,-], \delta_{r})$ is a Lie bialgebra,
     which is called a {\bf quasi-triangular Lie bialgebra} associated with $r$.
\item[$(ii)$] If $r$ is a skew-symmetric solution of the $\CYBE$ in $(\g, [-,-])$ then
     $(\g, [-,-], \delta_{r})$ is a Lie bialgebra, which is called a {\bf triangular Lie
     bialgebra} associated with $r$.
\item[$(iii)$] If $(\g, [-,-], \delta_{r})$ is a quasi-triangular Lie bialgebra and
     $\mathcal{I}=r^{\sharp}+\tau(r)^{\sharp}: \g^{\ast}\rightarrow\g$ is an isomorphism
     of vector spaces, then $(\g, [-,-], \delta_{r})$ is called a {\bf factorizable Lie
     bialgebra} associated with $r$.
\end{enumerate}
\end{pro}

Recall that a pre-Lie bialgebra $(A, \diamond, \vartheta)$ is called
{\bf coboundary} if there exists $r\in A\otimes A$ such that
\begin{align}
\vartheta(a)=\vartheta_{r}(a):=\big(\fl_{A}(a)\otimes\id
+\id\otimes(\fl_{A}-\fr_{A})(a)\big)(r), \label{preli-cobo}
\end{align}
for all $a\in A$. Let $(A, \diamond)$ be a pre-Lie algebra and $r\in A\otimes A$.
We say $r$ is a solution of the {\bf pre-Lie Yang-Baxter equation} (or $\PLYBE$)
in $(A, \diamond)$ if
\begin{align*}
\mathbf{PL}_{r}:&=r_{13}\diamond r_{12}+r_{23}\diamond r_{12}+r_{21}\diamond r_{13}
+r_{23}\diamond r_{13} \\[-1mm]
&\quad-r_{23}\diamond r_{21}-r_{12}\diamond r_{23}
-r_{13}\diamond r_{21}-r_{13}\diamond r_{23}=0,
\end{align*}
where $r_{13}\diamond r_{12}:=\sum_{i,j}(x_{i}\diamond x_{j})\otimes y_{j}\otimes y_{i}$,
$r_{23}\diamond r_{12}:=\sum_{i,j}x_{j}\otimes(x_{i}\diamond y_{j})\otimes y_{i}$,
$r_{21}\diamond r_{13}:=\sum_{i,j}(y_{i}\diamond x_{j})\otimes x_{i}\otimes y_{j}$,
$r_{23}\diamond r_{13}:=\sum_{i,j}x_{j}\otimes x_{i}\otimes(y_{i}\diamond y_{j})$,
$r_{23}\diamond r_{21}:=\sum_{i,j}y_{j}\otimes(x_{i}\diamond x_{j})\otimes y_{i}$,
$r_{12}\diamond r_{23}:=\sum_{i,j}x_{i}\otimes(y_{i}\diamond x_{j})\otimes y_{j}$,
$r_{13}\diamond r_{21}:=\sum_{i,j}(x_{i}\diamond y_{j})\otimes x_{j}\otimes y_{i}$,
$r_{13}\diamond r_{23}:=\sum_{i,j}x_{i}\otimes x_{j}\otimes(y_{i}\diamond y_{j})$.
One can check that $-r$ and $\tau(r)$ are solutions of the $\PLYBE$ in
$(A, \diamond)$ if $r\in A\otimes A$ is also a solution of the $\PLYBE$ in
$(A, \diamond)$. Recall that $r\in A\otimes A$ is called {\bf symmetric}
if $r=\tau(r)$, and $r$ is called {\bf $(\fl_{A}, \ad_{A})$-invariant} if
$$
\big(\fl_{A}(a)\otimes\id+\id\otimes(\fl_{A}-\fr_{A})(a)\big)(r)=0,
$$
for any $a\in A$.

\begin{pro}[\cite{Bai1,WBLS}]\label{pro:spec-plbia}
Let $(A, \diamond)$ be a pre-Lie algebra, $r\in A\otimes A$ and
$\vartheta_{r}: A\rightarrow A\otimes A$ be a linear map defined by Eq. \eqref{preli-cobo}.
Then, we have
\begin{enumerate}
\item[$(i)$] if $r$ is a symmetric solution of the $\PLYBE$ in $(A, \diamond)$,
      then $(A, \diamond, \vartheta_{r})$ is a pre-Lie bialgebra, which is
      called a {\bf triangular pre-Lie bialgebra} associated with $r$;
\item[$(ii)$] if $r$ is a solution of the $\PLYBE$ in $(A, \diamond)$ and $r-\tau(r)$ is
      $(\fl_{A}, \ad_{A})$-invariant, then $(A, \diamond, \vartheta_{r})$ is a
      pre-Lie bialgebra, which is called a {\bf quasi-triangular pre-Lie bialgebra}
      associated with $r$;
\item[$(iii)$] if $(A, \diamond, \vartheta_{r})$ is a quasi-triangular pre-Lie bialgebra and
     $\mathcal{I}=r^{\sharp}-\tau(r)^{\sharp}: A^{\ast}\rightarrow A$ is an isomorphism
     of vector spaces, then $(A, \diamond, \vartheta_{r})$ is called a {\bf factorizable
     pre-Lie bialgebra} associated with $r$.
\end{enumerate}
\end{pro}

Let $(B, \cdot)$ be a perm algebra. A bilinear form $\omega(-,-)$ on
$(B, \cdot)$ is called {\bf invariant} if it satisfies
$$
\omega(b_{1}b_{2},\; b_{3})=\omega(b_{1},\; b_{2}b_{3}-b_{3}b_{2}),
$$
for any $b_{1}, b_{2}, b_{3}\in B$. A {\bf quadratic perm algebra},
denoted by $(B, \cdot, \omega)$, is a perm algebra $(B, \cdot)$
together with an antisymmetric invariant nondegenerate bilinear form $\omega(-,-)$.
Let $(B, \cdot, \omega)$ be a quadratic perm algebra. Then the bilinear form
$\omega(-,-)$ can naturally expand to the tensor product $B\otimes B\otimes\cdots\otimes B$,
i.e.,
$$
\omega(-,-):\qquad (\underbrace{B\otimes\cdots\otimes B}_{\mbox{\tiny $k$-fold}})\otimes
(\underbrace{B\otimes\cdots\otimes B}_{\mbox{\tiny $k$-fold}})\longrightarrow\Bbbk,
$$
$\omega(b_{1}\otimes b_{2}\otimes\cdots\otimes b_{k},\ \ b'_{1}\otimes b'_{2}
\otimes\cdots\otimes b'_{k})=\sum_{i=1}^{k}\sum_{j=1}^{k}\omega(b_{i}, b'_{j})$,
for any $b_{1}, b_{2},\cdots, b_{k}, b'_{1}, b'_{2},\cdots, b'_{k}\in B$.
Then $\omega(-,-)$ on $B\otimes B\otimes\cdots\otimes B$ is also a
nondegenerate bilinear form. Recall that a {\bf perm coalgebra} $(B, \nu)$ is a
vector space $B$ with a linear map $\nu: B\rightarrow B\otimes B$ satisfying
$(\vartheta\otimes\id)\vartheta=(\id\otimes\vartheta)\vartheta
=(\tau\otimes\id)(\id\otimes\vartheta)\vartheta$.

\begin{lem}[\cite{LZB}]\label{lem:l-ass-dual}
Let $(B, \cdot, \omega)$ be a quadratic perm algebra. Define a linear map
$\nu_{\omega}: B\rightarrow B\otimes B$ by $\omega(\nu_{\omega}(b_{1}),\;
b_{2}\otimes b_{3})=-\omega(b_{1},\; b_{2}b_{3})$, for any $b_{1}, b_{2}, b_{3}\in B$.
Then $(B, \nu_{\omega})$ is a perm coalgebra.
\end{lem}

Next, we consider the relation between the solutions of the $\PLYBE$ in a pre-Lie
algebra and the solutions of the $\CYBE$ in the induced Lie algebra.
Let $(B, \cdot, \omega)$ be a quadratic perm algebra and
$\{e_{1}, e_{2},\cdots, e_{n}\}$ be a basis of $B$. Since $\omega(-,-)$ is
antisymmetric nondegenerate, we get a basis $\{f_{1}, f_{2},\cdots, f_{n}\}$ of $B$,
which is called the dual basis of $\{e_{1}, e_{2},\cdots, e_{n}\}$ with respect to
$\omega(-,-)$, by $\omega(e_{i}, f_{j})=\delta_{ij}$, where $\delta_{ij}$ is the
Kronecker delta. Here, for some special solutions of the $\PLYBE$ in a pre-Lie
algebra, we have

\begin{pro}\label{pro:PLYBE-CYBE}
Let $(A, \diamond)$ be a pre-Lie algebra and $(B, \cdot, \omega)$ be a
quadratic perm algebra, and $(A\otimes B, [-,-])$ be the induced Lie algebra.
Suppose that $r=\sum_{i} x_{i}\otimes y_{i}\in A\otimes A$ is a solution
of the $\PLYBE$ in $(A, \diamond)$, $r-\tau(r)$ is $(\fl_{A},\ad_{A})$-invariant. Then
\begin{align}
\widehat{r}=\sum_{i, j}(x_{i}\otimes e_{j})\otimes(y_{i}\otimes f_{j})
\in(A\otimes B)\otimes(A\otimes B)  \label{r-max}
\end{align}
is a solution of the $\CYBE$ in $(A\otimes B, [-,-])$ and $\widehat{r}+\tau(\widehat{r})$
is $\ad_{A\otimes B}$-invariant, where $\{e_{1}, e_{2},\cdots, e_{n}\}$ is a basis of $B$
and $\{f_{1}, f_{2},\cdots, f_{n}\}$ is the dual basis of $\{e_{1}, e_{2},\cdots, e_{n}\}$
with respect to $\omega(-,-)$.
\end{pro}

\begin{proof}
First, for any $1\leq p, q\leq n$, we have
\begin{align*}
&\;[\widehat{r}_{12},\; \widehat{r}_{13}]+[\widehat{r}_{12},\; \widehat{r}_{23}]
+[\widehat{r}_{13},\; \widehat{r}_{23}]\\
=&\;\sum_{i,j}\sum_{p,q}\big((x_{i}\diamond x_{j})\otimes y_{i}\otimes y_{j}\big)
\bullet\big((e_{p}e_{q})\otimes f_{p}\otimes f_{q}\big)
-\big((x_{j}\diamond x_{i})\otimes y_{i}\otimes y_{j}\big)
\bullet\big((e_{q}e_{p})\otimes f_{p}\otimes f_{q}\big)\\[-4mm]
&\qquad\qquad+\big(x_{i}\otimes(y_{i}\diamond x_{j})\otimes
y_{j}\big)\bullet\big(e_{p}\otimes(f_{p}e_{q})\otimes f_{q}\big)
-\big(x_{i}\otimes(x_{j}\diamond y_{i})\otimes
y_{j}\big)\bullet\big(e_{p}\otimes(e_{q}f_{p})\otimes f_{q}\big)\\[-1mm]
&\qquad\qquad+\big(x_{i}\otimes x_{j}\otimes(y_{i}\diamond y_{j})\big)\bullet
\big(e_{p}\otimes e_{q}\otimes(f_{p}f_{q})\big)
-\big(x_{i}\otimes x_{j}\otimes(y_{j}\diamond y_{i})\big)\bullet
\big(e_{p}\otimes e_{q}\otimes(f_{q}f_{p})\big).
\end{align*}
Moreover, for given $s, u, v\in\{1, 2,\cdots, n\}$, we have
\begin{align*}
\omega\Big(e_{s}\otimes e_{u}\otimes e_{v},\;
\sum_{p,q}(e_{p}e_{q})\otimes f_{p}\otimes f_{q}\Big)
&=\omega(e_{s},\; e_{u}e_{v}),\\[-2mm]
\omega\Big(e_{s}\otimes e_{u}\otimes e_{v},\;
\sum_{p,q}(e_{q}e_{p})\otimes f_{p}\otimes f_{q}\Big)
&=\omega(e_{s},\; e_{v}e_{u}),\\[-2mm]
\omega\Big(e_{s}\otimes e_{u}\otimes e_{v},\;
\sum_{p,q} e_{p}\otimes(f_{p}e_{q})\otimes f_{q}\Big)
&=\omega(e_{s},\; e_{v}e_{u}-e_{u}e_{v}).
\end{align*}
By the nondegeneracy of $\omega(-,-)$, we get
$$
\sum_{p,q}e_{p}\otimes(f_{p}e_{q})\otimes f_{q}
=\sum_{p,q}\big((e_{q}e_{p})\otimes f_{p}\otimes f_{q}
-(e_{p}e_{q})\otimes f_{p}\otimes f_{q}\big).
$$
Similarly, we also have $\sum_{p,q}e_{p}\otimes e_{q}\otimes(f_{p}f_{q})
=\sum_{p,q}\big((e_{q}e_{p})\otimes f_{p}\otimes f_{q}
-(e_{p}e_{q})\otimes f_{p}\otimes f_{q}\big)$,
$\sum_{p,q}e_{p}\otimes(e_{q}f_{p})\otimes f_{q}
=\sum_{p,q}(e_{q}e_{p})\otimes f_{p}\otimes f_{q}$ and
$\sum_{p,q}e_{p}\otimes e_{q}\otimes(f_{q}f_{p})
=-\sum_{p,q}(e_{p}e_{q})\otimes f_{p}\otimes f_{q}$.
Thus, we obtain
\begin{align*}
&\;[\widehat{r}_{12},\; \widehat{r}_{13}]+[\widehat{r}_{12},\; \widehat{r}_{23}]
+[\widehat{r}_{13},\; \widehat{r}_{23}]\\
=&\;\sum_{i,j}\sum_{p,q}\Big((x_{i}\diamond x_{j})\otimes y_{i}\otimes y_{j}
-x_{i}\otimes(y_{i}\diamond x_{j})\otimes y_{j}
-x_{i}\otimes x_{j}\otimes(y_{i}\diamond y_{j})\\[-5mm]
&\qquad\qquad\qquad\qquad\qquad\qquad+x_{i}\otimes x_{j}\otimes(y_{j}\diamond y_{i})
\Big)\bullet\big((e_{p}e_{q})\otimes f_{p}\otimes f_{q}\big)\\[-2mm]
&\quad-\Big((x_{j}\diamond x_{i})\otimes y_{i}\otimes y_{j}
-x_{i}\otimes(y_{i}\diamond x_{j})\otimes y_{j}
+x_{i}\otimes(x_{j}\diamond y_{i})\otimes y_{j}\\[-2mm]
&\qquad\qquad\qquad\qquad\qquad\qquad-x_{i}\otimes x_{j}\otimes(y_{i}\diamond y_{j})
\Big)\bullet\big((e_{q}e_{p})\otimes f_{p}\otimes f_{q}\big).
\end{align*}
Since $r-\tau(r)$ is $(\fl_{A},\ad_{A})$-invariant, we get
$\mathbf{PL}_{r}=\sum_{i,j}\big((y_{i}\diamond x_{j})\otimes x_{i}\otimes y_{j}
+x_{j}\otimes x_{i}\otimes(y_{i}\diamond y_{j})-x_{i}\otimes x_{j}\otimes(y_{i}\diamond y_{j})
-y_{j}\otimes(x_{j}\diamond x_{i})\otimes y_{i}\big)
=\sum_{i,j}\Big((y_{j}\diamond y_{i})\otimes x_{j}\otimes x_{i}
+y_{i}\otimes x_{j}\otimes(y_{j}\diamond x_{i})-y_{i}\otimes x_{j}\otimes(x_{i}\diamond y_{j})
-y_{j}\otimes(x_{j}\diamond x_{i})\otimes y_{i}\Big)$, and so that
\begin{align*}
-(\tau\otimes\id)(\mathbf{PL}_{r})
&=\sum_{i,j}\Big((x_{i}\diamond x_{j})\otimes y_{i}\otimes y_{j}
-x_{i}\otimes(y_{i}\diamond x_{j})\otimes y_{j}\\[-5mm]
&\qquad\quad-x_{i}\otimes x_{j}\otimes(y_{i}\diamond y_{j})
+x_{i}\otimes x_{j}\otimes(y_{j}\diamond y_{i})\Big),\\
(\id\otimes\tau)(\tau\otimes\id)(\mathbf{PL}_{r})
&=\sum_{i,j}\Big(x_{j}\otimes x_{i}\otimes(y_{j}\diamond y_{i})
+x_{j}\otimes(y_{j}\diamond x_{i})\otimes y_{i}\\[-5mm]
&\qquad\quad-x_{j}\otimes(x_{i}\diamond y_{j})\otimes y_{i}
-(x_{j}\diamond x_{i})\otimes y_{i}\otimes y_{j}\Big).
\end{align*}
Thus, $[\widehat{r}_{12},\; \widehat{r}_{13}]+[\widehat{r}_{12},\; \widehat{r}_{23}]
+[\widehat{r}_{13},\; \widehat{r}_{23}]=0$, i.e., $\widehat{r}$ is a solution of
the $\CYBE$ in $(A\otimes B, [-,-])$ if $r$ is a solution of the $\PLYBE$ in $(A, \diamond)$.

Second, we show that $\widehat{r}+\tau(\widehat{r})$ is $\ad_{A\otimes B}$-invariant.
For any $a\in A$, $e_{p}\in B$, $p=1, 2,\cdots, n$, we have
\begin{align*}
&\big(\id\otimes\ad_{A\otimes B}(a\otimes e_{p})
+\ad_{A\otimes B}(a\otimes e_{p})\otimes\id\big)(\widehat{r}+\tau(\widehat{r}))\\
=&\;\sum_{i, j}\Big((x_{i}\otimes e_{j})\otimes[a\otimes e_{p},\; y_{i}\otimes f_{j}]
+[a\otimes e_{p},\; x_{i}\otimes e_{j}]\otimes(y_{i}\otimes f_{j})\\[-5mm]
&\qquad\quad+(y_{i}\otimes f_{j})\otimes[a\otimes e_{p},\; x_{i}\otimes e_{j}]
+[a\otimes e_{p},\; y_{i}\otimes f_{j}]\otimes(x_{i}\otimes e_{j})\Big)\\
=&\;\sum_{i,j}\Big(\big(x_{i}\otimes(a\diamond y_{i})\big)\bullet
\big(e_{j}\otimes(e_{p}f_{j})\big)-\big(x_{i}\otimes(y_{i}\diamond a)\big)\bullet
\big(e_{j}\otimes(f_{j}e_{p})\big)\\[-5mm]
&\qquad+\big((a\diamond x_{i})\otimes y_{i}\big)\bullet
\big((e_{p}e_{j})\otimes f_{j}\big)-\big((x_{i}\diamond a)\otimes y_{i}\big)\bullet
\big((e_{j}e_{p})\otimes f_{j}\big)\\[-1mm]
&\qquad+\big(y_{i}\otimes(a\diamond x_{i})\big)\bullet
\big(f_{j}\otimes(e_{p}e_{j})\big)-\big(y_{i}\otimes(x_{i}\diamond a)\big)\bullet
\big(f_{j}\otimes(e_{j}e_{p})\big)\\[-1mm]
&\qquad+\big((a\diamond y_{i})\otimes x_{i}\big)\bullet
\big((e_{p}f_{j})\otimes e_{j}\big)-\big((y_{i}\diamond a)\otimes x_{i}\big)\bullet
\big((f_{j}e_{p})\otimes e_{j}\big).
\end{align*}
For given $s, t\in\{1, 2,\cdots, n\}$, since $\omega(-,-)$ is invariant, we have
\begin{align*}
&\qquad\qquad\qquad\qquad\omega\Big(e_{s}\otimes e_{t},\; \sum_{j}e_{j}\otimes
(f_{j}e_{p})\Big)=\omega(e_{s},\; e_{p}e_{t}-e_{t}e_{p}),\\[-3mm]
&\omega\Big(e_{s}\otimes e_{t},\; \sum_{j}(e_{p}e_{j})\otimes f_{j}\Big)
=\omega(e_{s},\; e_{p}e_{t}),\qquad
\omega\Big(e_{s}\otimes e_{t},\; \sum_{j}(e_{j}e_{p})\otimes f_{j}\Big)
=\omega(e_{s},\; e_{t}e_{p}).
\end{align*}
Thus, $\sum_{j}(e_{p}e_{j})\otimes f_{j}=\sum_{j}(e_{j}e_{p})\otimes f_{j}
+\sum_{j}e_{j}\otimes(f_{j}e_{p})$ since $\omega(-,-)$ is nondegenerate.
Similarly, we also have
\begin{align*}
&\qquad\sum_{j}(f_{j}e_{p})\otimes e_{j}
=-\sum_{j}(e_{j}e_{p})\otimes f_{j},\qquad\quad
\sum_{j}f_{j}\otimes(e_{j}e_{p})
=-\sum_{j}e_{j}\otimes(f_{j}e_{p}),\\[-2mm]
&\sum_{j}(e_{p}e_{j})\otimes f_{j}=-\sum_{j}(e_{p}f_{j})\otimes e_{j}
=\sum_{j}-f_{j}\otimes(e_{p}e_{j})
=\sum_{j}(e_{j}e_{p})\otimes f_{j}+\sum_{j}e_{j}\otimes(f_{j}e_{p}).
\end{align*}
Therefore, we obtain
\begin{align*}
&\big(\id\otimes\ad_{A\otimes B}(a\otimes e_{p})
+\ad_{A\otimes B}(a\otimes e_{p})\otimes\id\big)(\widehat{r}+\tau(\widehat{r}))\\
=&\;\sum_{i,j}\Big(\Big(x_{i}\otimes(a\diamond y_{i})+(a\diamond x_{i})\otimes y_{i}
-y_{i}\otimes(a\diamond x_{i})-(a\diamond y_{i})\otimes x_{i}\\[-5mm]
&\qquad\qquad\qquad\qquad-x_{i}\otimes(y_{i}\diamond a)+y_{i}\otimes(x_{i}\diamond a)
\Big)\bullet\big(e_{j}\otimes(f_{j}e_{p})\big)\\[-2mm]
&\quad+\Big(x_{i}\otimes(a\diamond y_{i})+(a\diamond x_{i})\otimes y_{i}
-y_{i}\otimes(a\diamond x_{i})-(a\diamond y_{i})\otimes x_{i}\\[-2mm]
&\qquad\qquad\qquad\qquad-(x_{i}\diamond a)\otimes y_{i}+(y_{i}\diamond a)\otimes x_{i}
\Big)\bullet\big((e_{p}e_{j})\otimes f_{j}\big)\Big).
\end{align*}
Since $r-\tau(r)$ is $(\fl_{A},\ad_{A})$-invariant, i.e.,
$\sum_{i}(a\diamond x_{i})\otimes y_{i}+x_{i}\otimes(a\diamond y_{i})
-x_{i}\otimes(y_{i}\diamond a)-(a\diamond y_{i})\otimes x_{i}-y_{i}\otimes(a\diamond x_{i})
+y_{i}\otimes(x_{i}\diamond a)=0$ for any $a\in A$, we get
$\big(\id\otimes\ad_{A\otimes B}(a\otimes e_{p})
+\ad_{A\otimes B}(a\otimes e_{p})\otimes\id\big)(\widehat{r}+\tau(\widehat{r}))=0$, i.e.,
$\widehat{r}+\tau(\widehat{r})$ is $\ad_{A\otimes B}$-invariant. The proof is finished.
\end{proof}

For the special case where $r$ is symmetric, we have the following corollary directly.

\begin{cor}[\cite{LZB}]\label{cor:sPLYBE-sCYBE}
Let $(A, \diamond)$ be a pre-Lie algebra, $(B, \cdot, \omega)$ be a quadratic perm
algebra, and $(A\otimes B, [-,-])$ be the induced Lie algebra.
If $r=\sum_{i}x_{i}\otimes y_{i}\in A\otimes A$ is a symmetric solution of the
$\PLYBE$ in $(A, \diamond)$, then
$$
\widehat{r}:=\sum_{i, j}(x_{i}\otimes e_{j})\otimes(y_{i}\otimes f_{j})
\in(A\otimes B)\otimes(A\otimes B)
$$
is a skew-symmetric solution of the $\CYBE$ in $(A\otimes B, [-,-])$.
\end{cor}

\begin{proof}
Since $r$ is symmetric, we get $r-\tau(r)=0$ is $(\fl_{A},\ad_{A})$-invariant,
and so that $\widehat{r}$ is a solution of the $\CYBE$ in $(A\otimes B, [-,-])$
by the first part of the proof of Proposition \ref{pro:PLYBE-CYBE}.  Moreover,
for any $s, t\in\{1, 2,\cdots, n\}$, we have
$$
\omega\Big(e_{s}\otimes e_{t},\; \sum_{j}e_{j}\otimes f_{j}\Big)
=\omega(e_{s}, e_{t})=-\omega(e_{t}, e_{s})
=-\omega\Big(e_{s}\otimes e_{t},\; \sum_{j}f_{j}\otimes e_{j}\Big).
$$
The nondegeneracy of $\omega(-,-)$ yields that $\sum_{j}e_{j}\otimes f_{j}
=-\sum_{j}f_{j}\otimes e_{j}$. Since $r$ is symmetric, we get $\widehat{r}$
is skew-symmetric. The proof is finished.
\end{proof}

We give the main conclusion of this section.

\begin{thm}\label{thm:indu-sLiebia}
Let $(A, \diamond, \vartheta)$ be a pre-Lie bialgebra, $(B, \cdot, \omega)$
be a quadratic perm algebra, and $(A\otimes B, [-,-], \delta)$ be the
induced Lie bialgebra by $(A, \diamond, \vartheta)$ and $(B, \cdot, \omega)$.
If $r\in A\otimes A$ such that $r-\tau(r)$ is $(\fl_{A},\ad_{A})$-invariant and
$\vartheta=\vartheta_{r}$ is defined by Eq. \eqref{preli-cobo}, then $(A\otimes B,
[-,-], \delta)=(A\otimes B, [-,-], \delta_{\widehat{r}})$ as Lie bialgebras,
where $\delta$ is defined by Eq. \eqref{pl-pm} and $\widehat{r}$ is defined
by Eq. \eqref{r-max}. Therefore, we have
\begin{enumerate}\itemsep=0pt
\item[$(i)$] $(A\otimes B, [-,-], \delta)$ is quasi-triangular if
     $(A, \diamond, \vartheta)$ is quasi-triangular;
\item[$(ii)$] $(A\otimes B, [-,-], \delta)$ is triangular if
     $(A, \diamond, \vartheta)$ is triangular;
\item[$(iii)$] $(A\otimes B, [-,-], \delta)$ is factorizable if
     $(A, \diamond, \vartheta)$ is factorizable.
\end{enumerate}
\end{thm}

\begin{proof}
Let $r=\sum_{i}x_{i}\otimes y_{i}$. First, for any $a\in A$ and $b\in B$, we have
\begin{align*}
\delta(a\otimes b)&=\sum_{i}\sum_{(b)}\Big(((a\diamond x_{i})\otimes y_{i})\bullet
(b_{(1)}\otimes b_{(2)})+(x_{i}\otimes(a\diamond y_{i}))\bullet(b_{(1)}\otimes b_{(2)})\\[-5mm]
&\qquad\qquad\quad-(x_{i}\otimes(y_{i}\diamond a))\bullet(b_{(1)}\otimes b_{(2)})
-(y_{i}\otimes(a\diamond x_{i}))\bullet(b_{(2)}\otimes b_{(1)})\\[-2mm]
&\qquad\qquad\quad-((a\diamond y_{i})\otimes x_{i})\bullet(b_{(2)}\otimes b_{(1)})
+((y_{i}\diamond a)\otimes x_{i})\bullet(b_{(2)}\otimes b_{(1)})\Big),
\end{align*}
where $\vartheta_{r}(a)=\big(\fl_{A}(a)\otimes\id+\id\otimes(\fl_{A}-\fr_{A})(a)\big)(r)
=\sum_{i}\big((a\diamond x_{i})\otimes y_{i}+x_{i}\otimes(a\diamond y_{i})
-x_{i}\otimes(y_{i}\diamond a)\big)$ and $\nu_{\omega}(b)=\sum_{(b)}b_{(1)}
\otimes b_{(2)}$. Suppose that $\{e_{1}, e_{2},\cdots, e_{n}\}$ is a basis
of $B$ and $\{f_{1}, f_{2},\cdots, f_{n}\}$ is the dual basis of $\{e_{1}, e_{2},
\cdots, e_{n}\}$ with respect to $\omega(-,-)$, i.e., $\omega(e_{i}, f_{j})=\delta_{ij}$,
where $\delta_{ij}$ is the Kronecker delta. Then
\begin{align*}
\delta_{\widehat{r}}(a\otimes b)&=(\id\otimes\ad_{A\otimes B}(a\otimes b)
+\ad_{A\otimes B}(a\otimes b)\otimes\id)(\widehat{r})\\
&=\sum_{i, j}\Big((x_{i}\otimes(a\diamond y_{i}))\bullet(e_{j}\otimes(bf_{j}))
-(x_{i}\otimes(y_{i}\diamond a))\bullet(e_{j}\otimes(f_{j}b))\\[-5mm]
&\qquad\quad+((a\diamond x_{i})\otimes y_{i})\bullet((be_{j})\otimes f_{j})
-((x_{i}\diamond a)\otimes y_{i})\bullet((e_{j}b)\otimes f_{j})\Big),
\end{align*}
where $\widehat{r}=\sum_{i, j}(x_{i}\otimes e_{j})\otimes(y_{i}\otimes f_{j})$.
For two basis elements $e_{s}, e_{t}\in B$, since
$$
\omega\Big(e_{s}\otimes e_{t},\; \sum_{(b)}b_{(1)}\otimes b_{(2)}\Big)
=\omega(e_{s},\; be_{t}-e_{t}b)
=\omega\Big(e_{s}\otimes e_{t},\; \sum_{j}e_{j}\otimes(f_{j}b)\Big),
$$
and $\omega(-,-)$ is nondegenerate, we get $\sum_{(b)}b_{(1)}\otimes b_{(2)}
=\sum_{j}e_{j}\otimes(f_{j}b)$. Similarly, we also have
$\sum_{(b)}b_{(2)}\otimes b_{(1)}=-\sum_{j}(e_{j}b)\otimes f_{j}$ and
$\sum_{j}(be_{j})\otimes f_{j}=\sum_{j}e_{j}\otimes (bf_{j})
=\sum_{(b)}\big(b_{(1)}\otimes b_{(2)}-b_{(2)}\otimes b_{(1)}\big)$.
Thus,
\begin{align*}
\delta(a\otimes b)-\delta_{\widehat{r}}(a\otimes b)
=&\;\sum_{i}\sum_{(b)}\Big((y_{i}\diamond a)\otimes x_{i}
+x_{i}\otimes(a\diamond y_{i})-y_{i}\otimes(a\diamond x_{i})\\[-5mm]
&\qquad\qquad
-(a\diamond y_{i})\otimes x_{i}+(a\diamond x_{i})\otimes y_{i}
-(x_{i}\diamond a)\otimes y_{i}\Big)\bullet(b_{(2)}\otimes b_{(1)}).
\end{align*}
If $r-\tau(r)$ is $(\fl_{A},\ad_{A})$-invariant, i.e.,
$\sum_{i}(a\diamond x_{i})\otimes y_{i}+x_{i}\otimes(a\diamond y_{i})
-x_{i}\otimes(y_{i}\diamond a)-(a\diamond y_{i})\otimes x_{i}-y_{i}\otimes(a\diamond x_{i})
+y_{i}\otimes(x_{i}\diamond a)=0$ for any $a\in A$, we get $\delta(a\otimes b)
=\delta_{\widehat{r}}(a\otimes b)$ for any $a\in A$ and $b\in B$. Therefore,
$(A\otimes B, [-,-], \delta)=(A\otimes B, [-,-], \delta_{\widehat{r}})$ as Lie bialgebras.

Second, $(i)$ and $(ii)$ follow from Proposition \ref{pro:PLYBE-CYBE} and Corollary
\ref{cor:sPLYBE-sCYBE}. Finally, if $(A, \diamond, \vartheta_{r})$ is factorizable,
i.e., $r$ is a solution of $\PLYBE$ in $(A, \diamond)$, $r-\tau(r)$ is
$(\fl_{A},\ad_{A})$-invariant, and the map $\mathcal{I}=r^{\sharp}-\tau(r)^{\sharp}:
A^{\ast}\rightarrow A$ is an isomorphism of vector spaces.
We need show that $\widehat{\mathcal{I}}=\widehat{r}^{\sharp}
+\tau(\widehat{r})^{\sharp}: (A\otimes B)^{\ast}\rightarrow A\otimes B$ is an
isomorphism of vector spaces. Denote $\kappa:=\sum_{j}e_{j}\otimes
f_{j}\in B\otimes B$. Define $\kappa^{\sharp}: B^{\ast}\rightarrow B$ by
$\langle\eta_{2},\; \kappa^{\sharp}(\eta_{1})\rangle=\langle\eta_{1}
\otimes\eta_{2},\; \kappa\rangle$, for any $\eta_{1}, \eta_{2}\in B^{\ast}$.
Then, one can check that $\kappa^{\sharp}$ is
a linear isomorphism and $\langle\eta_{2},\; \kappa^{\sharp}(\eta_{1})\rangle
=\langle\eta_{1}\otimes\eta_{2},\; \kappa\rangle
=-\langle\eta_{2}\otimes\eta_{1},\; \kappa\rangle
=-\langle\eta_{1},\; \kappa^{\sharp}(\eta_{2})\rangle$.
Therefore, for any $\xi_{1}, \xi_{2}\in A^{\ast}$ and $\eta_{1}, \eta_{2}\in B^{\ast}$,
\begin{align*}
\langle\xi_{2}\otimes\eta_{2},\; \tau(\widehat{r})^{\sharp}(\xi_{1}\otimes\eta_{1})\rangle
&=\sum_{i,j}\langle(\xi_{2}\otimes\eta_{2})\otimes(\xi_{1}\otimes\eta_{1}),\ \
(x_{i}\otimes e_{j})\otimes(y_{i}\otimes f_{j})\rangle\\[-2mm]
&=\Big(\sum_{i}\langle\xi_{2}, x_{i}\rangle\langle\xi_{1}, y_{i}\rangle\Big)
\Big(\sum_{j}\langle\eta_{1}, f_{j}\rangle\langle\eta_{2}, e_{j}\rangle\Big)\\[-2mm]
&=-\langle\xi_{2},\; \tau(r)^{\sharp}(\xi_{1})\rangle
\langle\eta_{1},\; \kappa^{\sharp}(\eta_{2})\rangle\\
&=-\langle\xi_{2}\otimes\eta_{2},\; \tau(r)^{\sharp}(\xi_{1})\otimes
\kappa^{\sharp}(\eta_{1})\rangle.
\end{align*}
That is, $\tau(\widehat{r})^{\sharp}=-\tau(r)^{\sharp}\otimes\kappa^{\sharp}$.
Similarly, we have $\widehat{r}^{\sharp}=r^{\sharp}\otimes\kappa^{\sharp}$. Thus,
$\widehat{\mathcal{I}}=\widehat{r}^{\sharp}+\widehat{r}^{\sharp}
=(r^{\sharp}-r^{\sharp})\otimes\kappa^{\sharp}=\mathcal{I}\otimes\kappa^{\sharp}$
is an isomorphism of vector spaces. Thus, $(A\otimes B, [-,-], \delta)$ is a factorizable
Lie bialgebra if $(A, \diamond, \vartheta)$ is a factorizable pre-Lie bialgebra.
The proof is finished.
\end{proof}

\begin{rmk}\label{rmk:indu-Liebia}
The conclusion $(ii)$ in Theorem \ref{thm:indu-sLiebia} is a special case of $(i)$,
which has been considered in \cite{LZB}. Moreover, the notions of completed Lie
bialgebras and completed solutions of the $\CYBE$ are introduced in \cite{LZB}.
In Theorem \ref{thm:indu-sLiebia}, if we replace the quadratic perm algebra
$(B, \cdot, \omega)$ with a quadratic $\bz$-graded perm algebra $(B, \cdot, \varpi)$
(see Definition \ref{def:zgrad-alg}), then $(A\otimes B, [-,-], \delta)
=(A\otimes B, [-,-], \delta_{\widehat{r}})$ as completed Lie bialgebras.
Similar to the proof of Theorem \ref{thm:indu-sLiebia}, one can show that
the completed Lie bialgebra $(A\otimes B, [-,-], \delta)$ is quasi-triangular
if $(A, \diamond, \vartheta)$ is quasi-triangular.
\end{rmk}

Let $(A, \diamond)$ be a pre-Lie algebra and $r\in A\otimes A$. If $r$ is a solution of
the $\PLYBE$ in $(A, \diamond)$ such that $r-\tau(r)$ is $(\fl_{A},\ad_{A})$-invariant,
then the Theorem \ref{thm:indu-sLiebia} also gives the following commutative diagram:

$$
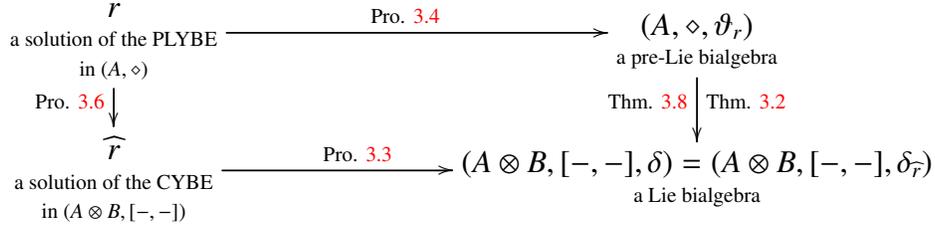
\xymatrix@C=3cm@R=0.5cm{
\txt{$r$ \\ {\tiny a solution of the $\PLYBE$}\\ {\tiny in $(A, \diamond)$}}
\ar[d]_{{\rm Pro.}~\ref{pro:PLYBE-CYBE}}\ar[r]^{{\rm Pro.}~\ref{pro:spec-plbia}} &
\txt{$(A, \diamond, \vartheta_{r})$ \\ {\tiny a pre-Lie bialgebra}}
\ar[d]^{{\rm Thm.}~\ref{thm:pre-lie}}_{{\rm Thm.}~\ref{thm:indu-sLiebia}} \\
\txt{$\widehat{r}$ \\ {\tiny a solution of the $\CYBE$}\\ {\tiny in $(A\otimes B, [-,-])$}}
\ar[r]^{{\rm Pro.}~\ref{pro:splie-bia}\qquad\quad} &
\txt{$(A\otimes B, [-,-], \delta)=(A\otimes B, [-,-], \delta_{\widehat{r}})$ \\
{\tiny a Lie bialgebra}}}
$$

The $\mathcal{O}$-operator is considered to be the operator form of the solution
of the Yang-Baxter equation. We now consider the relationship between the
$\mathcal{O}$-operators of pre-Lie algebra and the induced Lie algebra.
Let $(\g, [-,-])$ be a Lie algebra and $(V, \rho)$ be a representation of $(\g, [-,-])$.
Recall that a linear map $P: V\rightarrow\g$ is called an {\bf $\mathcal{O}$-operator
of $(\g, [-,-])$ associated to $(V, \rho)$} if for any $v_{1}, v_{2}\in V$,
$$
[P(v_{1}), P(v_{2})]=P\big(\rho(P(v_{1}))(v_{2})-\rho(P(v_{2}))(v_{1})\big).
$$
The notion of $\mathcal{O}$-operator of Lie algebras was introduced by
Kupershmidt in \cite{Kup}, which is considered to be the operator form of
solution of $\CYBE$ in $(\g, [-,-])$.

\begin{pro}[\cite{Kup,CP}]\label{pro:o-lie}
Let $(\g, [-,-])$ be a Lie algebra, $r\in\g\otimes\g$ be skew-symmetric.
Then $r$ is a solution of the $\CYBE$ in $(\g, [-,-])$ if and only if $r^{\sharp}$
is an $\mathcal{O}$-operator of $(\g, [-,-])$ associated to the coadjoint
representation $(\g^{\ast}, -\ad_{\g}^{\ast})$, where $\ad_{\g}(g_{1})(g_{2})
=[g_{1}, g_{12}]$ for any $g_{1}, g_{2}\in\g$.
\end{pro}

The $\mathcal{O}$-operator of pre-Lie algebras was considered in \cite{Bai1}.
Recall that an {\bf $\mathcal{O}$-operator of a pre-Lie algebra $(A, \diamond)$
associated to a bimodule $(V, \kl, \kr)$} is a linear map $P: V\rightarrow A$ such that
$$
P(v_{1})\diamond P(v_{2})=P\big(\kl(P(v_{1}))(v_{2})+\kr(P(v_{2}))(v_{1})\big),
$$
for any $v_{1}, v_{2}\in V$.

\begin{pro}[\cite{Bai1}]\label{pro:o-PL}
Let $(A, \diamond)$ be a pre-Lie algebra, $r\in A\otimes A$ be symmetric.
Then $r$ is a solution of the $\PLYBE$ in $(A, \diamond)$ if and only if $r^{\sharp}$
is an $\mathcal{O}$-operator of $(A, \diamond)$ associated to the coregular
bimodule $(A^{\ast}, \fr_{A}^{\ast}-\fl_{A}^{\ast}, \fr_{A}^{\ast})$.
\end{pro}

By the proof of Theorem \ref{thm:indu-sLiebia}, we also have the following commutative diagram:
$$
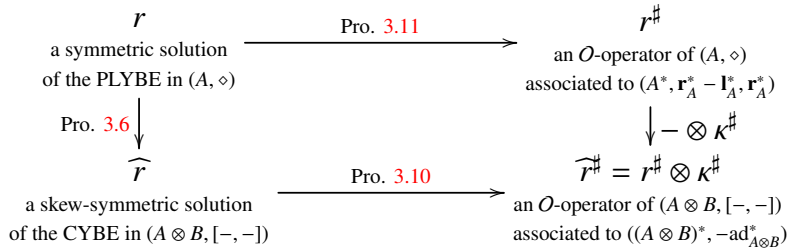
\xymatrix@C=3cm@R=0.5cm{
\txt{$r$ \\ {\tiny a symmetric solution} \\ {\tiny of the $\PLYBE$ in $(A, \diamond)$}}
\ar[d]_-{{\rm Pro.}~\ref{pro:PLYBE-CYBE}}\ar[r]^-{{\rm Pro.}~\ref{pro:o-PL}} &
\txt{$r^{\sharp}$\\ {\tiny an $\mathcal{O}$-operator of $(A, \diamond)$} \\
{\tiny associated to $(A^{\ast}, \fr_{A}^{\ast}-\fl_{A}^{\ast}, \fr_{A}^{\ast})$}}
\ar[d]^-{\mbox{$-\otimes\kappa^{\sharp}$}} \\
\txt{$\widehat{r}$ \\ {\tiny a skew-symmetric solution} \\ {\tiny of the $\CYBE$ in
$(A\otimes B, [-,-])$}} \ar[r]^-{{\rm Pro.}~\ref{pro:o-lie}}
& \txt{$\widehat{r}^{\sharp}=r^{\sharp}\otimes\kappa^{\sharp}$ \\
{\tiny an $\mathcal{O}$-operator of $(A\otimes B, [-,-])$ } \\
{\tiny associated to $((A\otimes B)^{\ast}, -\ad^{\ast}_{A\otimes B})$}}}
$$

\begin{ex}\label{ex:prelie-indlie}
Define a bilinear operator $\diamond$ on 2-dimensional vector space
$A=\Bbbk\{e_{1}, e_{2}\}$ by $e_{1}\diamond e_{1}=e_{1}$, $e_{1}\diamond e_{2}=-e_{2}$
and $e_{2}\diamond e_{1}=e_{2}\diamond e_{2}=0$. Then $(A, \diamond)$ is a pre-Lie
algebra. One can check that $r=e_{1}\otimes e_{1}$ is a symmetric solution of the
$\PLYBE$ in $(A, \diamond)$. Thus, by Proposition \ref{pro:spec-plbia}, $r$ induces a
linear map $\vartheta: A\rightarrow A\otimes A$, $\vartheta(e_{1})=e_{1}\otimes e_{1}$,
$\vartheta(e_{2})=e_{1}\otimes e_{2}$, such that $(A, \diamond, \vartheta)$
is a triangular pre-Lie bialgebra. Let $(B=\Bbbk\{x_{1}, x_{2}\}, \cdot, \omega)$ be
a quadratic perm algebra, where $x_{2}x_{1}=x_{1}$, $x_{2}x_{2}=x_{2}$, $x_{1}x_{1}=0
=x_{1}x_{2}$ and $\omega(x_{1}, x_{2})=1$. Then, by Theorem \ref{thm:pre-lie},
we get a Lie bialgebra $(A\otimes B, [-,-], \delta)$. If we denote
$y_{1}:=e_{1}\otimes x_{1}$, $y_{2}:=e_{1}\otimes x_{2}$, $y_{3}:=e_{2}\otimes x_{1}$
and $y_{4}:=e_{2}\otimes x_{2}$. Then the nonzero brackets and
coproducts in $(A\otimes B, [-,-], \delta)$ are given by
\begin{align*}
&\;[y_{1}, y_{2}]=-y_{1},\qquad\qquad\qquad\; [y_{2}, y_{3}]=-y_{3},\qquad\qquad\qquad\;
[y_{2}, y_{4}]=y_{4},\\
&\delta(y_{2})=y_{2}\otimes y_{1}-y_{1}\otimes y_{2},\qquad
\delta(y_{3})=y_{3}\otimes y_{1}-y_{1}\otimes y_{3},\qquad
\delta(y_{4})=y_{4}\otimes y_{1}-y_{1}\otimes y_{4}.
\end{align*}
On the other hand, by Theorem \ref{thm:indu-sLiebia},
we get a skew-symmetric solution $\widehat{r}=y_{1}\otimes y_{2}-y_{2}\otimes y_{1}$
of the $\CYBE$ in Lie algebra $(A\otimes B, [-,-])$. One can check that the triangular
Lie bialgebra associated with $\widehat{r}$ is just the Lie bialgebra $(A\otimes B,
[-,-], \delta)$ given above.

Moreover, if we denote the dual basis of $\{e_{1}, e_{2}\}$ by $\{\xi_{1}, \xi_{2}\}$
and denote the dual basis of $\{x_{1}, x_{2}\}$ by $\{\eta_{1}, \eta_{2}\}$, then
$\{\xi_{1}\otimes\eta_{1}, \xi_{1}\otimes\eta_{2}, \xi_{2}\otimes\eta_{1},
\xi_{2}\otimes\eta_{2}\}$ is a dual basis of $\{y_{1}, y_{2}, y_{3}, y_{4}\}$.
By Propositions \ref{pro:o-PL} and \ref{pro:o-lie}, we get an $\mathcal{O}$-operator
$r^{\sharp}: A^{\ast}\rightarrow A$, $r^{\sharp}(\xi_{1})=e_{1}$, $r^{\sharp}(\xi_{2})=0$,
and an $\mathcal{O}$-operator $\widehat{r}^{\sharp}: (A\otimes B)^{\ast}\rightarrow
A\otimes B$, $\widehat{r}^{\sharp}(\xi_{1}\otimes\eta_{1})=y_{2}$,
$r^{\sharp}(\xi_{1}\otimes\eta_{1})=-y_{1}$ and others are all zero, respectively.
Note that the linear map $\kappa^{\sharp}: B^{\ast}\rightarrow B$ is given by
$\kappa^{\sharp}(\eta_{1})=x_{2}$ and $\kappa^{\sharp}(\eta_{2})=-x_{1}$.
We get $\widehat{r}^{\sharp}=r^{\sharp}\otimes\kappa^{\sharp}$.
\end{ex}

\section{Infinite-dimensional ASI bialgebras vis affinization of dendriform $\md$-bialgebras}
\label{sec:dend-ASI}
In this section, we recall the notion of a quadratic $\bz$-graded perm algebra, as
a $\bz$-graded perm algebra equipped with an invariant bilinear form. We show that
the tensor product of a finite-dimensional dendriform $\md$-bialgebra and a quadratic
$\bz$-graded perm algebra can be naturally endowed with a completed ASI bialgebra.
The converse of this result also holds when the quadratic $\bz$-graded perm algebra
is special, giving the desired characterization of the dendriform $\md$-bialgebra
that its affinization is a completed ASI bialgebra.

\subsection{Affinization of dendriform algebras and dendriform coalgebras}\label{subsec:aff}
To provide the affinization characterization of dendriform algebra, we first recall
the notion of $\bz$-graded perm algebra.

\begin{defi}\label{def:zgrad-alg}
A {\bf $\bz$-graded associative algebra} (resp. {\bf $\bz$-graded perm algebra})
is an associative algebra $(A, \ast)$ (resp. a perm algebra $(B, \cdot)$) with a
linear decomposition $A=\oplus_{i\in\bz}A_{i}$ (resp. $B=\oplus_{i\in\bz}B_{i}$) such that
each $A_{i}$ (resp. $B_{i}$) is finite-dimensional and $A_{i}\ast A_{j}\subseteq A_{i+j}$
(resp. $B_{i}\cdot B_{j}\subseteq B_{i+j}$) for all $i, j\in\bz$.
\end{defi}

\begin{ex}[\cite{LZB}]\label{ex:grperm}
Let $B=\{f_{1}\partial_{1}+f_{2}\partial_{2}\mid f_{1}, f_{2}\in\Bbbk[x_{1}^{\pm},
x_{2}^{\pm}]\}$ and define a binary operation $\cdot: B\otimes B\rightarrow B$ by
\begin{align*}
(x_{1}^{i_{1}}x_{2}^{i_{2}}\partial_{s})\cdot(x_{1}^{j_{1}}x_{2}^{j_{2}}\partial_{t})
:=\delta_{s,1}x_{1}^{i_{1}+j_{1}+1}x_{2}^{i_{2}+j_{2}}\partial_{t}
+\delta_{s,2}x_{1}^{i_{1}+j_{1}}x_{2}^{i_{2}+j_{2}+1}\partial_{t},
\end{align*}
for any $i_{1}, i_{2}, j_{1}, j_{2}\in\bz$ and $t\in\{1, 2\}$.
Then $(B, \cdot)$ is a $\bz$-graded perm algebra with the linear decomposition
$B=\oplus_{i\in\bz}B_{i}$, where
$$
B_{i}=\Big\{\sum_{k=1}^{2}f_{k}\partial_{k}\mid f_{k}\text{ is a homogeneous
polynomial with }\deg(f_{k})=i-1,\; k=1,2\Big\}
$$
for all $i\in\bz$.
\end{ex}

By Proposition \ref{pro:perm-dend}, we have the following proposition.

\begin{pro}\label{pro:aff-dalg}
Let $(D, \prec, \succ)$ be a finite-dimensional dendriform algebra and
$(B=\oplus_{i\in\bz}B_{i}, \cdot)$ be a $\bz$-graded perm algebra.
Define a binary operation on $D\otimes B$ by
$$
(d_{1}\otimes b_{1})\ast(d_{2}\otimes b_{2})=(d_{1}\succ d_{2})\otimes(b_{1}b_{2})
+(d_{1}\prec d_{2})\otimes(b_{2}b_{1}),
$$
for any $d_{1}, d_{2}\in D$ and $b_{1}, b_{2}\in B$. Then $(D\otimes B, \ast)$ is a
$\bz$-graded associative algebra, which is called an {\bf affine associative algebra}
from $(D, \prec, \succ)$ by $(B=\oplus_{i\in\bz}B_{i}, \cdot)$.
Moreover, if $(B=\oplus_{i\in\bz}B_{i}, \cdot)$ is the $\bz$-graded perm algebra given in
Example \ref{ex:grperm}, then $(D\otimes B, \ast)$ is a $\bz$-graded associative
algebra if and only if $(D, \prec, \succ)$ is a dendriform algebra.
\end{pro}

\begin{proof}
By Proposition \ref{pro:perm-dend}, $(D\otimes B, \ast)$ is an associative algebra.
Since $(B=\oplus_{i\in\bz}B_{i}, \cdot)$ is $\bz$-graded, $(D\otimes B, \ast)$ is a
$\bz$-graded associative algebra. If $(B=\oplus_{i\in\bz}B_{i}, \cdot)$ is the
$\bz$-graded perm algebra given in Example \ref{ex:grperm}, then we have
\begin{align*}
\Big((d_{1}\otimes\partial_{1})\ast(d_{2}\otimes\partial_{2})\Big)
\ast(d_{3}\otimes\partial_{1})&=\Big((d_{1}\succ d_{2})\otimes x_{1}\partial_{2})
+(d_{1}\prec d_{2})\otimes x_{2}\partial_{1}\Big)\ast(d_{3}\otimes\partial_{1})\\
&=((d_{1}\succ d_{2})\succ d_{3})\otimes x_{1}x_{2}\partial_{1}
+((d_{1}\succ d_{2})\prec d_{3})\otimes x_{1}^{2}\partial_{2}\\[-1mm]
&\quad+((d_{1}\prec d_{2})\succ d_{3})\otimes x_{1}x_{2}\partial_{1}
+((d_{1}\prec d_{2})\prec d_{3})\otimes x_{1}x_{2}\partial_{1},
\end{align*}
and
\begin{align*}
(d_{1}\otimes\partial_{1})\ast\Big((d_{2}\otimes\partial_{2})
\ast(d_{3}\otimes\partial_{1})\Big)&=(d_{1}\otimes\partial_{1})\ast
\Big((d_{2}\succ d_{3})\otimes x_{2}\partial_{1}
+(d_{2}\prec d_{3})\otimes x_{1}\partial_{2}\Big)\\
&=(d_{1}\succ(d_{2}\succ d_{3}))\otimes x_{1}x_{2}\partial_{1}
+(d_{1}\prec(d_{2}\succ d_{3}))\otimes x_{1}x_{2}\partial_{1}\\[-1mm]
&\quad+(d_{1}\succ(d_{2}\prec d_{3}))\otimes x_{1}^{2}\partial_{2}
+(d_{1}\prec(d_{2}\prec d_{3}))\otimes x_{1}x_{2}\partial_{1}.
\end{align*}
Comparing the coefficients of $x_{1}^{2}\partial_{2}$, we get $(d_{1}\succ d_{2})\prec d_{3}
=d_{1}\succ(d_{2}\prec d_{3})$. Similarly, comparing the coefficients of
$x_{1}^{2}\partial_{2}$ in equations $((d_{1}\otimes\partial_{1})\ast
(d_{2}\otimes\partial_{1}))\ast(d_{3}\otimes\partial_{2})=(d_{1}\otimes\partial_{1})
\ast((d_{2}\otimes\partial_{1})\ast(d_{3}\otimes\partial_{2}))$ and
$((d_{1}\otimes\partial_{2})\ast(d_{2}\otimes\partial_{1}))\ast(d_{3}\otimes\partial_{1})
=(d_{1}\otimes\partial_{2})\ast((d_{2}\otimes\partial_{1})
\ast(d_{3}\otimes\partial_{1}))$, we get $(d_{1}\prec d_{2})\prec d_{3}
=d_{1}\prec(d_{2}\prec d_{3})+d_{1}\prec(d_{2}\succ d_{3})$ and $d_{1}\succ(d_{2}\succ d_{3})
=(d_{1}\prec d_{2})\succ d_{3}+(d_{1}\succ d_{2})\succ d_{3}$ respectively.
Thus, $(D, \prec, \succ)$ is a dendriform algebra. The proof is finished.
\end{proof}

Recall that a {\bf dendriform coalgebra} is a triple $(D, \theta_{\prec}, \theta_{\succ})$,
where $D$ is a vector space and $\theta_{\prec}, \theta_{\succ}: D\rightarrow D\otimes D$
are linear maps such that the following conditions hold:
\begin{align*}
&\qquad\qquad\qquad\quad(\theta_{\prec}\otimes\id)\theta_{\prec}
=(\id\otimes\theta_{\prec})\theta_{\prec}+(\id\otimes\theta_{\succ})\theta_{\prec},\\
&(\theta_{\succ}\otimes\id)\theta_{\prec}=(\id\otimes\theta_{\prec})\theta_{\succ},\qquad\qquad
(\id\otimes\theta_{\succ})\theta_{\succ}=(\theta_{\prec}\otimes\id)\theta_{\succ}
+(\theta_{\succ}\otimes\id)\theta_{\succ}.
\end{align*}
To carry out the dendriform coalgebra affinization, we need to extend the codomain
of the comultiplications $\theta_{\prec}, \theta_{\succ}$ to allow infinite sums.
Let $U=\oplus_{i\in\bz}U_{i}$ and $V=\oplus_{j\in\bz}V_{j}$ be $\bz$-graded vector spaces.
We call the {\bf completed tensor product} of $U$ and $V$ to be the vector space
$$
U\,\hat{\otimes}\,V:=\prod_{i,j\in\bz}U_{i}\otimes V_{j}.
$$
If $U$ and $V$ are finite-dimensional, then $U\,\hat{\otimes}\,V$ is just the usual
tensor product $U\otimes V$. In general, an element in $U\,\hat{\otimes}\,V$ is an
infinite formal sum $\sum_{i,j\in\bz}X_{ij} $ with $X_{ij}\in U_{i}\otimes V_{j}$.
So $X_{ij}=\sum_{\alpha} u_{i, \alpha}\otimes v_{j, \alpha}$ for pure tensors
$u_{i, \alpha}\otimes v_{j, \alpha}\in U_{i}\otimes V_{j}$ with $\alpha$ in a finite
index set. Thus a general term of $U\,\hat{\otimes}\,V$ is a possibly infinite sum
$\sum_{i,j,\alpha}u_{i\alpha}\otimes v_{j\alpha}$, where $i, j\in\bz$ and $\alpha$ is
in a finite index set (which might depend on $i, j$). With these notations, for linear
maps $f: U\rightarrow U'$ and $g: V\rightarrow V'$, define
$$
f\,\hat{\otimes}\,g: U\,\hat{\otimes}\,V\rightarrow U'\,\hat{\otimes}\,V',
\qquad \sum_{i,j,\alpha}u_{i,\alpha}\otimes v_{j, \alpha}\mapsto
\sum_{i,j,\alpha} f(u_{i, \alpha})\otimes g(v_{j, \alpha}).
$$
Also the twist map $\tau$ has its completion
$\hat{\tau}: V\,\hat{\otimes}\,V\rightarrow V\,\hat{\otimes}\,V$,
$\sum_{i,j,\alpha}u_{i, \alpha}\otimes v_{j, \alpha}\mapsto
\sum_{i,j,\alpha}v_{j, \alpha}\otimes u_{i, \alpha}$.
Moreover, we define a (completed) comultiplication to be a linear map
$\Delta: V\rightarrow V\,\hat{\otimes}\,V$,
$\Delta(v):=\sum_{i, j, \alpha}v_{1, i, \alpha}\otimes v_{2, j, \alpha}$.
Then we have the well-defined map
$$
(\Delta\,\hat{\otimes}\,\id)\Delta(v)=(\Delta\,\hat{\otimes}\,\id)
\Big(\sum_{i,j,\alpha}v_{1, i, \alpha}\otimes v_{2, j, \alpha}\Big)
:=\sum_{i,j,\alpha}\Delta(v_{1, i, \alpha})\otimes v_{2, j, \alpha}
\in V\,\hat{\otimes}\,V\,\hat{\otimes}\,V.
$$

\begin{defi}\label{def:ali-coa}
\begin{enumerate}
\item[$(i)$] A {\bf completed coassociative coalgebra} is a pair $(A, \Delta)$ where
    $A=\oplus_{i\in\bz}A_{i}$ is a $\bz$-graded vector space and
    $\Delta: A\rightarrow A\,\hat{\otimes}\, A$ is a linear map satisfying
    $$
    (\Delta\,\hat{\otimes}\,\id)\Delta=(\id\,\hat{\otimes}\,\Delta)\Delta.
    $$
\item[$(ii)$] A {\bf completed perm coalgebra} is a pair $(B, \nu)$, where
    $B=\oplus_{i\in\bz}B_{i}$ is a $\bz$-graded vector space and $\nu: B\rightarrow
    B\,\hat{\otimes}\,B$ is a linear map satisfying
    $$
    (\nu\,\hat{\otimes}\,\id)\nu=(\id\,\hat{\otimes}\,\nu)\nu
    =(\tau\,\hat{\otimes}\,\id)(\nu\,\hat{\otimes}\,\id)\nu.
    $$
\end{enumerate}
\end{defi}

\begin{ex}[\cite{LZB}]\label{ex:grpermco}
Consider the $\bz$-graded vector space $B=\{f_{1}\partial_{1}+f_{2}\partial_{2}\mid
f_{1}, f_{2}\in\Bbbk[x_{1}^{\pm}, x_{2}^{\pm}]\}=\oplus_{i\in\bz}B_{i}$ given in
Example \ref{ex:grperm}. Define a linear map $\nu: B\rightarrow B\,\hat{\otimes}\,B$ by
\begin{align*}
\nu(x_{1}^{m}x_{2}^{n}\partial_{1})&=\sum_{i_{1}, i_{2}\in\bz}
\Big(x_{1}^{i_{1}}x_{2}^{i_{2}}\partial_{1}\otimes x_{1}^{m-i_{1}}
x_{2}^{n-i_{2}+1}\partial_{1} -x_{1}^{i_{1}}x_{2}^{i_{2}}\partial_{2}
\otimes x_{1}^{m-i_{1}+1}x_{2}^{n-i_{2}}\partial_{1}\Big), \\[-2mm]
\nu(x_{1}^{m}x_{2}^{n}\partial_{2})&=\sum_{i_{1}, i_{2}\in\bz}
\Big(x_{1}^{i_{1}}x_{2}^{i_{2}}\partial_{1}\otimes x_{1}^{m-i_{1}}
x_{2}^{n-i_{2}+1}\partial_{2}-x_{1}^{i_{1}}x_{2}^{i_{2}}\partial_{2}
\otimes x_{1}^{m-i_{1}+1}x_{2}^{n-i_{2}}\partial_{2}\Big),
\end{align*}
for any $m, n \in\bz$. Then $(B=\oplus_{i\in\bz}B_{i}, \nu)$ is a completed perm coalgebra.
\end{ex}

Now, we consider the dual version of the dendriform algebra affinization,
for dendriform coalgebras. We give the dual version of Proposition \ref{pro:aff-dalg}.

\begin{pro}\label{pro:dco-coass}
Let $(D, \theta_{\prec}, \theta_{\succ})$ be a finite-dimensional dendriform coalgebra
and $(B=\oplus_{i\in\bz}B_{i}, \nu)$ be a completed perm coalgebra. Define a linear
map $\Delta: D\otimes B\rightarrow(D\otimes B)\,\hat{\otimes}\,(D\otimes B)$ by
\begin{align*}
\Delta(d\otimes b)&=\theta_{\succ}(d)\bullet\nu(b)
+\theta_{\prec}(d)\bullet\,\hat{\tau}(\nu(b))\\
&:=\sum_{[d]}\sum_{i,j,\alpha}(d_{[1]}\otimes b_{1,i,\alpha})
\otimes(d_{[2]}\otimes b_{2,j,\alpha})+\sum_{(d)}\sum_{i,j,\alpha}(d_{(1)}\otimes
b_{2,j,\alpha})\otimes(d_{(2)}\otimes b_{1,i,\alpha}),
\end{align*}
for any $d\in D$ and $b\in B$, where $\theta_{\prec}(d)=\sum_{(d)}d_{(1)}\otimes d_{(2)}$,
$\theta_{\succ}(d)=\sum_{[d]}d_{[1]}\otimes d_{[2]}$ in the Sweedler notation
and $\nu(b)=\sum_{i,j,\alpha}b_{1,i,\alpha}\otimes b_{2,j,\alpha}$.
Then $(D\otimes B, \Delta)$ is a completed coassociative coalgebra.
Moreover, if $(B=\oplus_{i\in\bz}B_{i}, \nu)$ is the completed perm coalgebra given in
Example \ref{ex:grpermco}, then $(D\otimes B, \Delta)$ is a completed coassociative
coalgebra if and only if $(D, \theta_{\prec}, \theta_{\succ})$ is a dendriform coalgebra.
\end{pro}

\begin{proof}
For any $\sum_{l}d'_{l}\otimes d''_{l}\otimes d'''_{l}\in D\otimes D\otimes D$ and
$\sum_{i,j,k,\alpha}b'_{i,\alpha}\otimes b''_{j,\alpha}
\otimes b'''_{k,\alpha}\in B\,\hat{\otimes}\,B\,\hat{\otimes}\,B$, we denote
$$
\Big(\sum_{l}d'_{l}\otimes d''_{l}\otimes d'''_{l}\Big)\bullet
\Big(\sum_{i,j,k,\alpha}b'_{i,\alpha}\otimes b''_{j,\alpha}\otimes b'''_{k,\alpha}\Big)
=\sum_{l}\sum_{i,j,k,\alpha}(d'_{l}\otimes b'_{i,\alpha})\otimes
(d''_{l}\otimes b''_{j,\alpha})\otimes(d'''_{l}\otimes b'''_{k,\alpha}).
$$
Then, by using the above notations, since $(B, \nu)$ is a perm coalgebra
and $(D, \theta_{\prec}, \theta_{\succ})$ is a dendriform coalgebra,
for any $d\otimes b\in D\otimes B$, we have
\begin{align*}
&\qquad\qquad\qquad\qquad\quad\;(\theta_{\prec}\otimes\id)(\theta_{\prec}(d))\bullet
(\hat{\tau}\,\hat{\otimes}\,\id)((\id\,\hat{\otimes}\,\hat{\tau})
((\id\,\hat{\otimes}\,\nu)(\nu(b))))\\[-1mm]
&\qquad\qquad\qquad\qquad=(\id\otimes\theta_{\prec})(\theta_{\prec}(d))\bullet
(\hat{\tau}\,\hat{\otimes}\,\id)((\id\,\hat{\otimes}\,\hat{\tau})
((\id\,\hat{\otimes}\,\nu)(\nu(b))))\\[-1mm]
&\qquad\qquad\qquad\qquad\qquad+(\id\otimes\theta_{\succ})(\theta_{\prec}(d))\bullet
(\hat{\tau}\,\hat{\otimes}\,\id)((\id\,\hat{\otimes}\,\hat{\tau})
((\id\,\hat{\otimes}\,\nu)(\nu(b)))),\\
&\qquad\quad(\id\otimes\theta_{\succ})(\theta_{\succ}(d))
\bullet(\id\,\hat{\otimes}\,\nu)(\nu(b))
=(\theta_{\succ}\otimes\id)(\theta_{\succ}(d))
\bullet(\id\,\hat{\otimes}\,\nu)(\nu(b))\\[-1mm]
&\qquad\qquad\qquad\qquad\qquad\qquad\qquad\qquad\quad
+(\theta_{\prec}\otimes\id)(\theta_{\succ}(d))
\bullet(\hat{\tau}\,\hat{\otimes}\,\id)((\id\,\hat{\otimes}\,\nu)(\nu(b))),\\
&(\id\otimes\theta_{\prec})(\theta_{\succ}(d))\bullet(\id\,\hat{\otimes}\,\hat{\tau})
((\id\,\hat{\otimes}\,\nu)(\nu(b)))=(\theta_{\succ}\otimes\id)(\theta_{\prec}(d))\bullet
(\id\,\hat{\otimes}\,\hat{\tau})((\hat{\tau}\,\hat{\otimes}\,\id)
((\id\,\hat{\otimes}\,\nu)(\nu(b)))).
\end{align*}
Therefore, we get
\begin{align*}
&\;(\id\,\hat{\otimes}\,\Delta)(\Delta(d\otimes b))\\
=&\;(\id\otimes\theta_{\succ})(\theta_{\succ}(d))\bullet(\id\,\hat{\otimes}\,\nu)(\nu(b))
+(\id\otimes\theta_{\prec})(\theta_{\prec}(d))\bullet(\id\,\hat{\otimes}\,\hat{\tau})
((\hat{\tau}\,\hat{\otimes}\,\id)((\id\,\hat{\otimes}\,\hat{\tau})
((\nu\,\hat{\otimes}\,\id)(\nu(b)))))\\
&\;+(\id\otimes\theta_{\prec})(\theta_{\succ}(d))
\bullet(\id\,\hat{\otimes}\,\hat{\tau})((\id\,\hat{\otimes}\,\nu)(\nu(b)))
+(\id\otimes\theta_{\succ})(\theta_{\prec}(d))\bullet(\hat{\tau}\,\hat{\otimes}\,\id)
((\id\,\hat{\otimes}\,\hat{\tau})((\nu\,\hat{\otimes}\,\id)(\nu(b))))\\
=&\;(\theta_{\succ}\otimes\id)(\theta_{\succ}(d))\bullet(\nu\,\hat{\otimes}\,\id)(\nu(b))
+(\theta_{\prec}\otimes\id)(\theta_{\prec}(d))\bullet(\hat{\tau}\,\hat{\otimes}\,\id)
((\id\,\hat{\otimes}\,\hat{\tau})((\hat{\tau}\,\hat{\otimes}\,\id)
((\id\,\hat{\otimes}\,\nu)(\nu(b)))))\\
&\;+(\theta_{\prec}\otimes\id)(\theta_{\succ}(d))
\bullet(\hat{\tau}\,\hat{\otimes}\,\id)((\nu\,\hat{\otimes}\,\id)(\nu(b)))
+(\theta_{\succ}\otimes\id)(\theta_{\prec}(d))\bullet(\id\,\hat{\otimes}\,\hat{\tau})
((\hat{\tau}\,\hat{\otimes}\,\id)((\id\,\hat{\otimes}\,\nu)(\nu(b))))\\
=&\;(\Delta\,\hat{\otimes}\,\id)(\Delta(d\otimes b)),
\end{align*}
since $(\id\,\hat{\otimes}\,\hat{\tau})(\hat{\tau}\,\hat{\otimes}\,\id)
(\id\,\hat{\otimes}\,\hat{\tau})=(\hat{\tau}\,\hat{\otimes}\,\id)
(\id\,\hat{\otimes}\,\hat{\tau})(\hat{\tau}\,\hat{\otimes}\,\id)$.
Thus, $(D\otimes B, \Delta)$ is a completed coassociative coalgebra.

Moreover, if $(B=\oplus_{i\in\bz}B_{i}, \nu)$ is the $\bz$-graded perm algebra given in
Example \ref{ex:grpermco}, then we have
\begin{align*}
&\;(\id\,\hat{\otimes}\,\nu)(\nu(\partial_{1}))\\
=&\;\sum_{i_{1}, i_{2}\in\bz}\sum_{j_{1}, j_{2}\in\bz}
\Big(x_{1}^{i_{1}}x_{2}^{i_{2}}\partial_{1}
\otimes\Big(x_{1}^{j_{1}}x_{2}^{j_{2}}\partial_{1}\otimes
x_{1}^{-i_{1}-j_{1}}x_{2}^{1-i_{2}-j_{2}+1}\partial_{1}
-x_{1}^{j_{1}}x_{2}^{j_{2}}\partial_{2}\otimes
x_{1}^{-i_{1}-j_{1}+1}x_{2}^{1-i_{2}-j_{2}}\partial_{1}\Big)\\[-4mm]
&\qquad\qquad\quad-x_{1}^{i_{1}}x_{2}^{i_{2}}\partial_{2}\otimes\Big(
x_{1}^{j_{1}}x_{2}^{j_{2}}\partial_{1}\otimes
x_{1}^{1-i_{1}-j_{1}}x_{2}^{-i_{2}-j_{2}+1}\partial_{1}
-x_{1}^{j_{1}}x_{2}^{j_{2}}\partial_{2}\otimes
x_{1}^{1-i_{1}-j_{1}+1}x_{2}^{-i_{2}-j_{2}}\partial_{1}\Big)\Big),
\end{align*}
and
\begin{align*}
&\;(\nu\,\hat{\otimes}\,\id)(\nu(\partial_{1}))\\
=&\;\sum_{i_{1}, i_{2}\in\bz}\sum_{j_{1}, j_{2}\in\bz}
\Big(\Big(x_{1}^{j_{1}}x_{2}^{j_{2}}\partial_{1}\otimes
x_{1}^{i_{1}-j_{1}}x_{2}^{i_{2}-j_{2}+1}\partial_{1}
-x_{1}^{j_{1}}x_{2}^{j_{2}}\partial_{1}\otimes
x_{1}^{i_{1}-j_{1}+1}x_{2}^{i_{2}-j_{2}}\partial_{1}\Big)
\otimes x_{1}^{-i_{1}}x_{2}^{1-i_{2}}\partial_{1}\\[-4mm]
&\qquad\qquad\quad-\Big(x_{1}^{j_{1}}x_{2}^{j_{2}}\partial_{1}\otimes
x_{1}^{i_{1}-j_{1}}x_{2}^{i_{2}-j_{2}+1}\partial_{2}
-x_{1}^{j_{1}}x_{2}^{j_{2}}\partial_{2}\otimes
x_{1}^{i_{1}-j_{1}+1}x_{2}^{i_{2}-j_{2}}\partial_{2}
\otimes x_{1}^{1-i_{1}}x_{2}^{-i_{2}}\partial_{1}\Big)\Big).
\end{align*}
Comparing the coefficients of $\partial_{2}\otimes\partial_{2}\otimes\partial_{1}$
and $\partial_{1}\otimes\partial_{1}\otimes\partial_{1}$ in equation
$(\id\,\hat{\otimes}\,\Delta)(\Delta(d\otimes\partial_{1}))
=(\Delta\,\hat{\otimes}\,\id)(\Delta(d\otimes\partial_{1}))$, we get
$(\id\otimes\theta_{\succ})\theta_{\succ}=(\theta_{\prec}\otimes\id)\theta_{\succ}
+(\theta_{\succ}\otimes\id)\theta_{\succ}$ and $(\theta_{\succ}\otimes\id)\theta_{\prec}
=(\id\otimes\theta_{\prec})\theta_{\succ}$. Similarly, comparing the coefficients of
$\partial_{1}\otimes\partial_{1}\otimes\partial_{2}$ in equation
$(\id\,\hat{\otimes}\,\Delta)(\Delta(d\otimes\partial_{2}))
=(\Delta\,\hat{\otimes}\,\id)(\Delta(d\otimes\partial_{2}))$, we get
$(\theta_{\prec}\otimes\id)\theta_{\prec}=(\id\otimes\theta_{\prec})\theta_{\prec}
+(\id\otimes\theta_{\succ})\theta_{\prec}$. Thus, $(D, \theta_{\prec}, \theta_{\succ})$
is a dendriform coalgebra. The proof is finished.
\end{proof}

\subsection{Completed ASI bialgebras from dendriform $\md$-bialgebra}\label{subsec:affbia}

In Lemma \ref{lem:l-ass-dual}, we have shown that a quadratic perm algebra induces
a perm coalgebra. Now we extend this result to the case of completed.

\begin{defi}\label{def:quad}
Let $\varpi(-, -)$ be a bilinear form on a $\bz$-graded perm algebra
$(B=\oplus_{i\in\bz}B_{i}, \cdot)$.

$(i)$ $\varpi(-, -)$ is called {\bf invariant}, if $\varpi(b_{1}b_{2},\; b_{3})=
        \varpi(b_{1},\; b_{2}b_{3}-b_{3}b_{2})$ for any $b_{1}, b_{2}, b_{3}\in B$;

$(ii)$ $\varpi(-, -)$ is called {\bf graded}, if there exists some $m\in\bz$
        such that $\varpi(B_{i}, B_{j})=0$ when $i+j+m\neq0$.\\
A {\bf quadratic $\bz$-graded perm algebra}, denoted by $(B=\oplus_{i\in\bz}B_{i},
\cdot, \varpi)$, is a $\bz$-graded perm algebra together with an antisymmetric
invariant nondegenerate graded bilinear form.
In particular, if $B=B_{0}$, it is just the quadratic perm algebra.
\end{defi}

\begin{ex}[\cite{LZB}]\label{ex:qu-perm}
Let $(B=\oplus_{i\in\bz}B_{i}, \cdot)$ be the $\bz$-graded perm algebra given in
Example \ref{ex:grperm}, where $B=\{f_{1}\partial_{1}+f_{2}\partial_{2}\mid f_{1},
f_{2}\in\Bbbk[x_{1}^{\pm}, x_{2}^{\pm}]\}=\oplus_{i\in\bz}B_{i}$. Define an
antisymmetric bilinear form $\varpi(-,-)$ on $(B=\oplus_{i\in\bz}B_{i}, \cdot)$ by
\begin{align*}
&\varpi(x_{1}^{i_{1}}x_{2}^{i_{2}}\partial_{2},\ \ x_{1}^{j_{1}}x_{2}^{j_{2}}\partial_{1})
=-\varpi(x_{1}^{j_{1}}x_{2}^{j_{2}}\partial_{1},\ \ x_{1}^{i_{1}}x_{2}^{i_{2}}\partial_{2})
=\delta_{i_{1}+j_{1}, 0}\delta_{i_{2}+j_{2}, 0}, \\
&\qquad\quad\varpi(x_{1}^{i_{1}}x_{2}^{i_{2}}\partial_{1},\ \
x_{1}^{j_{1}}x_{2}^{j_{2}}\partial_{1})=\varpi(x_{1}^{i_{1}}x_{2}^{i_{2}}\partial_{2},\ \
x_{1}^{j_{1}}x_{2}^{j_{2}}\partial_{2})=0,
\end{align*}
for any $i_{1}, i_{2}, j_{1}, j_{2}\in\bz$.
Then $(B=\oplus_{i\in\bz}B_{i}, \cdot, \varpi)$ is a quadratic $\bz$-graded perm algebra.
Moreover, $\{x_{1}^{-i_{1}}x_{2}^{-i_{2}}\partial_{2}$,\; $-x_{1}^{-i_{1}}
x_{2}^{-i_{2}}\partial_{1}\mid i_{1}, i_{2}\in\bz\}$ is the dual basis of
$\{x_{1}^{i_{1}}x_{2}^{i_{2}}\partial_{1},\; x_{1}^{i_{1}}x_{2}^{i_{2}}\partial_{2}\mid
i_{1}, i_{2}\in\bz\}$ with respect to $\varpi(-,-)$, consisting of homogeneous elements.
\end{ex}

For a quadratic $\bz$-graded perm algebra $(B=\oplus_{i\in\bz}B_{i}, \cdot, \varpi)$,
we have $\varpi(b_{1}b_{2},\; b_{3})=\omega(b_{2},\; b_{1}b_{3})$
for any $b_{1}, b_{2}, b_{3}\in B$. Moreover, the antisymmetric nondegenerate
bilinear form $\varpi(-,-)$ induces bilinear forms
$$
(\underbrace{B\,\hat{\otimes}\,B\,\hat{\otimes}\,\cdots\,\hat{\otimes}\,
B}_{n\text{-fold}})\otimes(\underbrace{B\otimes B\otimes\cdots
\otimes B}_{n\text{-fold}})\longrightarrow\Bbbk,
$$
for all $n\geq2$, which are denoted by $\hat{\varpi}(-,-)$, are defined by
$$
\hat{\varpi}\Big(\sum_{i_{1},\cdots,i_{n},\alpha} x_{1, i_{1}, \alpha}
\otimes\cdots\otimes x_{n, i_{n}, \alpha},\ \ y_{1}\otimes\cdots\otimes y_{n}\Big)
=\sum_{i_{1},\cdots,i_{n},\alpha}\prod_{j=1}^{n}
\varpi(x_{j, i_{j}, \alpha},\; y_{j}).
$$
Then, one can check that $\hat{\varpi}(-,-)$ is {\bf left nondegenerate}, i.e., if
$$
\hat{\varpi}\Big(\sum_{i_{1}, \cdots, i_{n},\alpha}x_{1, i_{1}, \alpha}
\otimes\cdots\otimes x_{n, i_{n}, \alpha},\ \ y_{1}\otimes\cdots\otimes y_{n}\Big)
=\hat{\varpi}\Big(\sum_{i_{1},\cdots,i_{n},\alpha} z_{1, i_{1}, \alpha}
\otimes\cdots\otimes z_{n, i_{n}, \alpha},\ \ y_{1}\otimes\cdots\otimes y_{n}\Big),
$$
for all homogeneous elements $y_{1}, y_{2},\cdots, y_{n}\in B$, then
$$
\sum_{i_{1},\cdots,i_{n},\alpha} x_{1, i_{1}, \alpha}
\otimes\cdots\otimes x_{n, i_{n}, \alpha}
=\sum_{i_{1},\cdots,i_{n},\alpha} z_{1, i_{1}, \alpha}
\otimes\cdots\otimes z_{n, i_{n}, \alpha}.
$$

Similar to Lemma \ref{lem:l-ass-dual}, we have

\begin{lem}\label{lem:comp-dual}
Let $(B=\oplus_{i\in\bz}B_{i}, \cdot, \varpi)$ be a quadratic $\bz$-graded perm algebra.
Define a linear map $\nu_{\varpi}: B\rightarrow B\otimes B$ by
$\hat{\varpi}(\nu_{\omega}(b_{1}),\; b_{2}\otimes b_{3})
=-\varpi(b_{1},\; b_{2}b_{3})$, for any $b_{1}, b_{2}, b_{3}\in B$.
Then $(B, \nu_{\omega})$ is a completed perm coalgebra.
\end{lem}

\begin{ex}[\cite{LZB}]\label{ex:ind-coperm}
Consider the quadratic $\bz$-graded perm algebra $(B=\oplus_{i\in\bz}B_{i}, \cdot, \varpi)$
given in Example \ref{ex:qu-perm}. Then the induced completed perm coalgebra
$(B=\oplus_{i\in\bz}B_{i}, \nu_{\omega})$ is just the completed perm coalgebra
$(B=\oplus_{i\in\bz}B_{i}, \nu)$ given in Example \ref{ex:grpermco}.
\end{ex}

Next, we extend the dendriform algebra affinization and the dendriform coalgebra
affinization to dendriform $\md$-bialgebras.

\begin{defi}[\cite{Bai}]\label{def:D-bialg}
A {\bf dendriform $\md$-bialgebra} is a quintuple $(D, \prec, \succ, \theta_{\prec},
\theta_{\succ})$, where $(D, \prec$, $\succ)$ is a dendriform algebra,
$(D, \theta_{\prec}, \theta_{\succ})$ is a dendriform coalgebra and the following
compatible conditions hold:
\begin{align}
&\qquad \theta_{\prec}(d_{1}\prec d_{2}+d_{1}\succ d_{2})
=(\id\otimes\fl_{\succ}(d_{1}))(\theta_{\prec}(d_{2}))
+((\fr_{\prec}+\fr_{\succ})(d_{2})\otimes\id)(\theta_{\prec}(d_{1})),   \label{D-bi1}\\
&\qquad \theta_{\succ}(d_{1}\prec d_{2}+d_{1}\succ d_{2})
=(\id\otimes(\fl_{\prec}+\fl_{\succ})(d_{1}))(\theta_{\succ}(d_{2}))
+(\fr_{\prec}(d_{2})\otimes\id)(\theta_{\succ}(d_{1})),                 \label{D-bi2}\\
&\quad\; \theta_{\prec}(d_{1}\prec d_{2})+\theta_{\succ}(d_{1}\prec d_{2})
=(\id\otimes\fl_{\prec}(d_{1}))(\theta_{\succ}(d_{2}))
+(\fr_{\prec}(d_{2})\otimes\id)((\theta_{\prec}+\theta_{\succ})(d_{1})), \label{D-bi3}\\
&\quad\; \theta_{\prec}(d_{1}\succ d_{2})+\theta_{\succ}(d_{1}\succ d_{2})
=(\id\otimes\fl_{\prec}(d_{1}))((\theta_{\prec}+\theta_{\succ})(d_{2}))
+(\fr_{\prec}(d_{2})\otimes\id)(\theta_{\prec}(d_{1})),                  \label{D-bi4}\\
&\; \big((\fl_{\prec}+\fl_{\succ})(d_{1})\otimes\id-\id\otimes\fr_{\prec}(d_{1})\big)
(\theta_{\prec}(d_{2}))
=-\tau\big(\big(\fl_{\succ}(d_{2})\otimes\id-\id\otimes(\fr_{\prec}+\fr_{\succ})(d_{2})\big)
(\theta_{\succ}(d_{1}))\big),                                            \label{D-bi5}\\
&\big(\fl_{\succ}(d_{2})\otimes\id-\id\otimes\fr_{\prec}(d_{2})\big)
\big((\theta_{\prec}+\theta_{\succ})(d_{1})\big)
=\tau\big((\id\otimes\fr_{\succ}(d_{2}))(\theta_{\succ}(d_{2}))
-(\fl_{\prec}(d_{2})\otimes\id)(\theta_{\prec}(d_{2}))\big),             \label{D-bi6}
\end{align}
for any $d_{1}, d_{2}\in D$.
\end{defi}

Recall that an {\bf antisymmetric infinitesimal bialgebra}, or simply an ASI bialgebra
is a triple $(A, \ast, \Delta)$ consisting of a vector space $A$ and linear maps
$\ast: A\otimes A\rightarrow A$ and $\Delta: A\rightarrow A\otimes A$ such that
$(A, \ast)$ is an associative algebra, $(A, \Delta)$ is a coassociative coalgebra
and for any $a_{1}, a_{2}\in A$,
\begin{align*}
&\qquad\quad\Delta(a_{1}\ast a_{2})=(\fr_{A}(a_{2})\otimes\id)(\Delta(a_{1}))
+(\id\otimes\,\fl_{A}(a_{1}))(\Delta(a_{2})), \\
&\big(\fl_{A}(a_{1})\otimes\id-\id\otimes\,\fr_{A}(a_{1})\big)(\Delta(a_{2}))
=\tau\big(\big(\id\otimes\,\fr_{A}(a_{2})
-\fl_{A}(a_{2})\otimes\id\big)(\Delta(a_{1}))\big).
\end{align*}
We now give the notion and results on completed ASI bialgebras.

\begin{defi}[\cite{LZB}]\label{def:CASIbia}
A {\bf completed ASI bialgebra} is a triple $(A, \ast, \Delta)$ consisting of a
vector space $A$ and linear maps $\ast: A\otimes A\rightarrow A$ and
$\Delta: A\rightarrow A\otimes A$ such that
\begin{enumerate}\itemsep=0pt
\item[$(i)$] $(A, \ast)$ is a $\bz$-graded associative algebra;
\item[$(ii)$] $(A, \Delta)$ is a completed coassociative coalgebra;
\item[$(iii)$] for any $a_{1}, a_{2}\in A$,
\begin{align}
&\qquad\qquad\Delta(a_{1}\ast a_{2})=(\fr_{A}(a_{2})\,\hat{\otimes}\,\id)
(\Delta(a_{1}))+(\id\,\hat{\otimes}\,\fl_{A}(a_{1}))(\Delta(a_{2})),   \label{CASI1} \\
&\big(\fl_{A}(a_{1})\,\hat{\otimes}\,\id-\id\,\hat{\otimes}\,\fr_{A}(a_{1})\big)
(\Delta(a_{2}))=\hat{\tau}\big(\big(\id\,\hat{\otimes}\,\fr_{A}(a_{2})
-\fl_{A}(a_{2})\,\hat{\otimes}\,\id\big)(\Delta(a_{1}))\big).          \label{CASI2}
\end{align}
\end{enumerate}
\end{defi}

\begin{thm}\label{thm:den-perm-ass}
Let $(D, \prec, \succ, \theta_{\prec}, \theta_{\succ})$ be a finite-dimensional
dendriform $\md$-bialgebra, $(B=\oplus_{i\in\bz}B_{i}, \cdot, \varpi)$ be a quadratic
$\bz$-graded perm algebra and $(D\otimes B, \ast)$ be the induced $\bz$-graded
associative algebra from $(D, \prec, \succ)$ by $(B=\oplus_{i\in\bz}B_{i}, \cdot)$.
Define a linear map $\Delta: D\otimes B\rightarrow(D\otimes B)\otimes(D\otimes B)$ by
\begin{align*}
\Delta(d\otimes b)&=\theta_{\succ}(d)\bullet\nu_{\varpi}(b)
+\theta_{\prec}(d)\bullet\,\hat{\tau}(\nu_{\varpi}(b))\\
&:=\sum_{[d]}\sum_{i,j,\alpha}(d_{[1]}\otimes b_{1,i,\alpha})
\otimes(d_{[2]}\otimes b_{2,j,\alpha})+\sum_{(d)}\sum_{i,j,\alpha}(d_{(1)}\otimes
b_{2,j,\alpha})\otimes(d_{(2)}\otimes b_{1,i,\alpha}),
\end{align*}
for any $d\in D$ and $b\in B$, where $\theta_{\prec}(d)=\sum_{(d)}d_{(1)}
\otimes d_{(2)}$, $\theta_{\succ}(d)=\sum_{[d]}d_{[1]}\otimes d_{[2]}$ in the
Sweedler notation and $\nu_{\varpi}(b)=\sum_{i,j,\alpha}b_{1,i,\alpha}\otimes
b_{2,j,\alpha}$. Then $(D\otimes B, \ast, \Delta)$ is a completed ASI bialgebra,
which is called the {\bf completed ASI bialgebra induced from $(D, \prec, \succ,
\theta_{\prec}, \theta_{\succ})$ by $(B, \cdot, \varpi)$}.

Moreover, if $(B=\oplus_{i\in\bz}B_{i}, \cdot, \varpi)$ is the quadratic $\bz$-graded
perm algebra given in Example \ref{ex:qu-perm}, then $(D\otimes B, \ast, \Delta)$ is
a completed ASI bialgebra if and only if $(D, \prec, \succ, \theta_{\prec},
\theta_{\succ})$ is a dendriform $\md$-bialgebra.
\end{thm}

\begin{proof}
By Proposition \ref{pro:dco-coass} and Lemma \ref{lem:comp-dual}, we get
$(D\otimes B, \Delta)$ is a completed coassociative coalgebra. Thus, we only
need to show that Eqs. \eqref{CASI1} and \eqref{CASI2} hold.
For any $d, d'\in D$ and $b, b'\in B$, we have
\begin{align*}
&\;\Delta((d\otimes b)\ast(d'\otimes b'))
-(\fr_{D\otimes B}(d'\otimes b')\,\hat{\otimes}\,\id)(\Delta(d\otimes b))
-(\id\,\hat{\otimes}\,\fl_{D\otimes B}(d\otimes b))(\Delta(d'\otimes b'))\\
=&\;\Delta((d\succ d')\otimes(bb')+(d\prec d')\otimes(b'b))
-(\fr_{D\otimes B}(d'\otimes b')\,\hat{\otimes}\,\id)
\big(\theta_{\succ}(d)\bullet\nu_{\omega}(b)
+\theta_{\prec}(d)\bullet\hat{\tau}(\nu_{\varpi}(b))\big)\\[-1mm]
&\quad-(\id\,\hat{\otimes}\,\fl_{D\otimes B}(d\otimes b))\big(\theta_{\succ}(d')
\bullet\nu_{\varpi}(b')+\theta_{\prec}(d')\bullet\hat{\tau}(\nu_{\varpi}(b'))\big)\\
=&\;\theta_{\succ}(d\succ d')\bullet\nu_{\varpi}(bb')
+\theta_{\prec}(d\succ d')\bullet\hat{\tau}(\nu_{\varpi}(bb'))
+\theta_{\succ}(d\prec d')\bullet\nu_{\varpi}(b'b)
+\theta_{\prec}(d\prec d')\bullet\hat{\tau}(\nu_{\varpi}(b'b))
\end{align*}
\begin{align*}
&\;-\sum_{[d]}\sum_{i,j,\alpha}\Big(\big((d_{[1]}\succ d')\otimes(b_{1,i,\alpha}b')\big)
\otimes(d_{[2]}\otimes b_{2,j,\alpha})+\big((d_{[1]}\prec d')
\otimes(b'b_{1,i,\alpha})\big)\otimes(d_{[2]}\otimes b_{2,j,\alpha})\Big)\\[-2mm]
&\;-\sum_{(d)}\sum_{i,j,\alpha}\Big(\big((d_{(1)}\succ d')\otimes(b_{2,j,\alpha}b')\big)
\otimes(d_{(2)}\otimes b_{1,i,\alpha})+\big((d_{(1)}\prec d')
\otimes(b'b_{2,j,\alpha})\big)\otimes(d_{(2)}\otimes b_{1,i,\alpha})\Big)\\[-2mm]
&\;-\sum_{[d']}\sum_{i,j,\alpha}\Big((d'_{[1]}\otimes b'_{1,i,\alpha})\otimes
\big((d\succ d'_{[2]})\otimes(bb'_{2,j,\alpha})\big)+(d'_{[1]}\otimes b'_{1,i,\alpha})
\otimes\big((d\prec d'_{[2]})\otimes(b'_{2,j,\alpha}b)\big)\Big)\\[-2mm]
&\;-\sum_{(d')}\sum_{i,j,\alpha}\Big((d'_{(1)}\otimes b'_{2,j,\alpha})\otimes
\big((d\succ d'_{(2)})\otimes(bb'_{1,i,\alpha})\big)+(d'_{(1)}\otimes b'_{2,j,\alpha})
\otimes\big((d\prec d'_{(2)})\otimes(b'_{1,i,\alpha}b)\big)\Big).
\end{align*}
Since $\varpi(-,-)$ on $(B, \cdot)$ is invariant, we have
\begin{align*}
\hat{\varpi}(\nu_{\varpi}(bb'),\; e\otimes f)
&=-\varpi(bb',\; ef)=-\varpi(b',\; b(ef)),\\
\hat{\varpi}\Big(\sum_{i,j,\alpha}((bb'_{1,i,\alpha})\otimes
b'_{2,j,\alpha}),\; e\otimes f\Big)
&=-\varpi(b',\; (be)f)=-\varpi(b',\; b(ef)),
\end{align*}
for any $b, b', e, f\in B$. Since $\hat{\varpi}(-,-)$ is left nondegenerate,
we obtain $\nu_{\varpi}(bb')=\sum_{i,j,\alpha}((bb'_{1,i,\alpha})\otimes
b'_{2,j,\alpha})$. Similarly, we also have
\begin{align*}
&\qquad\qquad\sum_{i,j,\alpha}((b_{1,i,\alpha}b')\otimes b_{2,j,\alpha})
=\sum_{i,j,\alpha}(b'_{2,j,\alpha}\otimes(b'_{1,i,\alpha}b))=0,\\[-2mm]
&\qquad\; \hat{\tau}(\nu_{\varpi}(b'b))=\sum_{i,j,\alpha}((b'b_{2,j,\alpha})
\otimes b_{1,i,\alpha})=\hat{\tau}(\nu_{\varpi}(bb'))-\Phi,\\[-2mm]
&\qquad \hat{\tau}(\nu_{\varpi}(bb'))=\sum_{i,j,\alpha}((b_{2,j,\alpha}b')
\otimes b_{1,i,\alpha})=\sum_{i,j,\alpha}(b'_{2,j,\alpha}\otimes(bb'_{1,i,\alpha})),\\[-2mm]
& \nu_{\varpi}(b'b)=\sum_{i,j,\alpha}((b'b_{1,i,\alpha})\otimes b_{2,j,\alpha})
=\sum_{i,j,\alpha}(b_{1,i,\alpha}\otimes(b'b_{(2,j,\alpha)}))=\nu_{\varpi}(bb')-\Phi,
\end{align*}
where $\Phi\in B\otimes B$ such that $\hat{\varpi}(\Phi,\; e\otimes f)
=\varpi(b',\; e(fb))$. Thus,
\begin{align*}
&\;\Delta((d\otimes b)\ast(d'\otimes b'))
-(\fr_{D\otimes B}(d'\otimes b')\,\hat{\otimes}\,\id)(\Delta(d\otimes b))
-(\id\,\hat{\otimes}\,\fl_{D\otimes B}(d\otimes b))(\Delta(d'\otimes b'))\\
=&\;\theta_{\succ}(dd')\bullet\nu_{\varpi}(bb')
+\theta_{\prec}(d\succ d')\bullet\hat{\tau}(\nu_{\varpi}(bb'))
+\theta_{\succ}(d\prec d')\bullet\big(\nu_{\varpi}(bb')-\Phi\big)\\
&\quad+\theta_{\prec}(d\prec d')\bullet\big(\hat{\tau}(\nu_{\varpi}(bb')-\Phi)\big)
-(\fr_{\prec}(d')\otimes\id)(\theta_{\succ}(d))\bullet\big(\nu_{\varpi}(bb')-\Phi\big)\\
&\quad-(\fr_{\succ}(d')\otimes\id)(\theta_{\prec}(d))\bullet\hat{\tau}(\nu_{\varpi}(bb'))
-(\fr_{\prec}(d')\otimes\id)(\theta_{\prec}(d))\bullet
\big(\hat{\tau}(\nu_{\varpi}(bb'))-\Phi\big)\\
&\quad-(\id\otimes\fl_{\succ}(d))(\theta_{\succ}(d'))\bullet\nu_{\varpi}(bb')
-(\id\otimes\fl_{\prec}(d))(\theta_{\succ}(d'))\bullet\big(\nu_{\varpi}(bb')-\Phi\big)\\
&\quad-(\id\otimes\fl_{\succ}(d))(\theta_{\prec}(d'))\bullet\hat{\tau}(\nu_{\varpi}(bb'))\\
=&\;\Big(\theta_{\succ}(d\succ d')+\theta_{\succ}(d\prec d')
-(\fr_{\prec}(d')\otimes\id)(\theta_{\succ}(d))\\[-2mm]
&\qquad-(\id\otimes\fl_{\succ}(d))(\theta_{\succ}(d'))
-(\id\otimes\fl_{\prec}(d))(\theta_{\succ}(d'))\Big)\bullet\nu_{\varpi}(bb')\\[-1mm]
&\quad+\Big(\theta_{\prec}(d\succ d')+\theta_{\prec}(d\prec d')
-(\fr_{\succ}(d')\otimes\id)(\theta_{\prec}(d))\\[-2mm]
&\qquad-(\fr_{\prec}(d')\otimes\id)(\theta_{\prec}(d))
-(\id\otimes\fl_{\succ}(d))(\theta_{\prec}(d'))\Big)
\bullet\hat{\tau}(\nu_{\varpi}(bb'))\\[-1mm]
&\quad-\Big(\theta_{\succ}(d\prec d')+\theta_{\prec}(d\prec d')
-(\fr_{\prec}(d')\otimes\id)(\theta_{\succ}(d))\\[-2mm]
&\qquad-(\fr_{\prec}(d')\otimes\id)(\theta_{\prec}(d))
-(\id\otimes\fl_{\prec}(d))(\theta_{\succ}(d'))\Big)\bullet\Phi\\
=&\; 0.
\end{align*}
Similarly, we also have $\big(\fl_{D\otimes B}(d\otimes b)\,\hat{\otimes}\,\id
-\id\,\hat{\otimes}\,\fr_{D\otimes B}(d\otimes b)\big)(\Delta(d'\otimes b'))
=\tau\big(\big(\id\,\hat{\otimes}\,\fr_{D\otimes B}(d'\otimes b')
-\fl_{D\otimes B}(d'\otimes b')\,\hat{\otimes}\,\id\big)(\Delta(d\otimes b))\big)$,
for any $d, d'\in D$ and $b, b'\in B$. Thus, $(D\otimes B, \ast, \Delta)$ is a
completed ASI bialgebra.

Conversely, if $(B=\oplus_{i\in\bz}B_{i}, \cdot, \varpi)$ is the quadratic $\bz$-graded
perm algebra given in Example \ref{ex:qu-perm} and $(D\otimes B, \ast, \Delta)$ is
a completed ASI bialgebra, then $(D, \prec$, $\succ)$ is a
dendriform algebra and $(D, \theta_{\prec}, \theta_{\succ})$ is a dendriform coalgebra
by Propositions \ref{pro:aff-dalg} and \ref{pro:dco-coass} respectively.
Now we only need to prove that $(D, \prec, \succ, \theta_{\prec}, \theta_{\succ})$ is a
dendriform $\md$-bialgebra. Since $(D\otimes B, \ast, \Delta)$ is a completed ASI
bialgebra, we have
\begin{align*}
0=&\;\Delta((d\otimes\partial_{1})\ast(d'\otimes\partial_{1}))
-(\fr_{D\otimes B}(d'\otimes\partial_{1})\,\hat{\otimes}\,\id)(\Delta(d\otimes\partial_{1}))
-(\id\,\hat{\otimes}\,\fl_{D\otimes B}(d\otimes\partial_{1}))(\Delta(d'\otimes\partial_{1}))\\
=&\;\Delta((d\succ d')\otimes(x_{1}\partial_{1})+(d\prec d')\otimes(x_{1}\partial_{1}))
-(\fr_{D\otimes B}(d'\otimes\partial_{1})\,\hat{\otimes}\,\id)
\big(\theta_{\succ}(d)\bullet\nu_{\omega}(\partial_{1})\\[-1mm]
&\quad+\theta_{\prec}(d)\bullet\hat{\tau}(\nu_{\omega}(\partial_{1}))\big)
-(\id\,\hat{\otimes}\,\fl_{D\otimes B}(d\otimes\partial_{1}))\big(\theta_{\succ}(d')
\bullet\nu_{\omega}(\partial_{1})
+\theta_{\prec}(d')\bullet\hat{\tau}(\nu_{\omega}(\partial_{1}))\big)\\
=&\;\big(\theta_{\succ}(d\succ d')+\theta_{\succ}(d\prec d')\big)
\bullet\nu_{\varpi}(x_{1}\partial_{1})
+\big(\theta_{\prec}(d\succ d')+\theta_{\prec}(d\prec d')\big)
\bullet\,\hat{\tau}(\nu_{\varpi}(x_{1}\partial_{1}))\\
&\quad-\sum_{[d]}\sum_{i_{1},i_{2}}\Big(
\Big((d_{[1]}\succ d')\otimes(x_{1}^{i_{1}+1}x_{2}^{i_{2}}\partial_{1})
+(d_{[1]}\prec d')\otimes(x_{1}^{i_{1}+1}x_{2}^{i_{2}}\partial_{1})\Big)
\otimes(d_{[2]}\otimes(x_{1}^{-i_{1}}x_{2}^{1-i_{2}}\partial_{1}))\\[-5mm]
&\qquad\qquad\quad-\Big((d_{[1]}\succ d')\otimes(x_{1}^{i_{1}}x_{2}^{i_{2}+1}\partial_{1})
+(d_{[1]}\prec d')\otimes(x_{1}^{i_{1}+1}x_{2}^{i_{2}}\partial_{2})\Big)
\otimes(d_{[2]}\otimes(x_{1}^{1-i_{1}}x_{2}^{-i_{2}}\partial_{1}))\Big)\\
&\quad-\sum_{(d)}\sum_{i_{1},i_{2}}\Big(
\Big((d_{(1)}\succ d')\otimes(x_{1}^{i_{1}+1}x_{2}^{i_{2}}\partial_{1})
+(d_{(1)}\prec d')\otimes(x_{1}^{i_{1}+1}x_{2}^{i_{2}}\partial_{1})\Big)
\otimes(d_{(2)}\otimes(x_{1}^{-i_{1}}x_{2}^{1-i_{2}}\partial_{1}))\\[-5mm]
&\qquad\qquad\quad-\Big((d_{(1)}\succ d')\otimes(x_{1}^{i_{1}+1}x_{2}^{i_{2}}\partial_{1})
+(d_{(1)}\prec d')\otimes(x_{1}^{i_{1}+1}x_{2}^{i_{2}}\partial_{1})\Big)
\otimes(d_{(2)}\otimes(x_{1}^{1-i_{1}}x_{2}^{-i_{2}}\partial_{2}))\Big)\\
&\quad-\sum_{[d']}\sum_{i_{1},i_{2}}\Big(
(d'_{[1]}\otimes(x_{1}^{i_{1}}x_{2}^{i_{2}}\partial_{1}))\otimes
\Big((d\succ d'_{[2]})\otimes(x_{1}^{1-i_{1}}x_{2}^{1-i_{2}}\partial_{1})
+(d\prec d'_{[2]})\otimes(x_{1}^{1-i_{1}}x_{2}^{1-i_{2}}\partial_{1})\Big)\\[-5mm]
&\qquad\qquad\quad-(d'_{[1]}\otimes(x_{1}^{i_{1}}x_{2}^{i_{2}}\partial_{2}))\otimes
\Big((d\succ d'_{[2]})\otimes(x_{1}^{2-i_{1}}x_{2}^{-i_{2}}\partial_{1})
+(d\prec d'_{[2]})\otimes(x_{1}^{2-i_{1}}x_{2}^{-i_{2}}\partial_{1})\Big)\Big)\\
&\quad-\sum_{(d')}\sum_{i_{1},i_{2}}\Big(
(d'_{(1)}\otimes(x_{1}^{i_{1}}x_{2}^{i_{2}}\partial_{1}))\otimes
\Big((d\succ d'_{(2)})\otimes(x_{1}^{1-i_{1}}x_{2}^{1-i_{2}}\partial_{1})
+(d\prec d'_{(2)})\otimes(x_{1}^{1-i_{1}}x_{2}^{1-i_{2}}\partial_{1})\Big)\\[-5mm]
&\qquad\qquad\quad-(d'_{(1)}\otimes(x_{1}^{i_{1}}x_{2}^{i_{2}}\partial_{1}))\otimes
\Big((d\succ d'_{(2)})\otimes(x_{1}^{2-i_{1}}x_{2}^{-i_{2}}\partial_{2})
+(d\prec d'_{(2)})\otimes(x_{1}^{1-i_{1}}x_{2}^{1-i_{2}}\partial_{1})\Big)\Big).
\end{align*}
Comparing the coefficients of $\partial_{2}\otimes\partial_{1}$ and $\partial_{1}\otimes
\partial_{2}$ in the equation above, we get Eqs. \eqref{D-bi1} and \eqref{D-bi2} hold.
Similarly, comparing the coefficients of $\partial_{1}\otimes\partial_{1}$ and
$\partial_{2}\otimes\partial_{2}$ in the equation $\Delta((d\otimes\partial_{1})
\ast(d'\otimes\partial_{2}))-(\fr_{D\otimes B}(d'\otimes\partial_{2})\,\hat{\otimes}\,\id)
(\Delta(d\otimes\partial_{1}))-(\id\,\hat{\otimes}\,\fl_{D\otimes B}(d\otimes\partial_{1}))
(\Delta(d'\otimes\partial_{2}))=0$, we get Eqs. \eqref{D-bi3} and \eqref{D-bi4} hold.
Comparing the coefficients of $\partial_{1}\otimes\partial_{2}$ in the equation
$\big(\fl_{D\otimes B}(d\otimes\partial_{1})\,\hat{\otimes}\,\id
-\id\,\hat{\otimes}\,\fr_{D\otimes B}(d\otimes\partial_{1})\big)
(\Delta(d'\otimes\partial_{1}))=\hat{\tau}\big(\big(\id\,\hat{\otimes}\,\fr_{D\otimes B}
(d'\otimes\partial_{1})-\fl_{D\otimes B}(d'\otimes\partial_{1})\,\hat{\otimes}\,\id\big)
(\Delta(d\otimes\partial_{1}))\big)$, we get Eq. \eqref{D-bi5} holds.
Comparing the coefficients of $\partial_{2}\otimes\partial_{2}$ in the equation
$\big(\fl_{D\otimes B}(d\otimes\partial_{2})\,\hat{\otimes}\,\id
-\id\,\hat{\otimes}\,\fr_{D\otimes B}(d\otimes\partial_{2})\big)
(\Delta(d'\otimes\partial_{1}))=\hat{\tau}\big(\big(\id\,\hat{\otimes}\,\fr_{D\otimes B}
(d'\otimes\partial_{1})-\fl_{D\otimes B}(d'\otimes\partial_{1})\,\hat{\otimes}\,\id\big)
(\Delta(d\otimes\partial_{2}))\big)$, we get Eq. \eqref{D-bi6} holds.
Thus, $(D, \prec, \succ, \theta_{\prec}, \theta_{\succ})$ is a dendriform
$\md$-bialgebra in this case.
\end{proof}

We have constructed an infinite-dimensional ASI bialgebra using the affinization
of a dendriform bialgebra. Now let us return to finite-dimensional ASI bialgebras.
First, for a special case of Theorem \ref{thm:den-perm-ass}, where $(B=B_{0}, \cdot,
\omega)$ is a finite-dimensional quadratic perm algebra algebra, we have

\begin{cor}\label{cor:indassbia}
Let $(D, \prec, \succ, \theta_{\prec}, \theta_{\succ})$ be a dendriform
$\md$-bialgebra and $(B, \cdot, \omega)$ be a quadratic perm algebra, and
$(D\otimes B, \ast)$ be the induced associative algebra from $(D, \prec, \succ)$
by $(B, \cdot)$. Define a linear map $\Delta: D\otimes B\rightarrow(D\otimes B)
\otimes(D\otimes B)$ by
\begin{align}
\Delta(d\otimes b)&=\theta_{\succ}(d)\bullet\nu_{\omega}(b)
+\theta_{\prec}(d)\bullet\tau(\nu_{\omega}(b))   \label{ind-coass}\\
&:=\sum_{[d]}\sum_{(b)}(d_{[1]}\otimes b_{(1)})\otimes(d_{[2]}\otimes b_{(2)})
+\sum_{(d)}\sum_{(b)}(d_{(1)}\otimes b_{(2)})\otimes(d_{(2)}\otimes b_{(1)}), \nonumber
\end{align}
for any $d\in D$ and $b\in B$, where $\theta_{\prec}(d)=\sum_{(d)}d_{(1)}
\otimes d_{(2)}$, $\theta_{\succ}(d)=\sum_{[d]}d_{[1]}\otimes d_{[2]}$ and
$\nu_{\omega}(b)=\sum_{(b)}b_{(1)}\otimes b_{(2)}$ in the Sweedler notation.
Then $(D\otimes B, \ast, \Delta)$ is an ASI bialgebra, which is called the {\bf
ASI bialgebra induced from $(D, \prec, \succ,\theta_{\prec}, \theta_{\succ})$
by $(B, \cdot, \omega)$}.
\end{cor}


\subsection{Triangular ASI bialgebras from triangular dendriform $\md$-bialgebras}
\label{subsec:triASI}

Recall that an ASI bialgebra $(A, \ast, \Delta)$ is called {\bf coboundary} if
there exists an element $r=\sum_{i}x_{i}\otimes y_{i}\in A\otimes A$ such that
$\Delta=\Delta_{r}$, where
\begin{align}
\Delta_{r}(a)=(\id\otimes\fl_{A}(a)-\fr_{A}(a)\otimes\id)(r), \label{ass-cobo}
\end{align}
for any $a\in A$. The equation
$$
\mathbf{A}_{r}:=r_{12}\ast r_{13}+r_{13}\ast r_{23}-r_{23}\ast r_{12}=0
$$
is called the {\bf associative Yang-Baxter equation} (or, $\AYBE$) in $(A, \ast)$,
where $r_{12}\ast r_{13}=\sum_{i,j}(x_{i}\ast x_{j})\otimes y_{i}\otimes y_{j}$,
$r_{23}\ast r_{13}=\sum_{i,j}x_{j}\otimes x_{i}\otimes(y_{i}\ast y_{j})$ and
$r_{23}\ast r_{12}=\sum_{i,j}x_{j}\otimes(x_{i}\ast y_{j})\otimes y_{i}$.

\begin{pro}[\cite{Bai}]\label{pro:triASI}
Let $(A, \ast)$ be an associative algebra and $r\in A\otimes A$. Define a linear map
$\Delta_{r}: A\rightarrow A\otimes A$ by Eq. \eqref{ass-cobo}. If $r$ is a skew-symmetric
solution of the $\AYBE$ in $(A, \ast)$, then $(A, \ast, \Delta_{r})$ is an ASI bialgebra,
which is called a {\bf triangular ASI bialgebra} associated with $r$.
\end{pro}

A dendriform $\md$-bialgebra $(D, \prec, \succ, \theta_{\prec}, \theta_{\succ})$ is
called {\bf coboundary} if there exists two elements $r_{\prec}, r_{\succ}\in D\otimes D$
such that $\theta_{\prec}$ and $\theta_{\succ}$ are given by
\begin{align}
\theta_{\prec}(d)&=\theta_{\prec,r}(d):=((\fr_{\prec}+\fr_{\succ})(d)\otimes\id
-\id\otimes\fl_{\succ}(d))(r_{\prec}),               \label{d-cobo1}\\
\theta_{\succ}(d)&=\theta_{\succ,r}(d):=(\fr_{\prec}(d)\otimes\id
-\id\otimes(\fl_{\prec}+\fl_{\succ})(d))(r_{\succ}),    \label{d-cobo2}
\end{align}
for any $d\in D$. The equation
$$
\mathbf{D}_{r}:=r_{12}\prec r_{13}+r_{12}\succ r_{13}-r_{13}\prec r_{23}-r_{23}\succ r_{12}=0
$$
is called the {\bf dendriform Yang-Baxter equation} (or, $\DYBE$) in $(D, \prec, \succ)$,
where $r_{12}\prec r_{13}=\sum_{i,j}(x_{i}\prec x_{j})\otimes y_{i}\otimes y_{j}$,
$r_{12}\succ r_{13}=\sum_{i,j}(x_{i}\succ x_{j})\otimes y_{i}\otimes y_{j}$,
$r_{13}\prec r_{23}=\sum_{i,j}x_{i}\otimes x_{j}\otimes(y_{i}\prec y_{j})$ and
$r_{23}\succ r_{12}=\sum_{i,j}x_{j}\otimes(x_{i}\succ y_{j})\otimes y_{i}$ if we denote
$r=\sum_{i}x_{i}\otimes y_{i}\in D\otimes D$.

\begin{pro}[\cite{Bai}]\label{pro:tridend}
Let $(D, \prec, \succ)$ be a dendriform algebra and $r\in D\otimes D$. Define linear
maps $\theta_{\prec, r}, \theta_{\succ, r}: D\rightarrow D\otimes D$ by
Eqs. \eqref{d-cobo1} and \eqref{d-cobo2} for $r_{\prec}=r$ and $r_{\succ}=-\tau(r)$.
If $r$ is a symmetric solution of the $\DYBE$ in $(D, \prec, \succ)$, then
$(D, \prec, \succ, \theta_{\prec,r}, \theta_{\succ,r})$ is a dendriform $\md$-bialgebra,
which is called a {\bf triangular dendriform $\md$-bialgebra} associated with $r$.
\end{pro}

Triangular ASI bialgebra (resp. triangular dendriform $\md$-bialgebra) is a special
type of ASI bialgebra (resp. dendriform $\md$-bialgebra) that is related to the
solutions of the Yang-Baxter equation. Next, we consider the relation between the
solutions of the $\DYBE$ in a dendriform algebra and the solutions of the $\AYBE$
in the induced associative algebra.

\begin{pro}\label{pro:DYBE-AYBE}
Let $(D, \prec, \succ)$ be a dendriform algebra and $(B, \cdot, \omega)$ be a
quadratic perm algebra and $(D\otimes B, \ast)$ be the induced associative algebra.
Suppose that $r=\sum_{i}x_{i}\otimes y_{i}\in D\otimes D$ is a solution
of the $\DYBE$ in $(D, \prec, \succ)$. If $r$ is symmetric, then
\begin{align}
\widehat{r}:=\sum_{i, j}(x_{i}\otimes e_{j})\otimes(y_{i}\otimes f_{j})
\in(D\otimes B)\otimes(D\otimes B)  \label{assr-max}
\end{align}
is a skew-symmetric solution of the $\AYBE$ in $(D\otimes B, [-,-])$, where $\{e_{1},
e_{2},\cdots, e_{n}\}$ is a basis of $B$ and $\{f_{1}, f_{2},\cdots, f_{n}\}$
is the dual basis of $\{e_{1}, e_{2},\cdots, e_{n}\}$ with respect to $\omega(-,-)$.
\end{pro}

\begin{proof}
First, for any $1\leq p, q\leq n$, we have
\begin{align*}
&\; \widehat{r}_{12}\ast\widehat{r}_{13}+\widehat{r}_{13}\ast\widehat{r}_{23}
-\widehat{r}_{23}\ast\widehat{r}_{12}\\
=&\;\sum_{i,j}\sum_{p,q}\Big(((x_{i}, e_{p})\ast(x_{j}, e_{q}))\otimes(y_{i}, f_{p})
\otimes(y_{j}, f_{q})+(x_{i}, e_{p})\otimes(x_{j}, e_{q})\otimes
((y_{i}, f_{p})\ast(y_{j}, f_{q}))\\[-5mm]
&\qquad\qquad-(x_{j}, e_{q})\otimes((x_{i}, e_{p})
\ast(y_{j}, f_{q}))\otimes(y_{i}, f_{p})\Big)\\
=&\;\sum_{i,j}\sum_{p,q}\Big(((x_{i}\succ x_{j})\otimes(e_{p}e_{q}))\otimes(y_{i}, f_{p})
\otimes(y_{j}, f_{q})+((x_{i}\prec x_{j})\otimes(e_{q}e_{p}))\otimes(y_{i}, f_{p})
\otimes(y_{j}, f_{q})\\[-4mm]
&\qquad\qquad+(x_{i}, e_{p})\otimes(x_{j}, e_{q})\otimes
((y_{i}\succ y_{j})\otimes(f_{p}f_{q}))+(x_{i}, e_{p})\otimes(x_{j}, e_{q})\otimes
((y_{i}\prec y_{j})\otimes(f_{q}f_{p}))\\[-1mm]
&\qquad\qquad-(x_{j}, e_{q})\otimes((x_{i}\succ y_{j})\otimes(e_{p}f_{q}))\otimes(y_{i}, f_{p})
-(x_{j}, e_{q})\otimes((x_{i}\prec y_{j})\otimes(f_{q}e_{p}))\otimes(y_{i}, f_{p})\Big).
\end{align*}
Moreover, for given $s, u, v\in\{1, 2,\cdots, n\}$, we have
$$
\omega\Big(e_{s}\otimes e_{u}\otimes e_{v},\;
\sum_{p,q}(e_{q}e_{p})\otimes f_{p}\otimes f_{q}\Big)
=\omega(e_{s},\; e_{v}e_{u})=\omega\Big(e_{s}\otimes e_{u}\otimes e_{v},\;
\sum_{p,q} e_{q}\otimes(e_{p}f_{q})\otimes f_{p}\Big).
$$
By the nondegeneracy of $\omega(-,-)$, we get $\sum_{p,q}(e_{q}e_{p})\otimes f_{p}
\otimes f_{q}=\sum_{p,q} e_{q}\otimes(e_{p}f_{q})\otimes f_{p}$. Similarly, we also have
$\sum_{p,q} e_{q}\otimes e_{p}\otimes(f_{p}f_{q})=-\sum_{p,q}(e_{p}e_{q})\otimes f_{p}
\otimes f_{q}$ and $\sum_{p,q} e_{q}\otimes(f_{q}e_{p})\otimes f_{p}=
\sum_{p,q} e_{q}\otimes e_{p}\otimes(f_{q}f_{p})=\sum_{p,q}(e_{q}e_{p})\otimes f_{p}
\otimes f_{q}-\sum_{p,q}(e_{p}e_{q})\otimes f_{p}\otimes f_{q}$.
Thus, we obtain
\begin{align*}
&\; \widehat{r}_{12}\ast\widehat{r}_{13}+\widehat{r}_{13}\ast\widehat{r}_{23}
-\widehat{r}_{23}\ast\widehat{r}_{12}\\
=&\;\sum_{i,j}\sum_{p,q}\Big((x_{i}\succ x_{j})\otimes y_{i}\otimes y_{j}
-x_{i}\otimes x_{j}\otimes(y_{i}\succ y_{j})
-x_{i}\otimes x_{j}\otimes(y_{i}\prec y_{j})\\[-5mm]
&\qquad\qquad\qquad\qquad\qquad\qquad+x_{j}\otimes(x_{i}\prec y_{j})\otimes y_{i}
\Big)\bullet\big((e_{p}e_{q})\otimes f_{p}\otimes f_{q}\big)\\[-2mm]
&\quad+\Big((x_{i}\prec x_{j})\otimes y_{i}\otimes y_{j}
+x_{i}\otimes x_{j}\otimes(y_{i}\succ y_{j})
-x_{j}\otimes(x_{i}\succ y_{j})\otimes y_{i}\\[-2mm]
&\qquad\qquad\qquad\qquad\qquad\qquad-x_{j}\otimes(x_{i}\prec y_{j})\otimes y_{i}
\Big)\bullet\big((e_{q}e_{p})\otimes f_{p}\otimes f_{q}\big).
\end{align*}
Since $r$ is symmetric and $\mathbf{D}_{r}=0$, we get
\begin{align*}
\sum_{i,j}\Big((x_{i}\succ x_{j})\otimes y_{i}\otimes y_{j}
-x_{i}\otimes x_{j}\otimes(y_{i}\succ y_{j})-x_{i}\otimes x_{j}\otimes(y_{i}\prec y_{j})
+x_{j}\otimes(x_{i}\prec y_{j})\otimes y_{i}\Big)=0,\\[-2mm]
\sum_{i,j}\Big((x_{i}\prec x_{j})\otimes y_{i}\otimes y_{j}
+x_{i}\otimes x_{j}\otimes(y_{i}\succ y_{j})-x_{j}\otimes(x_{i}\succ y_{j})\otimes y_{i}
-x_{j}\otimes(x_{i}\prec y_{j})\otimes y_{i}\Big)=0.
\end{align*}
Therefore, we get $\widehat{r}_{12}\ast\widehat{r}_{13}
+\widehat{r}_{13}\ast\widehat{r}_{23}-\widehat{r}_{23}\ast\widehat{r}_{12}=0$,
$\widehat{r}$ is a solution of the $\AYBE$ in $(D\otimes B, [-,-])$.
Finally, note that $\sum_{j}e_{j}\otimes f_{j}=-\sum_{j}f_{j}\otimes e_{j}$
(see the proof of Corollary \ref{cor:sPLYBE-sCYBE}). We obtain that $\widehat{r}$
is skew-symmetric if $r$ is symmetric. The proof is finished.
\end{proof}

\begin{rmk}\label{rmk:compASI-Lie}
Similar to \cite{HBG}, we can define the $\AYBE$ in a $\bz$-graded associative algebra.
More precisely, let $(A=\oplus_{i\in\bz}A_{i}, \ast)$ be a $\bz$-graded associative algebra
and $r=\sum_{i,j,\alpha}x_{i,\alpha}\otimes y_{j,\alpha}$. The equation $r_{12}\ast r_{13}
+r_{13}\ast r_{23}-r_{23}\ast r_{12}=0$ is called the $\AYBE$ in $(A=\oplus_{i\in\bz}A_{i},
\ast)$, where $r_{12}\ast r_{13}=\sum_{i,j,k,l,\alpha,\beta}(x_{i,\alpha}\ast x_{k,\beta})
\otimes y_{j,\alpha}\otimes y_{l,\beta}$, $r_{23}\ast r_{13}=\sum_{i,j,k,l,\alpha,\beta}
x_{k,\beta}\otimes x_{i,\alpha}\otimes(y_{j,\alpha}\ast y_{l,\beta})$ and
$r_{23}\ast r_{12}=\sum_{i,j,k,l,\alpha,\beta}x_{k,\beta}\otimes(x_{i,\alpha}\ast
y_{l,\beta})\otimes y_{j,\alpha}$. One can check that a skew-symmetric solution of
the $\AYBE$ in $(A=\oplus_{i\in\bz}A_{i}, \ast)$ induced a completed ASI bialgebra
structure on $A=\oplus_{i\in\bz}A_{i}$. In Proposition \ref{pro:DYBE-AYBE},
if we we replace the quadratic perm algebra $(B, \cdot, \omega)$ with a quadratic
$\bz$-graded perm algebra $(B=\oplus_{i\in\bz}B_{i}, \cdot, \varpi)$ and define
$\widehat{r}:=\sum_{p\in\Phi}\sum_{i}(x_{i}\otimes e_{p})\otimes(y_{i}\otimes f_{p})
\in(D\otimes B)\,\hat{\otimes}\,(D\otimes B)$,
where $\{e_{p}\}_{p\in\Phi}$ is a homogeneous basis of $B$ and $\{f_{p}\}_{p\in\prod}$
is its homogeneous dual basis associated with $\varpi(-,-)$, then one can check that
$\widehat{r}$ is a skew-symmetric solution of the $\AYBE$ in $\bz$-graded associative algebra
$(D\otimes B, \ast)$ if $r$ is a symmetric solution of the $\DYBE$ in $(D, \prec, \succ)$.
\end{rmk}

\begin{pro}\label{pro:spre-ASIbia}
Let $(D, \prec, \succ, \theta_{\prec}, \theta_{\succ})$ be a dendriform $\md$-bialgebra,
$(B, \cdot, \omega)$ be a quadratic perm algebra and $(D\otimes B, \ast)$ be the induced
associative algebra. If $r\in D\otimes D$ is symmetric, $\theta_{\prec}=\theta_{\prec,r}$
and $\theta_{\succ}=\theta_{\succ,r}$ are defined by Eqs. \eqref{d-cobo1} and
\eqref{d-cobo2} for $r_{\prec}=r$ and $r_{\succ}=-\tau(r)$, then $(D\otimes B, \ast,
\Delta)=(D\otimes B, \ast, \Delta_{\widehat{r}})$ as ASI bialgebras, where $\Delta$
is defined by Eq. \eqref{ind-coass} for $\theta_{\prec}=\theta_{\prec,r}$ and
$\theta_{\succ}=\theta_{\succ,r}$, $\widehat{r}$ is given by Eq. \eqref{assr-max}.

Therefore, we obtain that $(D\otimes B, \ast, \Delta)$ is a triangular ASI bialgebra
if $(D, \prec, \succ, \theta_{\prec,r}, \theta_{\succ,r})$ is a triangular
dendriform $\md$-bialgebra.
\end{pro}

\begin{proof}
Let $r=\sum_{i}x_{i}\otimes y_{i}$. For any $b\in B$ and $d\in D$, we have
\begin{align*}
\Delta(d\otimes b)&=\theta_{\succ,r}(d)\bullet\nu_{\omega}(b)
+\theta_{\prec,r}(d)\bullet\tau(\nu_{\omega}(b)) \\
&=\sum_{i}\sum_{(b)}\Big((x_{i}\otimes b_{(1)})\otimes((d\prec y_{i})\otimes b_{(2)})
+(x_{i}\otimes b_{(1)})\otimes((d\succ y_{i})\otimes b_{(2)})\\[-4mm]
&\qquad\qquad\quad-((x_{i}\prec d)\otimes b_{(1)})\otimes(y_{i}\otimes b_{(2)})
-(x_{i}\otimes b_{(2)})\otimes((d\succ y_{i})\otimes b_{(1)})\\[-1mm]
&\qquad\qquad\quad+((x_{i}\prec d)\otimes b_{(2)})\otimes(y_{i}\otimes b_{(1)})
+((x_{i}\succ d)\otimes b_{(2)})\otimes(y_{i}\otimes b_{(1)})\Big),
\end{align*}
and
\begin{align*}
\Delta_{\widehat{r}}(d\otimes b)
&=(\id\otimes\fl_{A}(d\otimes b)-\fr_{A}(d\otimes b)\otimes\id)\Big(
\sum_{i,j}(x_{i}\otimes e_{j})\otimes(y_{i}\otimes f_{j})\Big)\\[-2mm]
&=\sum_{i,j}\Big((x_{i}\otimes e_{j})\otimes((d\succ y_{i})\otimes(bf_{j}))
+(x_{i}\otimes e_{j})\otimes((d\prec y_{i})\otimes(f_{j}b))\\[-5mm]
&\qquad\quad-((x_{i}\succ d)\otimes(e_{j}b))\otimes(y_{i}\otimes f_{j})
-((x_{i}\prec d)\otimes(be_{j}))\otimes(y_{i}\otimes f_{j})\Big),
\end{align*}
where $\widehat{r}=\sum_{i, j}(x_{i}\otimes e_{j})\otimes(y_{i}\otimes f_{j})$,
$\{e_{1}, e_{2},\cdots, e_{n}\}$ is a basis of $B$ and $\{f_{1}, f_{2},\cdots, f_{n}\}$
is the dual basis of $\{e_{1}, e_{2},\cdots, e_{n}\}$ with respect to $\omega(-,-)$.
Similar to the proof of Theorem \ref{thm:indu-sLiebia}, by using the nondegeneracy
of $\omega(-,-)$, we can obtain $\sum_{(b)}b_{(1)}\otimes b_{(2)}
=\sum_{j}e_{j}\otimes(f_{j}b)$, $\sum_{(b)}b_{(2)}\otimes b_{(1)}
=-\sum_{j}(e_{j}b)\otimes f_{j}$ and $\sum_{j}(be_{j})\otimes f_{j}
=\sum_{j}e_{j}\otimes (bf_{j})=\sum_{(b)}\big(b_{(1)}\otimes b_{(2)}
-b_{(2)}\otimes b_{(1)}\big)$. Thus, $\Delta(d\otimes b)=\Delta_{\widehat{r}}
(d\otimes b)$ and so that $(D\otimes B, \ast, \Delta)
=(D\otimes B, \ast, \Delta_{\widehat{r}})$ as ASI bialgebras.
Finally, following from Proposition \ref{pro:DYBE-AYBE}, we obtain that
$(D\otimes B, \ast, \Delta)$ is a triangular ASI bialgebra if $(D, \prec, \succ,
\theta_{\prec,r}, \theta_{\succ,r})$ is a triangular dendriform $\md$-bialgebra.
\end{proof}

Let $(D, \prec, \succ)$ be a dendriform algebra and $(V, \kl_{\prec}, \kr_{\prec},
\kl_{\succ}, \kr_{\succ})$ be a bomodule over it.
A linear map $P: V\rightarrow D$ is called an {\bf $\mathcal{O}$-operator of
$(D, \prec, \succ)$ associated to $(V, \kl_{\prec}, \kr_{\prec}, \kl_{\succ},
\kr_{\succ})$} if for any $d_{1}, d_{2}\in D$,
\begin{align*}
P(d_{1})\prec P(d_{2})&=P\big(\kl_{\prec}(P(d_{1}))(d_{2})
+\kr_{\prec}(P(d_{2}))(d_{1})\big),\\
P(d_{1})\succ P(d_{2})&=P\big(\kl_{\succ}(P(d_{1}))(d_{2})
+\kr_{\succ}(P(d_{2}))(d_{1})\big).
\end{align*}

\begin{pro}[\cite{Bai}]\label{pro:o-dend}
Let $(D, \prec, \succ)$ be a dendriform algebra, $r\in\g\otimes\g$ be symmetric.
Then $r$ is a solution of the $\DYBE$ in $(D, \prec, \succ)$ if and only if $r^{\sharp}$
is an $\mathcal{O}$-operator of $(D, \prec, \succ)$ associated to the coregular
bimodule $(D^{\ast}, \fr_{\succ}^{\ast}+\fr_{\prec}^{\ast}, -\fl_{\prec}^{\ast},
-\fr_{\succ}^{\ast}, \fl_{\prec}^{\ast}+\fl_{\succ}^{\ast})$.
\end{pro}

An {\bf $\mathcal{O}$-operator} of an associative algebra $(A, \ast)$ associated to
$(V, \kl_{A}, \kr_{A})$ is a linear map $P: V\rightarrow A$ such that
$$
P(a_{1})\ast P(a_{2})=P\big(\kl_{A}(P(a_{1}))(a_{2})+\kr_{A}(P(a_{2}))(a_{1})\big),
$$
for any $a_{1}, a_{2}\in A$,

\begin{pro}[\cite{Bai}]\label{pro:o-ass}
Let $(A, \ast)$ be an associative algebra, $r\in A\otimes A$ be skew-symmetric.
Then $r$ is a solution of the $\AYBE$ in $(A, \ast)$ if and only if $r^{\sharp}$
is an $\mathcal{O}$-operator of $(A, \ast)$ associated to the coregular
bimodule $(A^{\ast}, \fr_{A}^{\ast}, \fl_{A}^{\ast})$.
\end{pro}

Similar to the proof of Theorem \ref{thm:indu-sLiebia}, one can check that
$\widehat{r}^{\sharp}: (D\otimes B)^{\ast}\rightarrow D\otimes B$ equals
$r^{\sharp}\otimes\kappa^{\sharp}$. Thus, we also have the following commutative diagram:
$$
\xymatrix@C=1.4cm@R=0.5cm{
\txt{$(D, \prec, \succ, \theta_{\prec,r}, \theta_{\succ,r})$ \\ {\tiny a triangular
dendriform $\md$-bialgebra}}
\ar[d]^{{\rm Cro.}~\ref{cor:indassbia}}_{{\rm Pro.}~\ref{pro:spre-ASIbia}} &
\txt{$r$ \\ {\tiny a symmetric solution} \\ {\tiny of the $\DYBE$ in $(D, \prec, \succ)$}}
\ar[d]_-{{\rm Pro.}~\ref{pro:DYBE-AYBE}}\ar[r]^-{{\rm Pro.}~\ref{pro:o-dend}}
\ar[l]_-{{\rm Pro.}~\ref{pro:tridend}}    &
\txt{$r^{\sharp}$\\ {\tiny an $\mathcal{O}$-operator of $(D, \prec, \succ)$} \\
{\tiny associated to $(D^{\ast}, \fr_{\succ}^{\ast}+\fr_{\prec}^{\ast},
-\fl_{\prec}^{\ast}, -\fr_{\succ}^{\ast}, \fl_{\prec}^{\ast}+\fl_{\succ}^{\ast})$}}
\ar[d]^-{\mbox{$-\otimes\kappa^{\sharp}$}} \\
\txt{$(D\otimes B, \ast, \Delta_{\widehat{r}})$ \\
{\tiny a triangular ASI bialgebra}}   &
\txt{$\widehat{r}$ \\ {\tiny a skew-symmetric solution} \\ {\tiny of the $\AYBE$ in
$(D\otimes B, \ast)$}} \ar[r]^-{{\rm Pro.}~\ref{pro:o-ass}}
\ar[l]_-{{\rm Pro.}~\ref{pro:triASI}}   &
\txt{$\widehat{r}^{\sharp}=r^{\sharp}\otimes\kappa^{\sharp}$ \\
{\tiny an $\mathcal{O}$-operator of $(D\otimes B, \ast)$ } \\
{\tiny associated to $((D\otimes B)^{\ast}, \fr^{\ast}_{D\otimes B},
\fl^{\ast}_{D\otimes B})$}}}
$$

\begin{ex}\label{ex:dendind-ASI}
We consider a 2-dimensional dendriform algebra $(D=\Bbbk\{e_{1}, e_{2}\}, \prec, \succ)$,
where the nonzero products are given by $e_{1}\succ e_{1}=e_{1}$ and $e_{2}\prec e_{1}=e_{2}$.
Then one can check that symmetric solutions of the $\DYBE$ in $(D, \prec, \succ)$
are $\alpha e_{1}\otimes e_{1}$ or $\beta(e_{1}\otimes e_{2}+e_{2}\otimes e_{1})
+\gamma e_{2}\otimes e_{2}$, $\alpha, \beta, \gamma\in\Bbbk$. Let $r=e_{1}\otimes e_{1}$.
Then we get a triangular dendriform $\md$-bialgebra $(D, \prec, \succ, \theta_{\prec},
\theta_{\succ})$ associated with $r$, where $\theta_{\succ}(e_{1})=e_{1}\otimes e_{1}$,
$\theta_{\succ}(e_{2})=e_{1}\otimes e_{2}$ and others are all zero.
Let $(B=\Bbbk\{x_{1}, x_{2}\}, \cdot, \omega)$ be the quadratic perm algebra given in
Example \ref{ex:prelie-indlie}. Then by Corollary \ref{cor:indassbia}, we obtain an
ASI bialgebra $(D\otimes B, \ast, \Delta)$, where the nonzero products and coproducts
are given by
\begin{align*}
&\;y_{2}y_{1}=y_{1},\qquad\qquad\quad y_{2}y_{2}=y_{2},\qquad\qquad\quad
y_{3}y_{2}=y_{3},\qquad\qquad\quad y_{4}y_{2}=y_{4},\\
&\Delta(y_{1})=-y_{1}\otimes y_{1},\qquad \Delta(y_{2})=-y_{1}\otimes y_{2},\qquad
\Delta(y_{3})=-y_{1}\otimes y_{3},\qquad \Delta(y_{4})=-y_{1}\otimes y_{4},
\end{align*}
$y_{1}:=e_{1}\otimes x_{1}$, $y_{2}:=e_{1}\otimes x_{2}$, $y_{3}:=e_{2}\otimes x_{1}$
and $y_{4}:=e_{2}\otimes x_{2}$. On the other hand, by Proposition \ref{pro:DYBE-AYBE},
we get a skew-symmetric solution $\widehat{r}=y_{1}\otimes y_{2}-y_{2}\otimes y_{1}$
of the $\CYBE$ in associative algebra $(D\otimes B, \ast)$. One can check that the
triangular ASI bialgebra associated with $\widehat{r}$ is just the ASI bialgebra
$(D\otimes B, \ast, \Delta)$ given above. Moreover, for the $\mathcal{O}$-operators
$r^{\sharp}$ and $\widehat{r}^{\sharp}$, similar to Example \ref{ex:prelie-indlie},
we also have $\widehat{r}^{\sharp}=r^{\sharp}\otimes\kappa^{\sharp}$.
\end{ex}

\section{Triangular Lie bialgebras vis dendriform $\md$-bialgebras} \label{sec:dend-lie}

In this section, we show that the commutative graph in Proposition \ref{pro:comm-diag}
is correct at the level of bialgebras. Moreover, by considering some special
solutions of the Yang-Baxter equations in dendriform algebras, pre-Lie algebras,
associative algebras and Lie algebras, show that the commutative graph is also
correct at the level of triangular bialgebras.

Given an associative algebra, we can get a Lie algebra by the commutator.
Considering the dual case, for a coassociative coalgebra $(A, \Delta)$,
we also have a Lie coalgebra structure on $A$ by setting a comultiplication
$\delta=\Delta-\tau\Delta: A\rightarrow A\otimes A$. Similarly, for any
dendriform coalgebra $(D, \theta_{\prec}, \theta_{\succ})$, there is a pre-Lie
coalgebra $(D, \vartheta)$, where $\vartheta=\theta_{\succ}-\tau\theta_{\prec}$.
Furthermore, for ASI bialgebras and dendriform $\md$-bialgebras, we have
the following proposition.

\begin{pro}[{\cite[Corollary 5.3.4]{Bai}}]\label{pro:ASI-Liebia}
With the above notations, we have
\begin{enumerate}
\item[$(i)$] if $(A, \ast, \Delta)$ is an ASI bialgebra, then $(A, [-,-], \delta)$
     is a Lie bialgebra, which is called the {\bf Lie bialgebra induced by $(A, \ast,
     \Delta)$}, where $(A, [-,-])$ is the induced Lie algebra from $(A, \ast)$;
\item[$(ii)$]  if $(D, \prec, \succ, \theta_{\prec}, \theta_{\succ})$ is a dendriform
     $\md$-bialgebra, then $(D, \diamond, \delta)$ is a pre-Lie bialgebra, which is called
     the {\bf pre-Lie bialgebra induced by $(D, \prec, \succ, \theta_{\prec},
     \theta_{\succ})$}, where $(D, \diamond)$ is the induced pre-Lie algebra from
     $(D, \prec, \succ)$.
\end{enumerate}
\end{pro}

\begin{rmk}\label{rmk:comp-AL}
In fact, by equation $\delta=\Delta-\tau\Delta$, a completed ASI bialgebra also
induces a completed Lie bialgebra.
\end{rmk}

Thus, for dendriform $\md$-bialgebras, pre-Lie bialgebras, ASI bialgebras and
Lie bialgebras, we have the following theorem.

\begin{thm}\label{thm:bialgebras}
Let $(D, \prec, \succ, \theta_{\prec}, \theta_{\succ})$ be a dendriform $\md$-bialgebra
and $(B, \cdot, \omega)$ be a quadratic perm algebra. Then we have the following
commutative diagram:
$$
\xymatrix@C=3cm@R=0.5cm{
\txt{$(D, \prec, \succ, \theta_{\prec}, \theta_{\succ})$ \\ {\tiny a dendriform bialgebra}}
\ar[d]_-{\mbox{\tiny $d_{1}\diamond d_{2}=d_{1}\succ d_{2}-d_{2}\prec d_{1}$}}
\ar[r]^-{\mbox{\tiny $-\otimes(B, \circ, \omega)$}}
&\txt{$(D\otimes B, \ast, \Delta)$\\ {\tiny an associative bialgebra}}
\ar[d]^-{\mbox{\tiny $[x_{1}, x_{2}]=x_{1}x_{2}-x_{2}x_{1}$}} \\
\txt{$(D, \diamond, \vartheta)$ \\ {\tiny a pre-Lie bialgebra}}
\ar[r]^-{\mbox{\tiny $-\otimes(B, \circ, \omega)$}}
& \txt{$(D\otimes B, [-,-], \delta)$ \\
{\tiny a Lie bialgebra}}}
$$
\end{thm}

\begin{proof}
On the one hand, by Proposition \ref{pro:ASI-Liebia} and Theorem \ref{thm:pre-lie},
we can get a Lie bialgebra $(D\otimes B, [-,-]_{1}, \delta_{1})$ form a dendriform
$\md$-bialgebra $(D, \prec, \succ, \theta_{\prec}, \theta_{\succ})$ by:
$$
\xymatrix@C=2cm@R=0.6cm{
\txt{$(D, \prec, \succ, \theta_{\prec}, \theta_{\succ})$ \\ {\tiny a dendriform bialgebra}}
\ar[r]^-{\mbox{\tiny Pro. \ref{pro:ASI-Liebia}}}
&\txt{$(D, \diamond, \vartheta)$ \\ {\tiny a pre-Lie bialgebra}}
\ar[r]^-{\mbox{\tiny Thm. \ref{thm:pre-lie}}}
& \txt{$(D\otimes B, [-,-]_{1}, \delta_{1})$. \\
{\tiny a Lie bialgebra}}}
$$
On the other hand, by Corollary \ref{cor:indassbia} and Proposition \ref{pro:ASI-Liebia},
we can get a Lie bialgebra $(D\otimes B, [-,-]_{2}, \delta_{2})$ form a dendriform
$\md$-bialgebra $(D, \prec, \succ, \theta_{\prec}, \theta_{\succ})$ by:
$$
\xymatrix@C=2cm@R=0.6cm{
\txt{$(D, \prec, \succ, \theta_{\prec}, \theta_{\succ})$ \\ {\tiny a dendriform bialgebra}}
\ar[r]^-{\mbox{\tiny Cor. \ref{cor:indassbia}}}
&\txt{$(D\otimes B, \ast, \Delta)$ \\ {\tiny an ASI bialgebra}}
\ar[r]^-{\mbox{\tiny Pro. \ref{pro:ASI-Liebia}}}
& \txt{$(D\otimes B, [-,-]_{2}, \delta_{2})$. \\
{\tiny a Lie bialgebra}}}
$$
In Proposition \ref{pro:comm-diag}, we have obtained $[-,-]_{1}=[-,-]_{2}$.
Now we need to show $\delta_{1}=\delta_{2}$. For any $b\in B$ and $d\in D$,
\begin{align*}
\delta_{1}(d\otimes b)
&=(\id\otimes\id-\tau)\big((\theta_{\succ}-\tau\theta_{\prec})(d)\bullet\nu_{\omega}(b)\big)\\
&=\sum_{[d]}\sum_{(b)}(d_{[1]}\otimes b_{(1)})\otimes(d_{[2]}\otimes b_{(2)})
-\sum_{(d)}\sum_{(b)}(d_{(2)}\otimes b_{(1)})\otimes(d_{(1)}\otimes b_{(2)})\\[-3mm]
&\qquad\qquad-\sum_{[d]}\sum_{(b)}(d_{[2]}\otimes b_{(2)})\otimes(d_{[1]}\otimes b_{(1)})
+\sum_{(d)}\sum_{(b)}(d_{(1)}\otimes b_{(2)})\otimes(d_{(2)}\otimes b_{(1)})\\[-2mm]
&=\Delta(d\otimes b)-\tau(\Delta(d\otimes b))\\
&=\delta_{2}(d\otimes b),
\end{align*}
where $\theta_{\prec}(d)=\sum_{(d)}d_{(1)}\otimes d_{(2)}$, $\theta_{\succ}(d)
=\sum_{[d]}d_{[1]}\otimes d_{[2]}$ and $\nu_{\omega}(b)=\sum_{(b)}b_{(1)}\otimes b_{(2)}$
in the Sweedler notation. Thus, $(D\otimes B, [-,-]_{1}, \delta_{1})=
(D\otimes B, [-,-]_{2}, \delta_{2})$ as Lie bialgebras, and so that the diagram
is commutative.
\end{proof}

\begin{ex}\label{ex:com-diag}
Let $(D, \prec, \succ, \theta_{\prec}, \theta_{\succ})$ be the triangular dendriform
$\md$-bialgebra given in Example \ref{ex:dendind-ASI} and $(B, \cdot, \omega)$ be
the quadratic perm algebra given in Example \ref{ex:prelie-indlie}. By Corollary
\ref{cor:indassbia}, we get an ASI bialgebra $(D\otimes B, \ast, \Delta)$ form
$(D, \prec, \succ, \theta_{\prec}, \theta_{\succ})$ by $(B, \cdot, \omega)$, which has
given in Example \ref{ex:dendind-ASI}. By Proposition \ref{pro:ASI-Liebia}, we get
a Lie bialgebra $(D\otimes B, [-,-], \delta)$, where the bracket $[-,-]$ and coproduct
$\delta$ are just the bracket and the coproduct given in Example \ref{ex:prelie-indlie}.

On the other hand, note that the pre-Lie bialgebra $(D, \diamond, \vartheta)$ induced by
$(D, \prec, \succ, \theta_{\prec}, \theta_{\succ})$ is just the pre-Lie bialgebra
$(A, \diamond, \vartheta)$ in Example \ref{ex:prelie-indlie}, and the Lie bialgebra
$(D\otimes B, [-,-], \delta)$ form $(D, \diamond, \vartheta)$ by $(B, \cdot, \omega)$
just the Lie bialgebra $(A\otimes B, [-,-], \delta)$ in Example \ref{ex:prelie-indlie}.
\end{ex}

Following, we consider some special bialgebras, such as quasi-triangular bialgebras,
triangular bialgebras and factorizable bialgebras. Let $(A, \ast)$ be an associative
algebra. Recall that an element $r=\sum_{i}x_{i}\otimes y_{i}\in A\otimes A$ is said
to be {\bf $(\fl_{A}, \fr_{A})$-invariant} if for any $a\in A$,
$$
(\id\otimes\fl_{A}(a)-\fr_{A}(a)\otimes\id)(r)=0.
$$

\begin{pro}[\cite{SW}]\label{pro:qtass-bia}
Let $(A, \ast)$ be an associative algebra, $r\in A\otimes A$ and
$\Delta_{r}: A\rightarrow A\otimes A$ is defined by Eq. \eqref{ass-cobo}.
If $r$ is a solution of the $\AYBE$ in $(A, \ast)$ and $r+\tau(r)$ is
$(\fl_{A}, \fr_{A})$-invariant, then $(A, \ast, \Delta_{r})$ is an ASI bialgebra,
which is called a {\bf quasi-triangular ASI bialgebra} associated with $r$.
\end{pro}

\begin{defi}[\cite{SW}]\label{def:fact-assbia}
Let $(A, \ast)$ be an associative algebra, $r\in A\otimes A$ and $(A, \ast, \Delta_{r})$
be a quasi-triangular ASI bialgebra associated with $r$. If $\mathcal{I}=
r^{\sharp}+\tau(r)^{\sharp}: A^{\ast}\rightarrow A$ is an isomorphism of vector spaces,
then $(A, \ast, \Delta_{r})$ is called a {\bf factorizable ASI bialgebra}.
\end{defi}

These ASI bialgebras are closely related to the solutions of $\AYBE$. Following,
we will examine the relationship between the solutions of $\AYBE$ and the solutions
of $\CYBE$ in the induced Lie algebra to provide the connection between these special
ASI bialgebras and the induced Lie bialgebras.

\begin{pro}\label{pro:ass-Lie-YBE}
Let $(A, \ast)$ be an associative algebra and $(A, [-,-])$ be the induced Lie algebra
from $(A, \ast)$. Suppose $r=\sum_{i}x_{i}\otimes y_{i}\in A\otimes A$ is a solution
of the $\AYBE$ in $(A, \ast)$. If $r+\tau(r)$ is $(\fl_{A}, \fr_{A})$-invariant, then
$r$ is a solution of the $\CYBE$ in $(A, [-,-])$ and $r+\tau(r)$ is $\ad_{A}$-invariant.

In particular, each skew-symmetric solution of the $\AYBE$ in $(A, \ast)$ is also a
skew-symmetric solution of the $\CYBE$ in $(A, [-,-])$.
\end{pro}

\begin{proof}
First, note that $\mathbf{C}_{r}:=[r_{12}, r_{13}]+[r_{13}, r_{23}]+[r_{12}, r_{23}]
=r_{12}\ast r_{13}-r_{13}\ast r_{12}+r_{13}\ast r_{23}-r_{23}\ast r_{13}
+r_{12}\ast r_{23}-r_{23}\ast r_{12}$. If $r+\tau(r)$ is $(\fl, \fr)$-invariant,
That is, $\sum_{i}x_{i}\otimes(a\ast y_{i})+y_{i}\otimes(a\ast x_{i})
=\sum_{i}(x_{i}\ast a)\otimes y_{i}+(y_{i}\ast a)\otimes x_{i}$ for any $a\in A$. Thus,
\begin{align*}
\mathbf{C}_{r}
=&\;\mathbf{A}_{r}-\sum_{i}\Big((x_{i}\ast x_{j})\otimes y_{j}\otimes y_{i}
+x_{i}\otimes x_{j}\otimes(y_{j}\ast y_{i})
-x_{j}\otimes(y_{j}\ast x_{i})\otimes y_{i}\Big)\\[-2mm]
=&\;\mathbf{A}_{r}-\sum_{i}\Big((x_{i}\ast x_{j})\otimes y_{j}\otimes y_{i}
+x_{i}\otimes(y_{i}\ast y_{j})\otimes x_{j}-x_{i}\otimes y_{j}\otimes(x_{j}\ast y_{i})\Big)\\
=&\;\mathbf{A}_{r}-(\id\otimes\tau)(\mathbf{A}_{r}).
\end{align*}
That is, $r$ is a solution of the $\CYBE$ in $(A, [-,-])$ if $r$ is a
solution of the $\AYBE$ in $(A, \ast)$ and $r+\tau(r)$ is $(\fl_{A}, \fr_{A})$-invariant.
Moreover, note that for any $a\in A$,
\begin{align*}
&\;(\id\otimes\ad_{A}(a)+\ad_{A}(a)\otimes\id)(r+\tau(r))\\
=&\;\sum_{i}\Big((a\ast x_{i})\otimes y_{i}-(x_{i}\ast a)\otimes y_{i}
+x_{i}\otimes(a\ast y_{i})-x_{i}\otimes(y_{i}\ast a)\\[-5mm]
&\qquad\quad+(a\ast y_{i})\otimes x_{i}-(y_{i}\ast a)\otimes x_{i}
+y_{i}\otimes(a\ast x_{i})-y_{i}\otimes(x_{i}\ast a)\Big)\\
=&\;0,
\end{align*}
since $r+\tau(r)$ is $(\fl_{A}, \fr_{A})$-invariant.
That is, $r+\tau(r)$ is $\ad_{A}$-invariant.

In particular, if $r$ is a skew-symmetric solution of the $\AYBE$ in $(A, \ast)$,
then $r+\tau(r)=0$ is $(\fl_{A}, \fr_{A})$-invariant, and so that $r$ is a
skew-symmetric solution of the $\CYBE$ in $(A, [-,-])$.
\end{proof}

\begin{pro}\label{pro:qtAss-qtLie}
Let $(A, \ast)$ be an associative algebra and $r\in A\otimes A$.
Suppose $(A, \ast, \Delta_{r})$ is an ASI bialgebra and $(A, [-,-], \delta)$ is the
induced Lie bialgebra from $(A, \ast, \Delta_{r})$, where $\Delta_{r}$ is given by
Eq. \eqref{ass-cobo}. If $r+\tau(r)$ is $(\fl_{A}, \fr_{A})$-invariant, then
$(A, [-,-], \delta)=(A, [-,-], \delta_{r})$ as Lie bialgebras, where
$\delta_{r}$ is defined Eq. \eqref{lie-cobo}. Thus, we have
\begin{enumerate}\itemsep=0pt
\item[$(i)$]   $(A, [-,-], \delta)$ is quasi-triangular if $(A, \ast, \Delta)$ is
               quasi-triangular;
\item[$(ii)$]  $(A, [-,-], \delta)$ is triangular if $(A, \ast, \Delta)$ is triangular;
\item[$(iii)$] $(A, [-,-], \delta)$ is factorizable if $(A, \ast, \Delta)$ is factorizable.
\end{enumerate}
\end{pro}

\begin{proof}
Let $r=\sum_{i}x_{i}\otimes y_{i}$. If $r+\tau(r)$ is $(\fl_{A}, \fr_{A})$-invariant,
that is, for any $a\in A$,
$$
\sum_{i}\Big((a\ast x_{i})\otimes y_{i}-x_{i}\otimes(y_{i}\ast a)\Big)
=\sum_{i}\Big(y_{i}\otimes(x_{i}\ast a)-(a\ast y_{i})\otimes x_{i}\Big).
$$
Therefore, we have
\begin{align*}
\delta(a)&=\Delta_{r}(a)-\tau(\Delta_{r}(a))\\
&=\sum_{i}\Big(x_{i}\otimes(a\ast y_{i})-x_{i}\otimes(y_{i}\ast a)
+(a\ast x_{i})\otimes y_{i}-(x_{i}\ast a)\otimes y_{i}\Big)\\
&=(\id\otimes\ad_{A}(a)+\ad_{A}(a)\otimes\id)(r)\\
&=\delta_{r}(a),
\end{align*}
for any $a\in A$. Thus, $(A, [-,-], \delta)=(A, [-,-], \delta_{r})$ as Lie bialgebras.
Finally, by Proposition \ref{pro:ass-Lie-YBE}, we get $(i)$ and $(ii)$.
Note that the map $\mathcal{I}$ for the associative algebra $(A, \ast)$ and the
map $\mathcal{I}$ for the Lie algebra $(A, [-,-])$ are the same. We get $(iii)$ holds.
\end{proof}

Let $(A, \ast)$ be an associative algebra and $r\in A\otimes A$. If $r$ is a solution of
the $\AYBE$ in $(A, \ast)$ such that $r+\tau(r)$ is $(\fl_{A},\ad_{A})$-invariant,
then Proposition \ref{pro:qtAss-qtLie} also gives the following commutative diagram:
$$
\xymatrix@C=3cm@R=0.5cm{
\txt{$r$ \\ {\tiny a solution of the $\AYBE$}\\ {\tiny in $(A, \ast)$}}
\ar[d]_{{\rm Pro.}~\ref{pro:ass-Lie-YBE}}\ar[r]^{{\rm Pro.}~\ref{pro:qtass-bia}} &
\txt{$(A, \ast, \Delta_{r})$ \\ {\tiny an ASI bialgebra}}
\ar[d]^{{\rm Pro.}~\ref{pro:ASI-Liebia}}_{{\rm Pro.}~\ref{pro:qtAss-qtLie}} \\
\txt{$r$ \\ {\tiny a solution of the $\CYBE$}\\ {\tiny in $(A, [-,-])$}}
\ar[r]^{{\rm Pro.}~\ref{pro:splie-bia}\qquad\quad} &
\txt{$(A, [-,-], \delta)=(A\otimes B, [-,-], \delta_{r})$ \\
{\tiny the induced Lie bialgebra}}}
$$

\begin{pro}\label{pro:Ass-ind-ooper}
Let $(A, \ast)$ be an associative algebra and $(A, [-,-])$ be the induced
Lie algebra from $(A, \ast)$. If $P: A\rightarrow A$ is an $\mathcal{O}$-operator
of $(A, \ast)$ associated to $(A^{\ast}, \fr_{A}^{\ast}, \fl_{A}^{\ast})$, then
$P$ is also an $\mathcal{O}$-operator of $(A, [-,-])$ associated to
$(A^{\ast}, -\ad^{\ast}_{A})$.
\end{pro}

In particular, by considering the skew-symmetric solution of the $\AYBE$ in $(A, \ast)$,
we obtain the following commutative diagram:
$$
\xymatrix@C=3cm@R=0.5cm{
\txt{$r$ \\ {\tiny a skew-symmetric solution} \\ {\tiny of the $\AYBE$ in $(A, \ast)$}}
\ar[d]_-{{\rm Pro.}~\ref{pro:ass-Lie-YBE}}\ar[r]^-{{\rm Pro.}~\ref{pro:o-ass}} &
\txt{$r^{\sharp}$\\ {\tiny an $\mathcal{O}$-operator of $(A, \ast)$} \\
{\tiny associated to $(A^{\ast}, \fr_{A}^{\ast}, \fl_{A}^{\ast})$}}
\ar[d]^-{{\rm Pro.}~\ref{pro:Ass-ind-ooper}} \\
\txt{$r$ \\ {\tiny a skew-symmetric solution} \\ {\tiny of the $\CYBE$ in
$(A, [-,-])$}} \ar[r]^-{{\rm Pro.}~\ref{pro:o-lie}}
& \txt{$r^{\sharp}$ \\
{\tiny an $\mathcal{O}$-operator of $(A, [-,-])$ } \\
{\tiny associated to $(A^{\ast}, -\ad^{\ast}_{A})$}}}
$$

Next, we consider the relationship between the solutions of $\DYBE$ in a dendriform algebra
and the solutions of $\PLYBE$ in the induced pre-Lie algebra.

\begin{pro}\label{pro:D-PL-YBE}
Let $(D, \prec, \succ)$ be a dendriform algebra and $(D, \diamond)$ be the induced
pre-Lie algebra from $(D, \prec, \succ)$. Suppose $r=\sum_{i}x_{i}\otimes y_{i}\in
D\otimes D$ is a symmetric solution of the $\DYBE$ in $(D, \prec, \succ)$.
Then $r$ is also a symmetric solution of the $\PLYBE$ in $(D, \diamond)$.
\end{pro}

\begin{proof}
First, note that
\begin{align*}
\mathbf{D}_{r}&=r_{12}\prec r_{13}+r_{12}\succ r_{13}-r_{13}\prec r_{23}-r_{23}\succ r_{12}\\
&=\sum_{i,j}\Big((x_{i}\prec x_{j})\otimes y_{i}\otimes y_{j}
+(x_{i}\succ x_{j})\otimes y_{i}\otimes y_{j}-x_{i}\otimes x_{j}\otimes(y_{i}\prec y_{j})
-x_{j}\otimes(x_{i}\succ y_{j})\otimes y_{i}\Big).
\end{align*}
Since $r$ is symmetric, i.e., $\sum_{i}x_{i}\otimes y_{i}=\sum_{i}y_{i}\otimes x_{i}$,
we have
\begin{align*}
\mathbf{PL}_{r}&=r_{13}\diamond r_{12}+r_{23}\diamond r_{12}+r_{21}\diamond r_{13}
+r_{23}\diamond r_{13}-r_{23}\diamond r_{21}-r_{12}\diamond r_{23}-r_{13}\diamond r_{21}
-r_{13}\diamond r_{23}\\
&=\sum_{i,j}\Big((x_{i}\succ x_{j})\otimes y_{j}\otimes y_{i}
-(x_{j}\prec x_{i})\otimes y_{j}\otimes y_{i}
+x_{j}\otimes(x_{i}\succ y_{j})\otimes y_{i}
-x_{j}\otimes(y_{j}\prec x_{i})\otimes y_{i}\\[-4mm]
&\qquad\quad +(y_{i}\succ x_{j})\otimes x_{i}\otimes y_{j}
-(x_{j}\prec y_{i})\otimes x_{i}\otimes y_{j}
+x_{j}\otimes x_{i}\otimes(y_{i}\succ y_{j})
-x_{j}\otimes x_{i}\otimes(y_{j}\prec y_{i})\\
&\qquad\quad -y_{j}\otimes(x_{i}\succ x_{j})\otimes y_{i}
+y_{j}\otimes(x_{j}\prec x_{i})\otimes y_{i}
-x_{i}\otimes(y_{i}\succ x_{j})\otimes y_{j}
+x_{i}\otimes(x_{j}\prec y_{i})\otimes y_{j}\\[-1mm]
&\qquad\quad -(x_{i}\succ y_{j})\otimes x_{j}\otimes y_{i}
+(y_{j}\prec x_{i})\otimes x_{j}\otimes y_{i}
-x_{i}\otimes x_{j}\otimes(y_{i}\succ y_{j})
+x_{i}\otimes x_{j}\otimes(y_{j}\prec y_{i})\Big)\\
&=(\id\otimes\tau)((\tau\otimes\id)((\id\otimes\tau)(\mathbf{D}_{r})))
-(\tau\otimes\id)((\id\otimes\tau)(\mathbf{D}_{r})).
\end{align*}
Thus, $r$ is also a symmetric solution of the $\PLYBE$ in $(D, \diamond)$ if
$r$ is symmetric solution of the $\DYBE$ in $(D, \prec, \succ)$.
\end{proof}

\begin{pro}\label{pro:qtD-qtpreLie}
Let $(D, \prec, \succ)$ be a dendriform algebra and $r\in D\otimes D$ be symmetric.
Suppose $(D, \prec$, $\succ, \theta_{\prec,r}, \theta_{\succ,r})$ is a triangular
dendriform $\md$-bialgebra associated with $r$ and $(D, \diamond, \vartheta)$ is
the induced pre-Lie bialgebra from $(D, \prec, \succ, \theta_{\prec,r},
\theta_{\succ,r})$. Then, we have $(D, \diamond, \vartheta)=(D, \diamond,
\vartheta_{r})$ as pre-Lie bialgebras, where $\vartheta_{r}$ is defined Eq.
\eqref{preli-cobo}.

Therefore, we get that $(D, \diamond, \vartheta)$ is a triangular pre-Lie bialgebra
if $(D, \prec, \succ, \theta_{\prec,r}, \theta_{\succ,r})$ is a triangular
dendriform $\md$-bialgebra.
\end{pro}

\begin{proof}
Let $r=\sum_{i}x_{i}\otimes y_{i}$. Since $r$ is symmetric, for any $d\in D$, we have
\begin{align*}
\vartheta_{r}(d)
&=\big(\fl_{A}(d)\otimes\id+\id\otimes(\fl_{A}-\fr_{A})(d)\big)(r)\\
&=\sum_{i}\Big((d\succ x_{i})\otimes y_{i}-(x_{i}\prec d)\otimes y_{i}
+x_{i}\otimes(d\succ y_{i})\\[-5mm]
&\qquad\quad-x_{i}\otimes(y_{i}\prec d)
-x_{i}\otimes(y_{i}\succ d)+x_{i}\otimes(d\prec y_{i})\Big)\\
&=\sum_{i}\Big(y_{i}\otimes(d\prec x_{i})+y_{i}\otimes(d\succ x_{i})
-(y_{i}\prec d)\otimes x_{i}\\[-5mm]
&\qquad\quad+(d\succ y_{i})\otimes x_{i}-y_{i}\otimes(x_{i}\prec d)
-y_{i}\otimes(x_{i}\succ d)\Big)\\
&=\theta_{\succ,r}(d)-\tau(\theta_{\prec,r}(d))\\
&=\vartheta(d).
\end{align*}
That is, $(D, \diamond, \vartheta)=(D, \diamond, \vartheta_{r})$ as pre-Lie bialgebras.
Thus, if $(D, \prec, \succ, \theta_{\prec,r}, \theta_{\succ,r})$ is a triangular
dendriform $\md$-bialgebra, i.e., $r$ is a symmetric solution of the $\DYBE$ in
$(D, \prec, \succ)$, by Propositions \ref{pro:spec-plbia} and \ref{pro:D-PL-YBE},
we get that $(D, \diamond, \vartheta)=(D, \diamond, \vartheta_{r})$ is a
triangular pre-Lie bialgebra.
\end{proof}

Let $(D, \prec, \succ)$ be a dendriform algebra. There is a close relationship between
solutions of symmetric solution of the $\DYBE$ in $(D, \prec, \succ)$ and the operator
$\mathcal{O}$-operator of $(D, \prec, \succ)$.

\begin{pro}[\cite{Bai}]\label{pro:o-D}
Let $(D, \prec, \succ)$ be a dendriform algebra and $r\in A\otimes A$ be symmetric.
Then $r$ is a solution of the $\DYBE$ in $(D, \prec, \succ)$ if and only if $r^{\sharp}$
is an $\mathcal{O}$-operator of $(D, \prec, \succ)$ associated to the coregular
bimodule $(D^{\ast}, \fr_{\succ}^{\ast}+\fr_{\prec}^{\ast}, -\fl_{\prec}^{\ast},
-\fr_{\succ}^{\ast}, \fl_{\prec}^{\ast}+\fl_{\succ}^{\ast})$.
\end{pro}

Moreover, by direct calculation, we have the following proposition.

\begin{pro}\label{pro:D-ind-ooper}
Let $(D, \prec, \succ)$ be a dendriform algebra and $(D, \diamond)$ be the induced
pre-Lie algebra from $(D, \prec, \succ)$. If $P: D^{\ast}\rightarrow D$ is an
$\mathcal{O}$-operator of $(D, \prec, \succ)$ associated to $(D^{\ast},
\fr_{\succ}^{\ast}+\fr_{\prec}^{\ast}, -\fl_{\prec}^{\ast}, -\fr_{\succ}^{\ast},
\fl_{\prec}^{\ast}+\fl_{\succ}^{\ast})$, then $P$ is also an $\mathcal{O}$-operator
of $(D, \diamond)$ associated to $(D^{\ast}, \fr^{\ast}_{D}-\fl^{\ast}_{D}, \fr^{\ast}_{D})$.
\end{pro}

Thus, for a symmetric solution of the $\DYBE$ in a dendriform algebra $(D, \prec, \succ)$,
we have the following commutative diagram:
$$
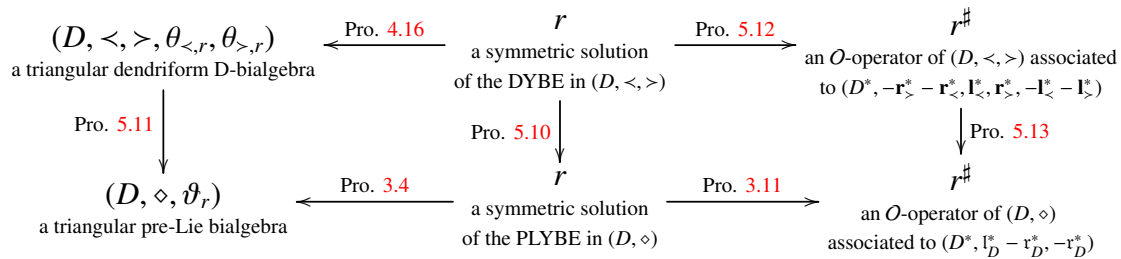
\xymatrix@C=1.6cm@R=0.6cm{
\txt{$(D, \prec, \succ, \theta_{\prec,r}, \theta_{\succ,r})$ \\ {\tiny a triangular
dendriform $\md$-bialgebra}}
\ar[d]_{{\rm Pro.}~\ref{pro:qtD-qtpreLie}} &
\txt{$r$ \\ {\tiny a symmetric solution} \\ {\tiny of the $\DYBE$ in $(D, \prec, \succ)$}}
\ar[d]_-{{\rm Pro.}~\ref{pro:D-PL-YBE}}\ar[r]^-{{\rm Pro.}~\ref{pro:o-D}}
\ar[l]_-{{\rm Pro.}~\ref{pro:tridend}}    &
\txt{$r^{\sharp}$\\ {\tiny an $\mathcal{O}$-operator of $(D, \prec, \succ)$ associated} \\
{\tiny to $(D^{\ast}, -\fr_{\succ}^{\ast}-\fr_{\prec}^{\ast},
\fl_{\prec}^{\ast}, \fr_{\succ}^{\ast}, -\fl_{\prec}^{\ast}-\fl_{\succ}^{\ast})$}}
\ar[d]^-{{\rm Pro.}~\ref{pro:D-ind-ooper}} \\
\txt{$(D, \diamond, \vartheta_{r})$ \\
{\tiny a triangular pre-Lie bialgebra}}   &
\txt{$r$ \\ {\tiny a symmetric solution} \\ {\tiny of the $\PLYBE$ in
$(D, \diamond)$}} \ar[r]^-{{\rm Pro.}~\ref{pro:o-PL}}
\ar[l]_-{{\rm Pro.}~\ref{pro:spec-plbia}}   &
\txt{$r^{\sharp}$ \\
{\tiny an $\mathcal{O}$-operator of $(D, \diamond)$ } \\
{\tiny associated to $(D^{\ast}, \kl^{\ast}_{D}-\kr^{\ast}_{D}, -\kr^{\ast}_{D})$}}}
$$
Therefore, we obtained the three-dimensional commutative diagram in the introduction.


\begin{ex}\label{ex:3-comdiagm}
Let $(D, \prec, \succ)$ be the 2-dimensional dendriform algebra given in Example
\ref{ex:dendind-ASI} and $r=e_{1}\otimes e_{1}$. Then $r$ is a symmetric solution
of the $\DYBE$ in $(D, \prec, \succ)$. By Example \ref{ex:dendind-ASI}, we get
the back side of the three-dimensional diagram is commutative. By Example
\ref{ex:prelie-indlie}, we get the front side of the three-dimensional diagram
is commutative. By Example \ref{ex:com-diag}, we get the top side of the
three-dimensional diagram is commutative. Moreover, it can be directly verified
from the calculation results of Examples \ref{ex:prelie-indlie} and \ref{ex:dendind-ASI}
that, the left side, the right side and the bottom side of the three-dimensional diagram
are also commutative.
\end{ex}

\bigskip
\noindent
{\bf Acknowledgements. } This work was financially supported by National
Natural Science Foundation of China (No.11771122).

\smallskip
\noindent
{\bf Declaration of interests.} The authors have no conflicts of interest to disclose.

\smallskip
\noindent
{\bf Data availability.} Data sharing is not applicable to this article as no new data were
created or analyzed in this study.


\begin{thebibliography}{abc}

\bibitem{Agu} M. Aguiar,
    Per-Poisson algebras,
    Lett. Math. Phys. {\bf 54} (2000), 263--277.

\bibitem{Bai1} C. Bai,
    Left-symmetric bialgebras and an analogue of the classical Yang-Baxter equation,
    Commun. Contemp. Math. {\bf 10} (2008), 221--260.

\bibitem{Bai} C. Bai,
    Double constructions of Frobenius algebras, Connes cocycles and their duality,
    J. Noncommut. Geom. \textbf{4} (2010), 475--530.

\bibitem{BLST} C. Bai, G. Liu, Y. Sheng, R. Tang,
    Quasi-triangular, factorizable Leibniz bialgebras and relative Rota-Baxter operators,
    Forum. Math. {\bf 37} (2025), 1083--1101.

\bibitem{BGN} C. Bai, L. Guo, X. Ni
    Nonabelian generalized Lax pairs, the classical Yang-Baxter equation and post-Lie algebras,
    Commun. Math. Phys. {\bf 297} (2010), 553--596.

\bibitem{BN} A. Balinsky, S. Novikov,
    Poisson brackets of hydrodynamic type, Frobenius algebras and Lie algebras,
    Sov. Math. Dokl. {\bf 32} (1985), 228--231.

\bibitem{Bax} R.J. Baxter,
    Partition function of the eight-vertex lattice model,
    Ann. Phys. {\bf 70} (1972), 193--228.

\bibitem{BD}  A.A. Belavin, V. G. Drinfeld,
    Solutions of the classical Yang-Baxter equation for simple Lie algebras,
    Functional Anal. Appl. {\bf 16} (1982), 1--29.

\bibitem{BCHM} K. Benali, T. Chtioui, A. Hajjaji, S. Mabrouk,
    Bialgebras, the Yang-Baxter equation and Manin triples for mock-Lie algebras,
    Acta et Commentationes Universitatis Tartuensis de Mathematica {\bf 27} (2023), 211--233.

\bibitem{CP} V. Chari, A. Pressley,
    {\it A Guide to Quantum Groups},
    Cambridge University Press, Cambridge 1995.

\bibitem{CFL} K. Cieliebak, K. Fukaya, J. Latschev,
    Homological algebra related to surfaces with boundary,
    Quantum Topol. {\bf 11} (2020), 691--837.

\bibitem{CH} Z. Cui, B. Hou,
    Quasi-triangular Novikov bialgebras and induced quasi-triangular Lie bialgebras,
    arXiv: 2505.19579
	
\bibitem{Dri} V.G. Drinfeld,
    Hamiltonian structures of Lie groups, Lie bialgebras and the geometric meaning of
    the classical Yang-Baxter equations,
    Soviet Math. Dokl. {\bf 27} (1983), 68--71.

\bibitem{Dri1} V. Drinfeld,
   Quantum groups,
   {\it In: Proc. Int. Congr. Math.}, volume 1, pages 798--820,
   Amer. Math. Soc., Providence RI, 1987.

\bibitem{ES} P. Etingof, O. Schiffmann,
    {\it Lectures on Quantum Groups},
    International Press Inc., Cambridge, 1998.

\bibitem{GK} V. Ginzburg, M. Kapranov,
    Koszul duality for operads,
    Duke Math. J. {\bf 6} (1994), 203--272.

\bibitem{Gol} W. Goldman,
    Invariant functions on Lie groups and Hamiltonian flows of surface group representations,
    Invent. Math. {\bf 85} (1986), 263--302.

\bibitem{HBG} Y. Hong, C. Bai, L. Guo,
    Infinite-dimensional Lie bialgebras via affinization of Novikov bialgebras and
    Koszul duality,
    Comm. Math. Phys. {\bf 401} (2023), 2011--2049.
	
\bibitem{HBG1} Y. Hong, C. Bai, L. Guo,
    Deformation families of Novikov bialgebras via differential antisymmetric
    infinitesimal bialgebras,
    arXiv: 2402.16155.

\bibitem{HBG2} Y. Hong, C. Bai, L. Guo,
    A bialgebrta theory of Gel'fand-Dorfman algebras with applications to Lie conformal
    bialgebras,
    arXiv: 2401.13608.

\bibitem{Hou} B. Hou,
    Extending structures for perm algebras and perm bialgebras,
    J. Algebra {\bf 649} (2024), 392--432.

\bibitem{JR} S.A. Joni, G.C. Rota,
    Coalgebras and bialgebras in combinatorics.
    Stud. Appl. Math. {\bf 61} (1979), 93--139.

\bibitem{Kup} B.A. Kupershmidt,
    What a classical $r$-matrix really is,
    J. Nonlinear Math. Phy. {\bf 6} (1999), 448--488.

\bibitem{LS} H. Lang, Y. Sheng,
    Factorizable Lie bialgebras, quadratic Rota-Baxter Lie algebras and Rota-Baxter Lie
    bialgebras,
    Commun. Math. Phys. {\bf 397} (2023), 763--791.

\bibitem{LLB} Y. Lin, X. Liu, C. Bai,
    Differential antisymmetric infinitesimal bialgebras, coherent derivations and
    Poisson bialgebras,
    Symmetry Integrability Geom. Methods Appl. {\bf 19} (2023), 018, 47pp.

\bibitem{LZB} Y. Lin, P. Zhou, C. Bai,
    Infinite-dimensional Lie bialgebras via affinization of perm bialgebras and pre-Lie
    bialgebras,
    J. Algebra {\bf 663} (2025), 210--258.

\bibitem{LFCG} J.L. Loday, A. Frabetti, F. Chapoton, F. Goichot,
    {\it Dialgebras and related operads},
    Lecture Notes in Mathematics, 1763, Springer, Berlin, 2001.

\bibitem{RB} I. S. Rakhimov, U. Bekbaev,
    Two-dimensional dendriform algebra structures over any base field,
    Malays. J. Math. Sci. {\bf 19} (2025), 891--920.

\bibitem{RS} N. Reshetikhin, M.A. Semenov-Tian-Shansky,
    Quantum $R$-matrices and factorization problems,
    J. Geom. Phys. {\bf 5} (1988), 533--550.

\bibitem{Sem} M.A. Semenov-Tian-Shansky,
    What is a classical $r$-matrix ?
    Funct. Anal. Appl. {\bf 17} (1983), 259--272.

\bibitem{SW} Y. Sheng, Y. Wang,
    Quasi-triangular and factorizable antisymmetric infinitesimal bialgebras,
    J. Algebra {\bf 628} (2023), 415--433.

\bibitem{Sto} A. Stolin,
    Rational solutions of the classical Yang-Baxter equation and quasi-Frobenius Lie algebras,
    J. Pure Appl. Algebra {\bf 137} (1999), 285--293.

\bibitem{Val} B. Vallette,
    A Koszul duality for props,
    Trans. Amer. Math. Soc. {\bf 359} (2007), 4865--4943.

\bibitem{WBLS} Y. Wang, C. Bai, J. Liu, Y, Sheng,
    Quasi-triangular pre-Lie bialgebras, factorizable pre-Lie bialgebras and Rota-Baxter
    pre-Lie algebras,
    J. Geom. Phys. {\bf 199} (2024), 105146, 22pp.

\bibitem{Yang} C.N. Yang,
    Some exact results for the many-body problem in one dimension with repulsive
    deltafunction interaction,
    Phys. Rev. Lett. {\bf 19} (1967), 1312--1314.

\bibitem{Zhe} V.N. Zhelyabin,
    Jordan bialgebras and their relation to Lie bialgebras,
    Algebra Logic, {\bf 36} (1997), 1--15.










\end{thebibliography}
 \end{document}